\documentclass[11pt]{article}
\usepackage[left=1in,top=1in,right=1in,bottom=0.8in, nohead,nofoot]{geometry}
\usepackage{graphics}

\usepackage{fullpage}
\usepackage[section]{placeins}
\usepackage{times}

\usepackage{graphicx}
\usepackage{amsmath}
\usepackage{mathrsfs}
\usepackage{amssymb}
\usepackage{amsthm}
\usepackage{url}
\usepackage{caption}

\usepackage{pstricks,pst-plot}

\usepackage{indentfirst}
\usepackage{clrscode}

\usepackage{mdwlist}

\usepackage[compact]{titlesec}

\newtheorem{theorem}{Theorem}[section]

\newtheorem{lemma}[theorem]{Lemma}
\newtheorem{proposition}[theorem]{Proposition}
\newtheorem{corollary}[theorem]{Corollary}
\newtheorem{definition}[theorem]{Definition}

\newcommand{\xiaorui}[1]{{{#1}}}

\newcommand{\E}{\mathrm{E}}

\newcommand{\cF}{{\mathcal F}}

\newcommand{\pr}{{\mathrm{Pr}}}

\newcommand{\cV}{{\mathcal V}}

\newcommand{\isd}{{\mathrm{Isd}}}
\newcommand{\child}{{\mathrm{child}}}
\newcommand{\dchild}{\mathrm{dchild}}
\newcommand{\T}{{\mathrm{Tr}}}

\newcommand{\cB}{{\mathcal B}}

\newcommand{\cG}{{\mathcal G}}

\newcommand{\cP}{{\mathcal P}}
\newcommand{\hell}{{\hat \ell}}

\newcommand{\bZ}{{\mathbf Z}}
\newcommand{\bC}{{\mathfrak C}}

\newcommand{\rt}{\mathbf r}

\newcommand{\insur}{\dot \partial} 
\newcommand{\exsur}{\partial} 
\newcommand{\mfa}{\mathfrak A}

\newcommand{\liusd}[1]{{{#1}}}
\newcommand{\zliu}[1]{{{#1}}}
\newcommand{\xiaoruin}[1]{{{#1}}}
\newcommand{\yajunfinal}[1]{{{#1}}}
\newcommand{\xiaoruiview}[1]{{{#1}}}
\newcommand{\yajun}[1]{{{#1}}}
\newcommand{\lam}[1]{{{#1}}}

\begin{document}

\title{Information Dissemination via Random Walks in $d$-Dimensional Space}


\author{Henry Lam
\thanks{Department of Mathematics and Statistics, Boston University. Email: {\tt khlam@math.bu.edu}.}
\quad
Zhenming Liu \thanks{Harvard School of Engineering and Applied Sciences. Supported by NSF grant CCF-0915922. Email: {\tt zliu@eecs.harvard.edu}.}
\quad
Michael Mitzenmacher \thanks{Harvard School of Engineering and Applied Sciences. Supported in part by NSF grants CCF-0915922 and IIS-0964473. Email: {\tt michaelm@eecs.harvard.edu}.}
\quad
Xiaorui Sun \thanks{Columbia University. Email: {\tt  xiaoruisun@cs.columbia.edu}.}
\quad
Yajun Wang \thanks{Microsoft Research Asia. Email: {\tt yajunw@microsoft.com}.}
}

\date{}
\setcounter{page}{0}

\maketitle

\begin{abstract}
  We study a natural information dissemination problem for multiple
  mobile agents in a bounded Euclidean space.  Agents are placed
  uniformly at random in the $d$-dimensional space $\{-n, ..., n\}^d$
  at time zero, and one of the agents holds a piece of information to
  be disseminated. All the agents then perform independent random
  walks over the space, and the information is transmitted from one
  agent to another if the two agents are sufficiently close.  We wish
  to bound the total time before all agents receive the information
  (with high probability).  Our work extends Pettarin et~al's
  work~\cite{PPPU11}, which solved the problem for $d \leq 2$.  We
  present tight bounds up to polylogarithmic factors for the case $d =
  3$.  (While our results extend to higher dimensions, for space and
  readability considerations we provide only the case $d=3$ here.)
  Our results show the behavior when $d \geq 3$ is qualitatively
  different from the case $d \leq 2$.  In particular, as the ratio
  between the volume of the space and the number of agents varies, we
  show an interesting phase transition for three dimensions that does
  not occur in one or two dimensions.
\end{abstract}

\newpage

\section{Introduction}
We study the following information diffusion problem: let $\mathrm
a_1, \mathrm a_2, ..., \mathrm a_m$ be $m$ agents initially starting
at locations chosen uniformly at random in $\mathcal V^d =
\{-n,,-(n-1),\ldots,n\}^d$ and performing independent random walks
over this space.  One of the agents initially has a message, and the
message is transmitted from one agent to another when they are
sufficiently close.  We are interested in the time needed to flood the
message, that is, the time when all agents obtain the message.  In
other settings, this problem has been described as a virus diffusion
problem, where the message is replaced by a virus that spreads
according to proximity. We use {\em information diffusion} and {\em
  virus spreading} interchangeably, depending on which is more useful
in context.  This is a natural model that has been extensively
studied. For example, Alves et~al. and Kesten et~al. coined the name
``frog model'' for this problem in the virus setting, and studied the
shape formed by the infected contour in the limiting case
\cite{AMPR01, AMP02, KS05}.  In the flooding time setting, early works
used a heuristic approximation based on simplifying assumptions to
characterize the dynamics of the spread of the
message~\cite{AYO93,KBTR02, WKB08}.  More recent works provide fully
rigorous treatments under this or similar random walk models
\cite{CMS10, CPS09, PPPU11, SS10, PSSS11}.

The most relevant recent works are those of Pettarin
et~al.~\cite{PPPU11} and Peres et~al.~\cite{PSSS11, SS10}.
The work of Pettarin et~al. examines the same model as ours, but their
analysis is only for one- and two-dimensional grids.  
The work of Sinclair and Stauffer~\cite{SS10} considers a similar model they call
mobile geometric graphs, and their work extends to higher dimensions.
However, their focus and model both have strong differences from ours.
For example, they assume a Poisson point process of constant
intensity, leading to a number of agents linear in the size of the
space. In contrast, our results allow a sublinear number of agents, a scenario not directly relevant to their model. Also, they focus on structural aspects on the mobile graphs, such as percolation, while we are primarily interested in the diffusion time. There
are additional smaller differences, but the main point is that for
our problem we require and introduce new techniques and analysis.

Our paper presents matching lower \xiaoruin{bounds} and upper bounds (up to
polylogarithmic factors) for the flooding problems in $d$-dimensional
space for an arbitrary constant $d$.  For ease of exposition, in this
paper we focus on the specific case where $d = 3$, which provides the
main ideas.  Two- and three- dimensional random walks have quite
different behaviors -- specifically, two-dimensional random walks are
recurrent while three-dimensional random walks are transient -- so it
is not surprising that previous results for two dimensions fail to
generalize immediately to three-dimensional space.  Our technical
contributions include new techniques and tools for tackling the
flooding problem by building sharper approximations on the effect of agent interactions. The techniques
developed in this paper are also robust enough that our results can be
extended to variations of the model, such as allowing probabilistic
infection rules, replacing discrete time random walks by continuous
time Brownian motions, or allowing the agents to make jumps
\cite{CPS09}. These extensions will be reported in future work.

Although the information diffusion problem in three or more dimensions
appears less practically relevant than the two-dimensional case, we
expect the model will still prove valuable. For instance, particles in
a high dimensional space may provide a latent-space representation of
the agents in a dynamic social network \cite{HRH02, SM05}, so
 understanding information
diffusion process may be helpful for designing appropriate latent space models in the future.
Also, the problem is mathematically interesting in its own right.

\subsection{Our models and results}
We follow the model developed in \cite{PPPU11}. Let $\mathcal V^d =
\{-n, -(n - 1), ..., 0, ..., (n - 1), n\}^d$ be a $d$-dimensional grid.  Let $A =
\{\mathrm a_1, \mathrm a_2, ..., \mathrm a_m\}$ be a set of moving
agents on $\mathcal V^d$. At $t = 0$, the agents
\yajun{spread} over the
space according to some distribution $\mathcal D$. Throughout this
paper, we focus on the case where $\mathcal D$ is uniform.
Agents move in discrete time steps. Every agent performs
a symmetric random walk defined in the natural way.  Specifically, at each time
step an agent not at a boundary moves to one of its $2d$ neighbors, each with
probability
\yajunfinal{$1/(2d)$}. If an agent is at a boundary, so there is no edge
in one or more directions, we treat each missing edge as a self-loop.
Let $\Xi_1(t), ..., \Xi_m(t) \in \{0, 1\}$ each be a random
variable, where $\Xi_i(t)$ represents whether the agent $\mathrm a_i$ is
infected at time step $t$. We assume $\Xi_1(0) = 1$ and $\Xi_i(0) = 0$ for all
$i \neq 1$.  The value $\Xi_i(t)$ will change from $0$ to $1$ if at time
$t$ it is within distance 1
\yajunfinal{to}
another infected agent $\mathrm a_j$.
(We use distance 1 instead of distance 0 to avoid parity issues.)  Once
a value $\Xi_j(t)$ becomes 1, it stays 1. 

\begin{definition}\label{def:diffusion} \emph{(Information diffusion problem)}. Let $A_1, A_2, \ldots, A_m \in \mathcal V^d$
be the initial positions of the agents $\mathrm a_1,
  \ldots, \mathrm a_m$ and let $S^1_t(A_1), S^2_t(A_2), \ldots,
  S^m_t(A_m)$ be $m$ independent random walks starting at
   $A_1, \ldots, A_m$ respectively, so that $S^i_t(P)$ is the position of
  agent $\mathrm a_i$ at time $t$ given \yajunfinal{that} at $t=0$ its position was
  $P \in \mathcal V^d$. The \emph{infectious state} of each agent at time step $t$ is a
  binary random variable $\Xi_i(t)$ such that
\begin{itemize*}
\item $\Xi_1(0) = 1$, $\Xi_i(0) = 0$ for all other $i$, and
\item for all $t > 0$, $\Xi_i(t) = 1$ if and only if
$$\left(\Xi_i(t - 1) = 1\right) \quad \mbox{ or } \quad \left( \exists j:
  \Xi_j(t - 1) = 1 \wedge \left\|S^i_t(A_i) - S^j_t(A_j)\right\|_1 \leq 1
\right).$$
\end{itemize*}
We define the \emph{finishing time} of the diffusion process, or the
\emph{diffusion time}, as $T=\inf\{t\geq0:|\{\Xi_i(t) = 1\}| =
\yajunfinal{m}
\}$.

\end{definition}

The following results for the diffusion time for $1$ and $2$ dimensional spaces are proved in \cite{PPPU11}.
\begin{theorem}
\label{thm:oldtheorem}
Consider the information diffusion problem for $d = 1, 2$ dimensions, and assume the agents are initially uniformly distributed over $\mathcal{V}^d$. Then, with high probability,
\begin{equation}\label{eqn:d12result}
 T = \tilde \Theta(n^2 \cdot m^{-1/d}).
 \end{equation}
\end{theorem}

\liusd{It is natural to ask whether Equation~\ref{eqn:d12result}
also holds for $d \geq 3$. Our results show this is not the case.}

\begin{theorem}
\label{thm:ourtheorem}
\emph{(Diffusion time for $d \geq 3$)} Consider the information
diffusion problem for $d \geq 3$ with initially uniformly distributed agents over $\mathcal{V}^d$.
Then there exists a constant $c$ such that
\begin{equation}\label{eqn:d3result}
\begin{array}{lll}
\mbox{if }&  cn^{d - 2}\log^2 n < m < n^d: & T = \tilde \Theta(n^{d/2 + 1} \cdot
m^{-1/2}) \mbox{ with high probability;} \\
\mbox{if }& m <  cn^{d - 2}\log^{-2}n: & T \leq \tilde \Theta(n^d / m) \mbox{ with
  high prob. and } T \geq \tilde \Theta(n^d / m) \mbox{ almost surely.}
\end{array}
\end{equation}
\end{theorem}


Notice that Theorems~\ref{thm:ourtheorem} and~\ref{thm:oldtheorem} yield
the same result for $d = 2$, as well as when $d = 1$ and $m =
\Theta(n)$.
Here when we say with high probability, we mean the statement holds with probability $1 -
n^{-\gamma}$ for any constant $\gamma$ and suitably large $n$.
When we say almost surely, we mean with probability $1
- o(1)$.  When $m \geq n^d$, the result is implicit in
\cite{KS05} and the diffusion time in this case is $\tilde
\Theta(n)$. {Finally, there are some technical challenges
regarding the case $cn^{d -
  2}\log^{-2}n \leq m \leq cn^{d - 2}\log^2 n$
that we expect to address in a later version of this work.}
%

An interesting point of our result is that when the number of agents $m$
is greater than $n^{d-2}$, the finishing time is less than the mixing
time of each individual random walk, and therefore the analysis requires
techniques that do not directly utilize the mixing time. The rest of this paper focuses on deriving both the lower and upper bounds for this interesting case; the case where $m<cn^{d-2}\log^{-2}n$, which harnesses similar ideas and a mixing time argument, is only briefly described at the end. Finally, as previously mentioned,
for space reasons we provide only the analysis for the three
dimensional case, and note that the results can be generalized to higher dimensions.

\lam{Theorem~\ref{thm:ourtheorem} can also be expressed in the terms of the density of agents. Let $\lambda=m/n^d$ be the density. We can express the diffusion time as $T=\tilde{\Theta}(n/\sqrt{\lambda})$ w.h.p. for $cn^{-2}\log^2n<\lambda<1$, whereas for $\lambda<cn^{-2}\log^{-2}n$ we have $T\leq\tilde{\Theta}(1/\lambda)$ w.h.p. and $T\geq\tilde{\Theta}(1/\lambda)$ almost surely.}

We remark that all theorems/propositions/lemmas in this paper are assumed to hold for sufficiently large $n$, but for conciseness we may not restate this condition in every instance.

\section{Preliminary results for random walks}
In this section we lay out some preliminary results on random walks that will be useful in the subsequent sections. These results focus on probabilistic estimates for the meeting time/position of multiple random walks. Along the way, we will also illustrate the limitations of some of these estimates, hence leading to the need of more sophisticated techniques in our subsequent analysis. For conciseness, all proofs in this section are left to Appendix~\ref{sec:mrw}.

Let $\mathbf Z$ be the set of integers, and $\mathbf Z^3$
be the set of integral lattice points in $\mathbf R^3$.
For two points $A, B \in \mathbf Z^3$, we write
$A - B$ as the $3$-dimensional vector pointing from $B$ to
$A$. For a vector $\vec x \in \mathbf R^3$, denote the $i$th
coordinate of $\vec x$ as $x_i$. Define the $L_p$ norm of a
vector as $\|\vec x\|_p = \left(\sum_{i \leq 3}|x_i|^p\right)^{1/p}$, and also the infinite-norm
in the standard manner $\|\vec x\|_{\infty} = \max_{i \leq 3}|x_i|$. We moderately overload $x$ in this paper, i.e. $x$ is a scalar
and $\vec x$ is a vector.

Let $S^1$ and $S^2$ be two random walks in either $\mathcal V \yajun{^3}$ (bounded walks)
or $\mathbf Z^3$ (unbounded walks). We say two walks $S^1$ and $S^2$ \emph{meet}
at time $t$ if their $L_1$-distance is within $1$ at that time
and two walks $S^1$ and $S^2$ \emph{collide} at $t$ if they
are exactly at the same position at time $t$.

\begin{definition}[Passage probability] Let $S$ be a
random walk in $\mathbf Z^3$ starting at the origin $O$.
Let $B$ be a point with $B - O = \vec x$ (which is a three dimensional vector). Define the probability that $S$ is at $B$ at
time $t$ as $p(t, \vec x)$. Define the probability that $S$ visits
$B$ \emph{within} time $t$ as $q(t, \vec x)$.
\end{definition}

We want to characterize the chance that two or more random walks
in either $\mathcal V \yajun{^3}$ or $\mathbf Z^3$
meet. More specifically, consider the following question. Let $A_1$, ...,
$A_j$, and $B$ be $j + 1$ points over the 3-dimensional space
$\mathbf Z^3$ such that for all $i \in [j]$, the $L_1$ distance between $A_i$ and $B$ is $\|A_i - B\|_1 \geq x$. Let $S^1(A_1), ..., S^j(A_j)$, $S^{j + 1}(B)$ be independent random walks that start with these points respectively.
 Our goal is to understand
the probability that all the walks $S^1$, ..., $S^j$ will meet or collide with the random walk
$S^{j + 1}$ within $x^2$ time steps. We note that if the agents starting at $B$ was stationary instead of
following its own random walk then the analysis of
the situation would be straightforward. In particular, the probability that all
the walks would intersect $B$ is $\tilde \Theta(1/x^k)$. This follows from standard results, including Theorem~\ref{thm:lowervisit}
and Lemma~\ref{lem:tightvisit} provided in the appendices. We need to consider a more challenging situation when the agent starting at $B$ is also
moving.

To begin, we shall consider the case where $
 \yajun{j}
= 1$, so that we have just two moving agents.

\begin{definition} Let $A$ and $B$ be two points over $\mathbf Z^3$ such that $A - B = \vec x$. where $\|\vec x\|_1$ is an even number. Let $S^1$ and $S^2$ be two independent unbounded random walks that start
at $A$ and $B$ respectively. Define $Q(t, \vec x)$ as the probability that $S^1$ and $S^2$ collide before time $t$.
\end{definition}

We can use a simple coupling argument to relate $Q(t, \vec x)$ with $q(t, \vec x)$. The result is described as follows.

\begin{lemma}\label{lem:couple} Let $A$ and $B$ be two points over
   $\mathbf Z^3$ such that $A- B  =\vec x$, where $\|\vec x\|_1$ is an even number.
 Consider $Q(t, \vec x)$ and $q(t, \vec x)$ defined above.
 We have $Q(t, \vec x) = q(2t, \vec x)$.
 Furthermore, for $t \geq \|\vec x\|^2_2$, $Q(t, \vec x) = \Theta(1/\|\vec x\|_2)$.
\end{lemma}


Next, let us move to the case of $j$ random walks in $\mathbf Z^3$, in which $j > 1$.

\begin{lemma}\label{lem:rwcatchall} Let $A_1$, $A_2$, \ldots, $A_j$, and $B$ be points
 in $\mathbf Z^3$ such that $\|A_i-B\|_1$ are even and
$\|A_i-B\|_1 \geq
x$ for all $i \leq j$. Let $S^1(A_1), \ldots, S^j(A_j)$, $S^{j + 1}(B)$ be $j + 1$ independent random walks that
start at $A_1, \ldots, A_j, B$ respectively. Let $t = x^2$.
Then the probability that all the walks $S^1, \ldots, S^j$ collide with
$S^{t + 1}$ within time  $t$
is at most $\left(\frac{\zeta j}{ x}\right)^j$, where $\zeta$ is a
sufficiently large constant.
\end{lemma}

Lemma~\ref{lem:rwcatchall} also helps us to analyze the scenario in which agents need to meet rather than to collide. This is summarized by the following corollary:

\begin{corollary}\label{cor:rwcatchall} Let $A_1$, $A_2$, ..., $A_j$, and $B$ be points
 in $\mathbf Z^3$ such that
$\|A_i-B\|_1 \geq
x$ for all $i \leq j$. Let $S^1(A_1), ..., S^j(A_j)$, $S^{j + 1}(B)$ be $j + 1$ independent random walks that
start at $A_1, ..., A_j, B$ respectively. Let $t = x^2$.
Then the probability that all the walks $S^1, ..., S^j$ meet with
$S^{j + 1}$ within time $t$
is at most $\left(\frac{\zeta' j}{ x}\right)^j$, where $\zeta'$ is a sufficiently large constant.
\end{corollary}

We note that both Lemma~\ref{lem:rwcatchall} and Corollary~\ref{cor:rwcatchall} are useful only when $x$ is large enough. This forms a barrier for analysis of close agents in our model. But as we will see, we can get around this issue by looking at a coupled diffusion process that possesses a different diffusion rule specifically designed for handling close agents.

Another important issue is the analysis on walks that are close to the boundary. For this, we show that the random walks
will not behave significantly different (in terms of
the desired bounds) when boundaries are added. We notice that similar results
are presented in~\cite{PPPU11}, but their results do not immediately
translate to the building blocks we need here. The following is the major building block
we need for our analysis:

\begin{lemma}
\label{lem:twowalkboundary}
Let $A$ and $B$
be two points in $\mathcal V^3$ such that $A-B = \vec x$ and
the distance between
$A$ and any boundary is at least $40\|  \yajunfinal{\vec x}\|_{1}$. Consider two random walks $ S^1(A)$
and $ S^2(B)$ that start at $A$ and $B$ respectively. Let $\tilde e_t$
be the event that $S^1(A)$ and $S^2(B)$ will meet before time
$t  \yajunfinal{=\|\vec x\|_1^2}$ and
before either of them visits a boundary. Then $\Pr[\tilde
e_{\|  \yajunfinal{\vec x}\|_1^2}] = \Omega(1/\|  \yajunfinal{\vec x}\|_1)$.
\end{lemma}




\newcommand{\ma}{{\mathrm a}}
\newcommand{\bbB}{{\mathbb B}}

\section{Lower bound}\label{sec:lb3}
Let us first state our lower bound result more precisely as follows.

\begin{theorem}\label{thm:lb3}
Let $\mathrm a_1, ..., \mathrm a_m$ be placed
uniformly at random on $\mathcal V^3$ such that $ 1600 n
\log^2n \leq m \leq n^3$. Let $\ell_2 = \sqrt{n^3 / m}$. For
sufficiently large $n$, the diffusion time $T$ satisfies the following
  \yajunfinal{inequality}
$$\Pr[T \leq  \frac 1 {81}\ell_2 n \log^{-29}n] \leq \exp\left(-\log
  n\log \log n\right).$$
\end{theorem}

We use a \emph{local analysis} to prove our lower bound. The key idea
is that under uniform distribution of agents, the extent any
particular infected agent can spread the virus within a small time
increment is confined to a small neighborhood with high
probability. By gluing together these local estimates, we can
approximate the total diffusion time.




To explain our local analysis,
assume we start with an arbitrary infected agent, say
$\mathrm{a}_1$. Let us also assume, for simplicity,
that all the other uniformly distributed agents are
uninfected. Consider the scenario within a small time increment, say
$\Delta t$. During this time  \yajun{increment} the agent $\mathrm{a}_1$ infects
whoever it meets in the small neighborhood that contains its
extent of movement. The newly infected agents then continue
to move and infect others. The size of the
final region that contains all the infected agents
at $\Delta t$ then depends on the rate of transmission
and the extent of movement of all of the infected agents. In particular, if
$\Delta t$ is small enough, the expected number of transmissions
performed by $\mathrm{a}_1$ is less than one; even if it infects another agent,
the number of infections it causes within the
\emph{same} $\Delta t$ is also less than one, and so on.
The net effect is an eventual dying-out of
this ``branching process" (which we later model by what we call a diffusion tree), which localizes the
 \yajun{positions}
of all infected
agents at time $\Delta t$ to a small neighborhood around the initial
position of $\mathrm{a}_1$.

\liusd{}
As it may not be clear as we go through our proofs, we briefly review the main methodologies in obtaining lower bound results in related work, and point out their relation to our analysis and difficulties in directly applying them to higher dimensions. (Some readers may wish to skip these next paragraphs all together;  for others, who would like a more thorough discussion that unavoidably requires more technical details, we devote Appendix~\ref{sec:existing} to more details.) Two potential existing methods arise in \cite{AMP02, KS05} and \cite{PPPU11}. The former analyzes the growth rate of the size of the \emph{total} infected region; an upper bound on this growth rate translates to a lower bound for the diffusion time. The latter work, focusing on $d=1,2$, uses an ``island diffusion rule", which essentially speeds up infection by allowing infections to occur immediately on connected components in an underlying graph where edges are based on the distance between agents. This approach avoids handling the issue of the meeting time of random walks when they are very close, a regime where asymptotic results such as Lemma~\ref{lem:couple} and \ref{lem:rwcatchall} may not apply, while still providing a way to bound the diffusion time by arguing about the low probability of interaction among different ``islands".

The results in \cite{AMP02, KS05} are not directly applicable in our setting because the growth rate they obtain is linear in time, as a result of their assumption of constant agent density in an infinite space, in contrast to our use of a size parameter $n$ that scales with the agent density. It is fairly simple to see that blindly applying a linear growth rate to our setting of $o(1)$ density is too crude.
On the other hand, analyzing how agent density affects the growth rate is a potentially feasible approach but certainly not straightforward.


Our approach more closely follows \cite{PPPU11}.  The main limitation of \cite{PPPU11}, when applied to higher dimension, is how to control the interaction among islands.  If islands interact too often, because they are too close together, the argument, which is based on a low probability of interaction, breaks down.  However, if one parametrizes islands to prevent such interaction, then the bound that can be obtained are too weak.  In Appendix~\ref{subsec:existlower} we provide further details arguing that for $d>2$ this constraint ultimately limits the analysis for the case of $o(1)$ density. We attempt to remedy the problem by using islands as an intermediate step to obtain local estimates of the influence of each initially infected agent over small periods of time. This analysis involves looking at a branching process representing the spread of the infection, significantly extending the approach of \cite{PPPU11}.


%

\subsection{Local diffusion problem}
This subsection focuses on the local analysis as discussed above. In Section~\ref{subsec:globallow}, we will proceed to discuss how to utilize this analysis to get the lower bound in Theorem~\ref{thm:lb3}. As discussed in the last section, the two main difficulties in our analysis are: 1) our probabilistic estimates for the meeting time/position of multiple random walks are only useful  \zliu{asymptotically}; 2) walks near the boundary introduce further analytical complication. To begin with, the following definition serves to handle the second issue:

\begin{definition}[Interior region] The interior region $\mathfrak V(r)$ parameterized by $r$
is the set of lattice points in $\mathcal V^3$ that have at least $L_{\infty}$-distance $r$ to the boundary.
\end{definition}

For any point $P\in \cV^3$,
define $\mathbb{B}(P,x)=\{Q\in\mathcal{V}^3:\|Q-P\|_\infty\leq x\}$ as
the $x$-ball of neighborhood of $P$ under $L_\infty$-norm. The following
proposition is our major result in this subsection.


\begin{proposition}
\label{lem:localbehavefinal}
Consider a diffusion following Definition~\ref{def:diffusion}. Let
$S_0$ be the initial position of the only infected agent $\mathrm a_1$
at time 0, and $\mathcal W$ be an arbitrary subset of lattice points
in $\mathfrak V(20\ell_2\log n)$,
  \yajunfinal{where $\ell_2=\sqrt{n^3/m}$. Denote}
$\Delta t= \ell^2_2\log^{-28}n$. Define the binary random variable $b(\mathcal W)$ as follows:
\begin{itemize*}
\item \emph{If $S_0\in \mathcal W$:} $b(\mathcal W)$ is set as 1 if and only if all the infected agents at time $\Delta t$ can be covered by the ball $\mathbb B(S_0, 9\ell_2\log n)$.
\item \emph{If $S_{0} \notin \mathcal W$:} $b(\mathcal W) = 1$.
\end{itemize*}
We have
\begin{equation}
\Pr[b(\mathcal W) = 1] \geq 1 - \exp(-5\log n \log \log n) \label{lower bound}
\end{equation}
\end{proposition}

The proposition yields that with high probability, all the infected agents lie within a neighborhood of distance $\tilde{O}(\ell_2)$ at time $\tilde{O}(\ell_2^2)$. The variable $\ell_2$ is chosen such that the expected number of infections spanned by an initially infected agent \xiaoruin{$\mathrm{a}_1$} within $\tilde{O}(\ell_2^2)$ units of time and a neighborhood of $\tilde{O}(\ell_2)$ distance is $O(1)$. This can be seen by solving $m(\ell_2/n)^3\times(1/\ell_2)=\tilde{O}(1)$, where $m(\ell_2/n)^3$ is the expected number of agents in a cube of size $\ell_2\times\ell_2\times\ell_2$, and $\tilde{O}(1/\ell_2)$ is the meeting probability within time $\tilde{O}(\ell_2^2)$ between any pair of random walks with initial distance $\ell_2$ (see Lemma~\ref{lem:couple}). This choice of $\ell_2$ appears to be the right threshold for our analysis. Indeed, a larger scale than $\ell_2$
would induce a large number of
infections
made by $\mathrm a_1$, and also
subsequent infections made by newly infected agents, with an exploding
affected region as an end result. On the other hand, a smaller scale
than $\ell_2$ would degrade our lower bound. This is because the
diffusion time is approximately of order $n/\ell_2$, the number of
spatial steps to cover $\mathcal{V}^3$, times $\ell_2^2$, the time
taken for each step, equaling $n\ell_2$. Hence a decrease in $\ell_2$
weakens the bound\footnote{In the case of general $d$-dimensional space, $\ell_2$ is chosen such that $m(\ell_2/n)^d\times(1/\ell_2^{d-2})=\tilde{O}(1)$, giving $\ell_2=\sqrt{n^d/m}$. Throughout the paper such $d$-dimensional analog can be carried out in similar fashion, but for ease of exposition we shall not bring up these generalizations and will focus on the 3-dimensional case.}.

Secondly, we introduce $\mathcal{W}$ in Proposition~\ref{lem:localbehavefinal} to avoid the case when $S_0$ is close to the boundary. As we have mentioned, such boundary conditions often complicate random walk analysis. Although the impact of the boundary's presence has been addressed (e.g., \cite{CPS09, PPPU11}),
existing results are not fully satisfactory. For example, when two simple random walks $S^1$ and $S^2$ start near the boundary, only a lower bound for the probability that two walks meet within a specific number of time steps is available (\cite{PPPU11}); we do not
know of an upper bound counterpart. We arrange our proof so that it is sufficient to analyze the diffusion pattern of a virus when it starts far from the boundary. Finally, we note that no effort has been made to optimize the exponent 28 in $\Delta t$'s definition.


We briefly explain how our global lower bound can be readily obtained from Proposition~\ref{lem:localbehavefinal}, which
 is a strong characterization of the local growth rate of infection region size.  Imagine the following evolution. Starting with a single infected agent, with high probability the infection spreads to a ball of radius at most $9\ell_2\log n$ in $\Delta t$ time units. At this time point, the newly infected agents \emph{inside} the ball continue to spread the virus to neighborhoods of size at most $9\ell_2\log n$, again with high probability. This gives an enlarged area of infection with radius at most $18\ell_2\log n$. Continuing in this way, the lower bound in Theorem~\ref{thm:ourtheorem} is then the time for the infection to spread over $\mathcal{V}^3$. This observation will be made rigorous in the next subsection.

The rest of this subsection is devoted to the proof of Proposition \ref{lem:localbehavefinal}. It consists of two
main steps. First, we need to estimate the expected number of
infections done by a single initially infected agent within distance
$9\ell_2\log n$ and time increment $\Delta t$.
Second, we iterate to consider each newly infected agent.  The analysis requires the condition that the global configuration behaves ``normally", a scenario that occurs with suitably high probability, as we show. We call this condition ``good behavior", which is introduced through the several definitions below:

\begin{definition}\label{def:island} \emph{(Island, \cite{PPPU11})} Let $A = \{\mathrm
  a_1, ..., \mathrm a_m\}$ be the set of agents in $\mathcal V^3$.
  For any positive integer $\gamma > 0$, let $G_t(\gamma)$ be the
  graph with vertex set $A$ such that there is an edge between two
  vertices if and only if the corresponding agents are within distance
  $\gamma$ (under $L_1$-norm) at time $t$. The island with parameter
  $\gamma$ of an agent $\mathrm a_i \in A$ at time step $t$, denoted
  by $\mathrm {Isd}_t(\mathrm a_i, \gamma)$ is the connected component
  of $G_t(\gamma)$ containing $\mathrm a_i$.
\end{definition}



\begin{definition} [Good behavior]
Let $\ell_1 = n m^{-1/3}$. For $1 \leq i \leq (\ell_2/\ell_1)\log^{-3}n$, define
$\mathcal B_i(P) = \mathbb B\left(P, i\ell_1\log^{-1}n\right)$ and let
$\partial \mathcal B_i(P) = \mathcal B_i(P) - \mathcal B_{i - 1}(P)$. For any $P \in \mathcal V^3$,
define
  \yajunfinal{$m_i(P) = \frac{(\log^5n)|\partial \mathcal B_i(P)|m}{(2n+1)^3}$. }
Let us define the following binary random variables:
\begin{itemize*}
\item {\bf Good density.}
Let $\{D_t: t \geq 0\}$  be a sequence of $0,1$ random
variables such that $D_t = 1$ if and only if
for all $P \in \mathcal V^3$ and all $i \leq (\ell_2 /
\ell_1)\log^{-3}n$, the number of agents in $\partial \mathcal B_i(P)$ is at
most $m_i(P)$, for all time
steps
up to $t$. We say the diffusion process has the
\emph{good density property} at time $t$ if $D_t = 1$.
\item {\bf Small islands.}
Let $\{E_t: t \geq 0\}$  be a sequence of $0,1$ random variables such that $E_t = 1$ if and only if $|\isd_{s}(\mathrm{a}_j, \ell_1\log^{-1}n)| \leq 3\log n$ for all $\mathrm a_j \in A$ and $0\leq s\leq t$. We say that the diffusion process has the
\emph{small islands property} at time $t$ if $E_t = 1$.
\item {\bf Short travel distance.}
Let $\{L_t : t \geq 0\}$ be a sequence of $0$, $1$ random variables
such that $L_t = 1$ if and only if for all $i \in [m]$ and all $t_1 < t_2 \leq t$ with $t_2 - t_1 \leq \ell^2_2 \log^{-12}n$,
we have $\|S^i_{t_1} - S^i_{t_2}\|_1 \leq 3\ell_2\log^{-4}n$. We say
the process has the
  \yajunfinal{\emph{short travel distance property} }
at time $t$ if $L_t = 1$.
\end{itemize*}
Finally, let $G_t = D_t \times E_t \times L_t$, and say the diffusion process \emph{behaves well} at time $t$ if $G_t = 1$. We also focus on $t\leq n^{2.5}$ and define the random variable $G = G_{n^{2.5}}$. \label{def:goodbehavior}
\end{definition}

The value $n^{2.5}$ in the definition is chosen such that it
lies well beyond our lower bound for the case
$m<n^3$, but is small enough for our
forthcoming union bound.
By using properties of random walks and techniques derived in ~\cite{PPPU11}, we have

\begin{lemma}
\label{lem:3dgoodevent}
Let $A = \{\mathrm a_1, ..., \mathrm a_m\}$ be agents that are
distributed uniformly in $\mathcal V^3$ at $t = 0$.
For
sufficiently large $n$, we have $\Pr[
G = 1] \geq
1- \exp(-6\log n\log \log n)$.
\end{lemma}

 \zliu{The proof of Lemma~\ref{lem:3dgoodevent} is presented in Appendix~\ref{sec:missingproofs}.}
With this global ``good behavior", we have the following estimate:

\def\lem3dlocalvisit{
Let $A = \{\mathrm a_1, \ldots, \mathrm a_m\}$ be
  agents that are  distributed in
  $\cV^3$ in such a way that $D_0 = 1$. Let $S^1, S^2, \ldots, S^m$
  be their corresponding random walks.
  Consider an
  arbitrary agent $\mathrm a_j$ with $S^j_0\in\mathfrak V(2\ell_2 \log^{-4}n)$.
   Let
  $\{\mathrm a_{i_1}, \ldots, \mathrm a_{i_k}\}$ be the set of agents outside
 $\mathcal B_1(S^j_{0})$ at time $0$. Define $X_{j, \ell}$ as
 the indicator random variable that represents whether
  the agents $\mathrm a_j$ and $\mathrm a_{i_{\ell}}$  meet within time
  $[0, \Delta t]$. We have
 $$\E\left[\sum_{\ell \leq k}X_{j, \ell} \Bigg| D_{0} = 1,
 S^j_0 \in \mathfrak V(2\ell_2\log^{-4}n)\right]
 < \log^{-3}n.$$
}
\begin{lemma}
\label{lem:3dlocalvisit}
{\lem3dlocalvisit}
\end{lemma}

\begin{proof}
First, notice that the number of lattice points in $\partial \mathcal B_i(P)$ satisfies
$$|\partial \mathcal B_i|  =  |\mathcal B_i| -
|\mathcal B_{i - 1}| \leq  (2i\ell_1\log^{-1}n)^3 - (2(i - 1)\ell_1\log^{-1}n)^3
 \leq  24i^2\ell^3_1\log^{-3}n.$$
We may also similarly show that
$$|\partial \mathcal B_i| \geq i^2\ell_1^3\log^{-3}n.$$

Let $q = (\ell_2 / \ell_1)\log^{-3}n $. For each $i \in [q]$, write
$\mathcal B_i = \mathcal B(S^j_{0})$, $\partial \mathcal B_i
= \partial \mathcal B_i(S^j_{0})$, and $m_i = m_i(S^j_{0})$. We want to estimate the meeting probability and hence the expected number of infections for each $i\in[q]$.


First, let us consider the
agents outside the ball $\mathcal B_q$. The probability that any
specific agent initially outside $\mathcal B_q$ ever travels into the ball $\mathbb B(S^j_0,
\ell_2\log^{-4}n)$ within time $\ell^2_2\log^{-12}n$ is at most
$\exp(-\Omega(\log^3n))$ ( \zliu{by}, e.g.,  Lemma~\ref{lem:basicmove} in the section on probability review). On the other
hand, the probability that $S^j$ ever travels out of $\mathbb
B(S^j_0, \ell_2\log^{-4}n)$ is also $\exp(-\Omega(\log^3n))$. For these two agents to meet,  \zliu{at least} one of these events \zliu{has to} occur. Therefore, with
probability $\exp(-\Omega(\log^3n))$ $S^j$ will meet an agent initially
outside $\mathcal B_q$. This leads to
$$\E\Bigg[\underbrace{\sum_{i_{k'}:S^{i_{k'}}_0 \notin \mathcal B_q}X_{j, k'}}_{\substack{\mbox{the set of agents initially }\\ \mbox{outside $\mathcal B_q$}}} \Bigg| D_{0} = 1, S^j_0 \in \mathfrak V(2\ell_2\log^{-4}n)\Bigg]
 \leq m\exp(-\Omega(\log^{3}n)).$$

Let us next focus on agents inside $\mathcal B_q$.
  Fix an arbitrary $\mathrm
a_{i_\ell} \in \partial \mathcal B_i$. Let $e^j$ and $e^{\ell}$
represents the events that $S^j$ and $ \yajun{S^{{i}_\ell}}$ ever visit a boundary before time $\ell^2_2\log^{-12}n$ respectively. Again by Lemma~\ref{lem:basicmove},  $\Pr[e^j \vee e^{\ell}|D_{0} = 1, S^j_0 \in \mathfrak V(2\ell_2\log^{-4}n)] = \exp(-\Omega(\log^3n))$.
We now have
\begin{eqnarray*}
& & \E[X_{j, \ell}| D_0 = 1, S^j_0 \in \mathfrak V(2\ell_2\log^{-4}n)]\\
& = &
\Pr[X_{j, \ell}=1;\neg e^j\wedge\neg e^\ell| D_0 = 1, S^j_0 \in \mathfrak V(2\ell_2\log^{-4}n)]+\Pr[X_{j, \ell}=1;e^j\vee e^\ell| D_0 = 1, S^j_0 \in \mathfrak V(2\ell_2\log^{-4}n)]\\
& \leq & \Pr[X_{j, \ell}=1;\neg e^j\wedge\neg e^\ell| D_0 = 1, S^j_0 \in \mathfrak V(2\ell_2\log^{-4}n)] + \exp(-\Omega(\log^3n)).
\end{eqnarray*}
To compute {\small
\begin{eqnarray*}
& & \Pr[X_{j, \ell};\neg e^j\wedge\neg e^\ell| D_0 = 1, S^j_0 \in \mathfrak V(2\ell_2\log^{-4}n)] \\
& =&\Pr\left[\left\{\exists t_0 \leq \frac{\ell^2_2}{\log^{12}n}: \|S^j_{t_0} - S^{i_\ell}_{t_0} \|_1\leq1\right\}\bigwedge \left(\neg e^j \wedge \neg e^{\ell} \right)\Bigg| D_0 = 1, S^j_0 \in \mathfrak V(2\ell_2\log^{-4}n)\right],
\end{eqnarray*}}
we couple $S^j$ and $ \zliu{S^{i_\ell}}$ with unbounded walks $\mathbf S^j$ and $ \zliu{\mathbf S^{i_\ell}}$ starting at the same positions at $t = 0$ in the natural way. Before the pair of bounded walks visit the boundary, they coincide with their unbounded counterparts. Therefore,
we have
{
\small
\begin{eqnarray*}
& & \Pr\left[\left\{\exists t_0 \leq \frac{\ell^2_2}{\log^{12}n}: \|S^j_{t_0} - S^{i_\ell}_{t_0}\|_1\leq1 \right\}\bigwedge \left(\neg e^j \wedge \neg e^{\ell} \right)\Bigg| D_0 = 1, S^j_0 \in \mathfrak V(2\ell_2\log^{-4}n)\right] \\
&\leq& Pr\left[\exists t_0 \leq \frac{\ell^2_2}{\log^{12}n}: \|\mathbf S^j_{t_0} - \mathbf S^{i_\ell}_{t_0}\|_1\leq1\Bigg| D_0 = 1, S^j_0 \in \mathfrak V(2\ell_2\log^{-4}n)\right] = O(\frac 1{(i - 1)\ell_1}) \quad \mbox{(Corollary~\ref{cor:rwcatchall})}
\end{eqnarray*}
}

We thus have $\E[X_{j, \ell}| D_0 = 1, S^j_0 \in \mathfrak V(2\ell_2\log^{-4}n)] \leq \frac{C_0}{(i-i)\ell_1}$ for some constant
$C_0$. Next, we estimate $\E_i \equiv \E[\sum_{\ell: S^{i_\ell}_0 \in \partial \mathcal
  B_i}X_{j, \ell} \mid D_0 = 1, S^j_0 \in \mathfrak V(2\ell_2\log^{-4}n) ]$ as
$$\E_i  \leq \frac{C_0m_i}{(i - 1)\ell_1} =
\frac{C_0}{(i-1)\ell_1}\cdot \frac{
  |\partial \mathcal B_i|m\log^5n }{8n^3} =
\frac{C_0(3i^2\ell^3_1\log^{-3}n)m\log^5n}{(i - 1)\ell_1n^3}\leq
 \frac{6C_0im\ell^2_1\log^2 \yajun{n}}{n^3}.$$
The first inequality holds because $D_{0} = 1$ and $m_i$ is  an upper
bound for the number of agents in $\partial \mathcal B_i$ (for all
$i$).
\begin{eqnarray*}
\E[\sum_{\ell \leq k}X_{j, \ell}\mid D_{0} = 1, S^j_0 \in \mathfrak V(2\ell_2\log^{-4}n)] & \leq & \sum_{i \leq q}\E_i + \underbrace{m \times \exp(-\Omega(\log^3n))}_{\mbox{upper bound for those outside }\mathbb \mathcal B_q}\\
& \leq & \left(\sum_{2 \leq i \leq q}i\right)\frac{6C_0m
  \ell^2_1\log^2 n}{n^3} + \exp(-\Omega(\log^2n))\\
& < & \frac{6C_0q^2m\ell^2_1\log^2 n}{n^3} + \exp(-\Omega(\log^2n))\\
& = & 6C_0\log^{-4}n + \exp(-\Omega(\log^2n))\\
& < & \log^{-3}n
\end{eqnarray*}
for sufficiently large $n$ as $m<n^3$.
\end{proof}

Lemma~\ref{lem:3dlocalvisit} says that if the initial distribution of agents possesses good behavior, then one can ensure that the expected number of direct infections on far-away agents is small.
For agents close to the initially infected agents, we instead utilize the concept of islands, which is also deeply related to the subsequent virus spreading behavior. Now we formally introduce a new diffusion process with a modified ``island diffusion" rule. It is easy to see that this new diffusion process can be naturally coupled with the original diffusion process (evolving with Definition~\ref{def:diffusion}) by using the same random walks in the same probability space.

\begin{definition}[Diffusion process with island diffusion rule]
  Consider a diffusion process in which $m$ agents are performing
  random walks on $\mathcal V^3$. An uninfected agent $\mathrm a_j$
  becomes infected at time $t$ if one of the following conditions
  holds:
\begin{enumerate*}
\item it meets a previously infected agent at time $t$. For convenience, we say $\mathrm a_j$ is \emph{directly} infected if it is infected in this way.
\item it is inside $\isd_{t}(\mathrm a_i, \ell_1\log^{-1}n)$ where $\mathrm{a}_i$ is directly infected at time $t$. 
\end{enumerate*} \label{def:islanddiffusion}
\end{definition}

This coupled process is different from the diffusion
  models introduced in ~\cite{PPPU11,SS10,PSSS11}. In our formulation, an island is infected only if meeting occurs between one uninfected and one previously infected agent. In ~\cite{PPPU11,SS10,PSSS11} (using our
  notations), an island is infected once it contains a previously infected agent. As a result, infections occur less frequently in our model than the models in ~\cite{PPPU11,SS10,PSSS11}. This difference is the key to getting a tight lower bound for dimensions higher than 2. More precisely, our infection rule allows us to build a terminating branching process, or what we call ``diffusion tree" in the following definition, whose generations are defined via the infection paths from the source. The termination of this branching process constrains the region of infection to a small neighborhood around the source with a probability of larger order than obtained in~\cite{PPPU11}. This in turn leads to a tighter global lower bound.
%

\begin{definition}[Diffusion tree]\label{def:diffusetree}
Let $\mathcal W \subseteq \mathfrak V(2\ell_2\log n)$ be a subset of lattice points. Consider a diffusion, following the island diffusion rule, that starts with an initially infected island $\isd_{0}(\mathrm a_1, \ell_1\log^{-1}n)$.
 Recall that $S^1_0$ denotes $\mathrm a_1$'s position at $t = 0$.
 The diffusion tree $\T$ with respect to  $\mathcal W$ has the following components:
\begin{enumerate*}
\item If $S^1_{0} \notin \mathcal W$, $\T = \emptyset$.
\item If $S^1_{0} \in \mathcal W$,
\begin{itemize*}
\item The \emph{root} of $\T$ is a dummy node $r$.
\item The children of $r$ are all
the agents in $\isd_{0}(\mathrm a_1, \ell_1\log^{-1}n)$.
\item $\mathrm a_{\ell'}$ is a
{\em child} of $\mathrm a_\ell$ ($\mathrm a_{\ell'} \in \child(\mathrm a_\ell)$)
if $\mathrm a_{\ell'}$ is infected   \yajunfinal{by}
$\mathrm a_\ell$ before time $\Delta t$.
\item $\mathrm a_{\ell'}$ is a
{\em direct child} of $a_\ell$ ($\mathrm a_{\ell'} \in
\dchild(\mathrm a_\ell)$) if $\mathrm a_{\ell'} \in \child(\mathrm
a_{\ell})$ and it is {\em directly} infected by $\mathrm a_\ell$.
\end{itemize*}
\end{enumerate*}
For technical reasons, if $\mathrm a_{\ell'}$ is not in $\T$, we let $\child(\mathrm a_{\ell}) = \emptyset$ and $\dchild(\mathrm a_{\ell}) = \emptyset$.
\end{definition}

Figure~\ref{fig:branching} in the Appendix shows an example of the diffusion process and its corresponding
diffusion tree at $t = 0, 20, 40, 60$. Notice that the diffusion tree $\T$ stops growing after $\Delta t$ steps.

We refer the root of the tree as the $0$th level of the tree and  count levels in the standard way.
The height of the tree is the number of levels in the tree. Note that diffusion tree defined in this way can readily be
interpreted as a branching process (See, e.g., Chapter 0 in
\cite{Williams91}),
in which the $j$th generation of the process corresponds with the
$j$th level nodes in $\T$.

Next we incorporate the good behavior variable $G_t$ with diffusion tree.
The motivation is that, roughly speaking, consistently good behavior
guarantees a small number of infections, or creation of children, at each
level. This can be seen through Lemma~\ref{lem:3dlocalvisit}.

\begin{definition}[Stopped diffusion tree]\label{def:stoptree}Consider a diffusion process with island diffusion rule, and let $T(\ell)$ be the time that
$\mathrm a_{\ell}$ becomes infected in the process.
The stopped diffusion tree
 $\T'$ (with respect to $\mathrm{a}_i$ and $\mathcal W$) is a subtree of
 $\T$ induced by the set of vertices
 $\{\mathrm a_{\ell}:\mathrm a_{\ell} \in \T\ \wedge \ G_{T(\ell)} = 1\}$.
 We write $\mathrm a_{\ell} \in \child'(\mathrm a_{\ell'})$ if $\mathrm a_{\ell} \in \child(\mathrm a_{\ell'})$ and $\mathrm a_{\ell} \in \T'$. Similarly, $\mathrm a_{\ell} \in \dchild'(\mathrm a_{\ell'})$ if $\mathrm a_{\ell} \in \dchild(\mathrm a_{\ell'})$ and $\mathrm a_{\ell} \in \T'$.
\end{definition}

Note that the definition of stopped diffusion tree involves global
behavior of the whole diffusion process due to the introduction of
$G_t$.  On the other
  hand, $\T=\T'$ with overwhelming probability, so we can
  translate the properties of $\T'$ back to $\T$ easily.

We next show two properties of
the (stopped) diffusion trees, one on the physical propagation
of children relative to their parents and one on the tree
height. These are our main ingredients for proving Proposition~\ref{lem:localbehavefinal}.
The properties are in brief:
\begin{enumerate*}
\item If $\mathrm a_{\ell}$ is a child of $\mathrm a_{\ell'}$ in the
stopped diffusion tree $\T'$,
$\|S^{\ell}_{T(\ell)} - S^{\ell'}_{T(\ell')}\|_{\infty}$ is $\tilde
O(\ell_2)$.
%
\item The height of the stopped diffusion tree $\T'$ is $\tilde O(1)$ with high probability.
\end{enumerate*}


Proving the first item requires the following notion:

\begin{definition}[Generation distance] Consider the diffusion
tree $\T$ with respect to $\mathcal W$.
Let $\mathrm a_{\ell}$ be an arbitrary agent and let $\mathrm a_{\ell'}$ be its parent on $\T$.
The \emph{generation distance} of $\mathrm a_{\ell}$ with respect to $\T$ is
\begin{equation}
\mathfrak d_{\ell} =
\left\{
\begin{array}{ll}
\|S^{\ell}_{T(\ell)}
- S^{\ell'}_{T(\ell')}\|_1 & \mbox{ if $\mathrm a_{\ell}$ is in $\T$ and is at the 2nd or deeper level} \\
0& \mbox{ otherwise.}
\end{array}
\right.
\end{equation}
In other words, the generation distance between $\mathrm a_{\ell}$ and $\mathrm a_{\ell'}$ is
the distance between where $\mathrm a_{\ell}$ and $\mathrm a_{\ell'}$ were infected.
The generation distance of $\mathrm a_{\ell}$ with respect to $\T'$ is $\mathfrak d'_{\ell}$, which
is set to be $\mathfrak d_{\ell}$ if $\mathrm a_{\ell}$ is in $\T'$ and $0$ otherwise.
\end{definition}

With this notion, we can  \zliu{derive} the following lemma:

\begin{lemma}\label{lem:gendistance}
Consider the stopped diffusion tree with respect to $\mathcal W$ that starts with an infected island $\isd_0(\mathrm{a}_i,\ell_1\log^{-1}n)$.
For an agent $\mathrm a_\ell$ in $\T'$,
$\mathfrak d'_{\ell} \leq 4\ell_2 $.
\end{lemma}
\begin{proof}We focus on the non-trivial case that $S_0^i\in\mathcal W$ and that $\mathrm a_\ell$ is at the 2nd level of $\T'$ or deeper. Suppose that the diffusion process behaves well up to time $T(\ell)$ i.e. $G_{T(\ell)}= 1$. Let $\mathrm a_{\ell'}$ be the parent of $\mathrm a_\ell$ on $\T'$. By the construction of $\T'$, there exists an $\mathrm a_{\ell''} \in \dchild'(\mathrm a_{\ell'})$ (possibly $\mathrm a_\ell$ itself) such that
\begin{itemize*}
\item $\mathrm a_{\ell} \in \isd_{T(\ell'')}(\mathrm a_{\ell''}, \ell_1\log^{-1}n)$.
\item $T(\ell) = T(\ell'')$ i.e., $\mathrm a_{\ell}$ and $\mathrm a_{\ell''}$ get infected at the same time due to the island diffusion rule.
\item $\|S^{\ell'}_{T(\ell)} - S^{\ell''}_{T(\ell)}\|_1 \leq 1$ i.e., $\mathrm a_{\ell''}$ gets infected because it meets an infected agent.
\end{itemize*}
By the triangle inequality,

$$\mathfrak d_{\ell}' = \|S^{\ell'}_{T(\ell')} - S^{\ell}_{T(\ell)}\|_1 \leq \|S^{\ell'}_{T(\ell')} - S^{\ell''}_{T(\ell'')}\|_1 + \|S^{\ell''}_{T(\ell'')} - S^{\ell}_{T(\ell)}\|_1.$$
Note that $\|S^{\ell'}_{T(\ell')} - S^{\ell''}_{T(\ell'')}\|_1 \leq
3\ell_2\log^{-4}n + 1 \leq \ell_2$  (short travel distance property)
and $\|S^{\ell''}_{T(\ell'')} - S^{\ell}_{T(\ell)}\|_1 \leq (\ell_1
\log^{-1}n) (3\log n) =3\ell_1 \leq 3\ell_2$ (small island property) since $G_{T(\ell)}=1$. Finally, the case when $G_{T(\ell)}=0$ is trivial, and the lemma follows.
\end{proof}

Next we show that with high probability the
height of the stopped diffusion tree is $\tilde O(1)$.
 Using standard notation, we let $\{\mathcal F_t\}_{t\geq0}$ be  \zliu{the} $\sigma$-algebra, or filtration, generated up to time $t$, i.e., $\mathcal F_t$ encodes
 all the information regarding the diffusion process up to $t$.
The special instance $\mathcal F_0$ is used to describe
 the initial positions of the agents.

The main
  property of stopped diffusion tree that we need is the following:

\def\lemdchild{
Consider a diffusion process
 with the island diffusion rule. Let $\mathrm a_{\ell}$ be an arbitrary
 agent with infection time $T(\ell)$. We have 
 \begin{equation}\label{eqn:meetcount}
 \E\left[|\dchild'(\mathrm a_{\ell})|\Big| \mathcal F_{T(\ell)}, S^{\ell}_{T(\ell)} \in \mathfrak V(2\ell_2\log^{-4}n)\right]\leq \log^{-3}n,
 \end{equation}
 where $\dchild'(\cdot)$ is defined for a stopped diffusion tree with respect to an arbitrary set $\mathcal W \subseteq \mathfrak V(20\ell_2\log n)$.
}
\begin{lemma}
\label{lem:dchild}
{\lemdchild}
\end{lemma}

We regard the conditional expectation in
 Equation~\ref{eqn:meetcount}
as a random variable. The interpretation is that the expected number
of $\mathrm{a}_\ell$'s direct children is less than
  \yajunfinal{$\log^{-3}n$,}
regardless of the global configuration at the infection time of $\mathrm{a}_\ell$, as long as it lies in $\mathfrak V(2\ell_2\log^{-4}n)$ at that time.

\begin{proof}We focus on the case when $S^1_0 \in \mathcal W$; otherwise $\T'$ is empty and the lemma trivially holds.
 First observe that all $\mathrm a_j \in \isd_{T(\ell)}(S^{\ell}_{T(\ell)}, \ell_1\log^{-1}n)$ are infected at or before the time $\mathrm a_{\ell}$ is infected. Therefore they cannot be direct children of $\mathrm a_{\ell}$ by Definition~\ref{def:diffusetree} and \ref{def:stoptree}. On the other hand, an agent $\mathrm a_j$ is outside $\isd_{T(\ell)}(S^{\ell}_{T(\ell)}, \ell_1\log^{-1}n)$ only if it is outside the ball $\mathbb B(S^{\ell}_{T(\ell)}, \ell_1\log^{-1}n)$. Hence $\dchild'(\mathrm a_{\ell})$ is bounded by the number of agents initially outside $\mathbb B(S^{\ell}_{T(\ell)}, \ell_1\log^{-1}n)$ that meet $\mathrm a_{\ell}$ before time $\Delta t$. We consider two cases:

{\noindent \emph{Case 1.}} $D_{T(\ell)} = 1$. By Lemma~\ref{lem:3dlocalvisit}, we have
$$\E\left[|\dchild'(\mathrm a_{\ell})|\Big|\mathcal
  F_{T(\ell)}, D_{T(\ell)} = 1, S^{\ell}_{T(\ell)} \in \mathfrak V(2\ell_2\log^{-4}n)
\right] \leq \log^{-3}n.$$

{\noindent \emph{Case 2.}} $D_{T(\ell)} = 0$. By Definition~\ref{def:stoptree}, we have
$$\E\left[|\dchild'(\mathrm a_{\ell})|\Big|\mathcal
  F_{T(\ell)}, D_{T(\ell)} = 0,  S^{\ell}_{T(\ell)} \in \mathfrak V(2\ell_2\log^{-4}n)
\right] = 0 \leq \log^{-3}n.$$

Therefore,
{\small
\begin{eqnarray*}
\E\left[|\dchild'(\mathrm a_{\ell})|\Big| \mathcal
  F_{T(\ell)}, S^{\ell}_{T(\ell)} \in \mathfrak
  V(2\ell_2\log^{-4}n)\right]  &= & \E_{\mathcal
  F_{T(\ell)}}\left[\E\left[|\dchild'(\mathrm
    a_{\ell})|\Big| \mathcal F_{T(\ell)}, S^{\ell}_{T(\ell)} \in
    \mathfrak V(2\ell_2\log^{-4}n), D_{T(\ell)}\right]\right]\\
 &\leq &\log^{-3}n.
\end{eqnarray*}}

\end{proof}

Recursive utilization of Lemma~\ref{lem:dchild} on successive tree levels leads to the following lemma:
%
\def\lemshortgeneration{
Consider a diffusion process with the island diffusion rule starting with an infected island $\isd \yajunfinal{_0}(\mathrm{a}_1,\ell_1\log^{-1}n)$. For the stopped diffusion tree $\T'$ with respect to any $\mathcal W \subseteq \mathfrak V(20\ell_2\log n)$, let $\mathrm{Height}(\T')$ be its height. Then we have
\begin{equation}
\Pr[\mathrm{Height}(\T') > 2\log n] \leq \exp(-3\log n\log \log n).\label{eqn:heightbound}
\end{equation}
}
\begin{lemma}
\label{lem:short_generation}
{\lemshortgeneration}
\end{lemma}

Let us denote the set of agents at the $k$th level as $\mathbb F_k$.
It is worth pointing out that, despite a similar analysis to that of standard branching process, there is a technical complication on the conditioning argument since the creation of each child within the \emph{same} level can be performed at \emph{different} times in the diffusion process. This implies that there is no single filtration that we can condition on each level to analyze the expected size of the next one. Nevertheless, conditioning can be tailored to \emph{each} agent at the same level.

\begin{proof} We focus on the case when $\T'$ is non-trivial i.e. $S_0^1\in \mathcal{W}$. Let $I(A)=1$ if $A$ occurs and 0 otherwise. We have, for any $k<2\log n$, 
\begin{eqnarray*}
& & \E[|\mathbb F'_{k+1}||\mathcal{F}_{0},S_{0}^1\in \mathcal{W}] \\
& = & \E\left[\sum_{\mathrm a_{\ell}
      \in \mathbb F'_k}\sum_{\mathrm a_{\ell'}\in \dchild'(\mathrm
      a_{\ell})}|\isd_{T(\ell')}(\mathrm a_{\ell'},
    \ell_1\log^{-1}n)|I(G_{T(\ell')}=1)\mathbf{\Bigg|}\mathcal F_{0}, S^1_{0} \in \mathcal W\right] \\
& \leq & (3\log n)\E\left[\sum_{\mathrm
      a_{\ell} \in \mathbb F'_k}|\dchild'(\mathrm
    a_{\ell})|\mathbf{\Bigg|}\mathcal F_{0}, S^1_{0} \in \mathcal W\right] \\
\end{eqnarray*}
The equality holds by the stopping rule and the inequality holds by the small islands property. Next we have
{\small
\begin{eqnarray*}
& & \E\left[\sum_{\mathrm a_{\ell} \in \mathbb F'_k}|\dchild'(\mathrm a_{\ell})|\mathbf{\Bigg|}\mathcal F_{0},S^1_{0} \in \mathcal W\right] \\
& = & \E\left[\sum_{\ell \in [m]}I(\mathrm a_{\ell} \in \mathbb F'_k)|\dchild'(\mathrm a_{\ell})|\mathbf{\Bigg|}\mathcal F_{0},S^1_{0} \in \mathcal W\right]\\
& = & \sum_{\ell \in [m]}\E\left[I(\mathrm a_{\ell} \in \mathbb F'_k)|\dchild'(\mathrm a_{\ell})|\mathbf{\Bigg|}\mathcal F_{0},S^1_{0} \in \mathcal W\right] \\
& = & \sum_{\ell \in [m]}\E\left[\E\left[I(\mathrm a_{\ell} \in \mathbb F'_k)|\dchild'(\mathrm a_{\ell})|\mathbf{\Bigg|}\mathcal F_{T(\ell)},S^1_{0} \in \mathcal{W}\right]\Bigg|\mathcal F_{1},S^1_{0} \in \mathcal W\right]\\
& = & \sum_{\ell \in [m]}\E\left[I(\mathrm a_{\ell} \in \mathbb F'_k)\E\left[|\dchild'(\mathrm a_{\ell})|\mathbf{\Bigg|}\mathcal F_{T(\ell)},S^1_{0} \in \mathcal{W}\right]\Bigg|\mathcal F_{0},S^1_{0} \in \mathcal W\right]
\mbox{(Because $I(\mathrm a_{\ell} \in \mathbb F'_k)$ is $\mathcal F_{T(\ell)}$-measurable)}
\end{eqnarray*}
}
Note that $S^1_{0} \in \mathcal W \subseteq \mathfrak V(20\ell_2\log
n)$ implies $S^{\ell}_{T(\ell)} \in \mathfrak (20\ell_2\log n -
4k\ell_2) \subset \mathfrak V(2\ell_2\log^{-4}n)$ if $G_{T(\ell)}=1$,
by using Lemma~\ref{lem:gendistance}.
Therefore, by Lemma~\ref{lem:dchild}
$$\E\left[|\dchild'(\mathrm a_{\ell})|\mathbf{\Bigg|}\mathcal F_{T(\ell)},S^1_{0} \in \mathcal{W},G_{T(\ell)}=1\right] = \E\left[|\dchild'(\mathrm a_{\ell})|\mathbf{\Bigg|}\mathcal F_{T(\ell)},S^{\ell}_{T(\ell)} \in \mathfrak V(2\ell_2\log^{-4}n),G_{T(\ell)}=1\right] \leq \log^{-3}n$$
On the other hand,
$$\E\left[|\dchild'(\mathrm a_{\ell})|\mathbf{\Bigg|}\mathcal F_{T(\ell)},S^1_{0} \in \mathcal{W},G_{T(\ell)}=0\right] =0$$
by the stopping rule. This leads to
$$\E\left[|\dchild'(\mathrm a_{\ell})|\mathbf{\Bigg|}\mathcal F_{T(\ell)},S^1_{0} \in \mathcal{W}\right]\leq\log^{-3}n$$
which implies
$$\sum_{\ell \in [m]}\E\left[I(\mathrm a_{\ell} \in \mathbb F'_k)\E\left[|\dchild'(\mathrm a_{\ell})|\mathbf{\Bigg|}\mathcal F_{T(\ell)},S^1_{0} \in \mathcal{W}\right]\Bigg|\mathcal F_{0},S^1_{0} \in \mathcal W\right]\leq\log^{-3}n\E[|\mathbb F'_k||\mathcal{F}_{0},S_{0}^1\in \mathcal{W}]$$
Therefore,
$$\E[|\mathbb F'_{k + 1}| \mid \mathcal F_0, S^1_0 \in \mathcal W] \leq 3\log^{-2}n \E[|\mathbb F'_k|\mid \mathcal F_0, S^1_0 \in \mathcal W] \leq (3\log^{-2}n)^k\E[|\mathbb F'_1|\mid \mathcal F_0, S^1_0 \in \mathcal W] \leq \log n(3\log^{-2}n)^k$$
and hence
 \zliu{$$\Pr[|\mathbb F'_{2\log n}| > 0] \leq \E[|\mathbb F'_{2\log n}| ] \leq \exp(-3\log n \log \log n).$$}
by combining with the case $\mathrm{a}_1\notin \mathcal W$.
\end{proof}

We now prove Proposition~\ref{lem:localbehavefinal}.
\begin{proof}[Proof of Proposition~\ref{lem:localbehavefinal}]
First note that the set of infected agents in a diffusion process with island diffusion rule, namely Definition~\ref{def:islanddiffusion}, is always a superset of the coupled original diffusion process using Definition \ref{def:diffusion}, at any time from 0 to $\Delta t$. Next we have
\begin{eqnarray*}
& & \Pr[\mathrm{Height}(\T) > 2\log n ] \\
& = &
 \Pr[(\mathrm{Height}(\T)> 2\log n) \wedge (\mathrm{Height}(\T) = \mathrm{Height}(\T'))] \\
  & & \quad +   \Pr[(\mathrm{Height}(\T)> 2\log n ) \wedge( \mathrm{Height}(\T) \neq \mathrm{Height}(\T'))] \\
  & \leq & \Pr[\mathrm{Height}(\T') > 2\log n] + \Pr[\T' \neq \T] \\
  & \leq & \exp(-3\log n \log \log n) + \Pr[G = 1] \\
  & \leq & 2\exp(-3\log n \log \log n)
\end{eqnarray*}
Therefore, we have
$$\Pr[(\mathrm{Height}(\T) \leq 2\log n)\wedge (G = 1)] \geq 1 - 3\exp(-3\log n \log \log n).$$
We will show that the viruses can be covered by the ball $\mathbb
B(S^1_{0}, \xiaorui{9} \ell_2\log n)$ when $$(\mathrm{Height}(\T) \leq 2\log
n)\wedge (G = 1) \wedge (S^1_0 \in \mathcal W).$$
Fix arbitrary infected $\mathrm a_{\ell} \in \mathbb F_k$ with $k \leq
2\log n$ .
By Lemma~\ref{lem:gendistance},
we have $\|S^1_{0} - S^{\ell}_{T(\ell)}\|_1 \leq 8 \ell_2\log n$.
Moreover,  $G = 1$ implies that for \emph{all}
$0 \leq t' \leq \Delta t$, $\|S^{\ell}_{T(\ell)} - S^{\ell}_{t'}\|_1 \leq 3\ell_2\log^{-4}n \leq \ell_2\log n$.
This suggests
$\|S^{\ell}_{t'} - S^{1}_{0}\|_{\infty} \leq 9\ell_2 \log n$ for all $t'\in[0,\Delta t]$.
Therefore, the virus does not escape the
ball $\mathbb B(S^1_{0}, 9\ell_2\log n)$ within time $[0,
\Delta t]$.
\end{proof}

%

\subsection{From local to global process}\label{subsec:globallow}
This section will be devoted to proving Theorem~\ref{thm:lb3} via Proposition~\ref{lem:localbehavefinal}, or in other words, to turn our local probabilistic bound into a global result on the diffusion time.

We note that Proposition~\ref{lem:localbehavefinal} deals with the case when there is only one initially infected agent. As discussed briefly in the discussion following the proposition, we want to iterate this estimate so that at every time increment $\Delta t$, the infected region is constrained within a certain radius from the initial positions of all the agents that are already infected at the start of the increment.
  \zliu{Our argument is aided by noting which agents infect other agents.}
 To ease the notation for this purpose, we introduce an artificial concept of virus type, denoted by $\nu_{i,t}$. We say an agent gets a virus of type $\nu_{i,t}$ if the meeting events of this agent can be traced upstream to the agent $\mathrm{a}_i$, where $\mathrm{a}_i$ is already infected at time $t$. In other words, assume that $\mathrm{a}_i$ is infected at time $t$, and imagine that we remove the viruses in all infected agents except $\mathrm{a}_i$ but we keep the same dynamics of all the random walks. We say a particular agent gets $\nu_{i,t}$ if it eventually gets infected under this imaginary scenario. Note that under this artificial framework of virus types it is obvious that an agent can get many different types of virus, in terms of both $i$ and $t$.

%
%

In parallel to Proposition~\ref{lem:localbehavefinal}, we introduce the family of binary random variables $b_{i, t}$ to represent whether a virus of type $\nu_{i, t}$ can be constrained in a ball with radius $9\ell_2\log n$:

\begin{definition} [$b_{i, t}$ and virus of type $\nu_{i, t}$]
 Let $\overline{\mathfrak B}=
\mathbb B(P, \frac n 4)$ where $P = (n/2,n/2,n/2)$.
Let $\mathrm a_1, ..., \mathrm a_m$ be agents that are uniformly distributed
on $\mathcal V^3$ at $t = 0$ and diffuse according to Definition~\ref{def:diffusion}.  Let $t$ be
an arbitrary time step and $i \in [m]$.
At time $t$, a virus of type $\nu_{i, t}$ emerges  \yajun{on agent
  $\ma_i$} and diffuses.
Define the binary random variable $b_{i, t}$ as follows:
\begin{itemize}
\item \emph{If $S^i_{t} \in \overline{\mathfrak B}$:} $b_{i, t}$ is
  set as $1$ if and only if all the agents infected by the virus of
  type
$\nu_{i, t}$ \liusd{at time $t + \Delta t$} can be covered by the ball $\mathbb B(S^i_{t}, 9\ell_2\log n)$.
\item \emph{If $S^i_{t} \notin \overline{\mathfrak B}$:} $b_{i, t} = 1$.
\end{itemize}
\end{definition}


Let us start with showing $b_{i, t} = 1$
for all $i$ and $t$ with high probability:

\begin{corollary}Consider the family of random variables $\{b_{i, t}: i \in [m], t \leq n^{2.5}\}$ defined above.
We have
$$\Pr\left[\bigwedge_{i \in [m], t \leq n^{2.5}}(b_{i, t}=1)\right] \geq 1 - \exp(-4\log n \log \log n).$$
\end{corollary}
\begin{proof} We first bound $\Pr[b_{i, t} = 1]$ for any specific $i$ and $t$.
Since the agents are placed according to stationary distribution at $t = 0$, each agent is still distributed
uniformly at time $t$. Next, at time $t$, we may relabel the agents so that $\mathrm a_i$ is regarded as the single initially infected agent in Proposition~\ref{lem:localbehavefinal}, where $\mathcal W$ is set as $\overline{\mathfrak B}$. We therefore have
$\Pr[b_{i, t} = 1] \geq 1 - \exp(5\log n \log \log n)$.

Next, we may apply a union bound across all $i$ and $t$ to get the desired result.
\end{proof}

\begin{lemma}\label{lem:nonempty} Let \xiaorui{$\mathfrak B = \mathbb B(P,
  n/8)$}. Let $B_t$ be the indicator variable that
  there is at least one agent in $\mathfrak B$ at time
  $t$. Let $B = \prod_{t \leq n^{2.5}}B_t$, the indicator variable that there is at least one agent in $\mathfrak B$ at all times in $[0,n^{2.5}]$. We have
$$\Pr[B = 0] \leq \exp(-\log^2n)$$
for sufficiently large $n$.
\end{lemma}
\begin{proof} First, notice that for any specific $t$, the expected
  number of agents in $\mathfrak B$ is $\Omega(m)$. Therefore, by
  Chernoff bound (using the version in Theorem~\ref{thm:chernoff})
  \liusd{$\Pr[B_t = 0] \leq \exp(-\Omega(m)) \leq \exp(-\log^3n)$}. Next, by a union bound, we have
  $\Pr[B = 0] \leq n^{2.5}\exp(-\log^3n) \leq \exp(-\log^2n)$.
\end{proof}

We next present our major lemma for this subsection.

\begin{lemma}\label{lem:combin}
Let $\mathrm a_1, ..., \mathrm a_m$ be placed uniformly at random
on $\mathcal V^3$ such that $m \geq 1600 n \log^2n$. Let $\ell_2 = \sqrt{n^3/m}$.
Let $\{b_{i, t}: i \in [m], t \leq n^{2.5}\}$ and $B$ be the random variables described
above. If $b_{i, t} = 1$ for all $i, t$ and $B = 1$, then the diffusion time is at least $T_c = \frac 1 {81}\ell_2 n \log^{-29}n$.
\end{lemma}

Notice that by \xiaorui{Proposition}~\ref{lem:localbehavefinal} and Lemma~\ref{lem:nonempty},
$$
\begin{array}{l}
\Pr\left[\bigwedge_{i \leq m, t \leq n^{2.5}}\left(b_{i, t} = 1\right)\right] \geq 1 - \exp(-4\log n \log \log n). \\
\Pr[B = 1] \geq 1 - \exp(-\log^2n).
\end{array}
$$
Together with Lemma~\ref{lem:combin}, Theorem~\ref{thm:lb3} then follows.

\begin{proof}
Without loss of generality, we assume the $x$, $y$, and $z$ coordinates
of $S^1_0$ are all negative. We can always rotate the space
$\mathcal V^3$ at $t = 0$ correspondingly to ensure this assumption to hold.


We shall prove by contradiction. Consider two balls $\mathfrak B$ and $\overline{\mathfrak B}$ defined above.
Assume the diffusion time is less than $T_c$. First, because $B = 1$, a necessary
condition for the diffusion to complete is that an infected agent ever visits
the smaller ball $\mathfrak B$ at a time $T' \leq T_c$ (since otherwise the agents in $\mathfrak{B}$ would be uninfected all the time, including at $T_c$). We call this agent $\mathrm{a}_{i'}$. Next, for  \zliu{the} infection to get into $\mathfrak B$, it must happen that there is an infected agent that enters $\overline{\mathfrak B}$ from outside, whose infection trajectory eventually reaches $\mathrm{a}_{i'}$. We denote $T''$ to be the \emph{last} time that this happens, and the responsible agent to be $\mathrm{a}_{i''}$. We focus on the trajectory of infection that goes from $\mathrm{a}_{i''}$ to $\mathrm{a}_{i'}$ that lies completely inside $\overline{\mathfrak B}$ (which exists since $T''$ is the last time of entry).  \zliu{Note that we consider at most $\lceil T_c/\Delta t\rceil$ time increments of $\Delta t$. }
Now, since $b_{i,t}=1$ for all $i$ and $t$, by repeated use of triangle inequality, we get
$$\|S_{T'}^{i'}-S_{T''}^{i''}\|_\infty\leq9\ell_2\log n\left\lceil\frac{T_c}{\Delta t}\right\rceil\leq9\ell_2\log n\left(\frac{(1/81)\ell_2n\log^{-29}n}{\ell_2^2\log^{-28}n}+1\right)\leq\frac{n}{9}+9\ell_2\log n<\frac{n}{8}-1$$
On the other hand, the physical dimensions of $\mathfrak B$ and $\overline{\mathfrak B}$ give that
$$\|S_{T'}^{i'}-S_{T''}^{i''}\|_\infty\geq\frac{n}{8}-1$$
which gives a contradiction.

\end{proof}

\newcommand{\newxiaorui}[1]{ {{#1}}}
\section{Upper bound}\label{sec:ub3}

We now focus on an upper bound for the diffusion time. Our main
result is the following:

\begin{theorem}\label{thm:up3} Let $a_1$, \ldots, $a_m$ be placed
  uniformly at random on
 {$\cV^3$, where  \zliu{$n \leq m \leq n^3$}.}
Let $\hat \ell_2 = \sqrt{n^3 / m} \cdot \log n$. When $n$ is
sufficiently large, the diffusion time $T$ satisfies
$$\Pr[T \geq  \newxiaorui{128n\hat \ell_2\log^{47}n}] \leq \exp(-\frac 1 2 \log^2 n).$$
\end{theorem}

Note that this theorem shows that an upper bound of
$\tilde{O}(n\sqrt{n^3/m})$ holds for the diffusion time with high probability. Hence the upper and lower
bounds ``match'' up to logarithmic factors. We remark that the constant 47 in the exponent has not
been optimized.

The main goal of this section is to prove this theorem. Our proof strategy relies
on calculating the growth rate of the \emph{total} infected agents
evolving over time; such growth rate turns out to be best characterized as the increase/decrease in infected/uninfected agents relative to the size of the corresponding population. More precisely, we show that for a well-chosen
time increment, either the number of infected agents doubles or the
number of uninfected agents reduces by half with high probability.  The choice of time increment is complex,
depending on the analysis of the local interactions in small cubes and the global geometric arrangements of these cubes with respect to the distribution of infected agents.

\lam{
As with the lower bound proof, our technique for proving Theorem~\ref{thm:up3} is different from existing methods. Let us briefly describe them and explain the challenges in extending to higher dimensional cases; further details are left to Appendix~\ref{subsec:existupper}. Roughly, existing methods can be decomposed into two steps (see for example~\cite{PPPU11}): 1) In the first step, consider a small ball of length $r$ that contains the initially infected agent. One can see that for $d=2$, when number of agents in the ball is $\tilde{\Theta}(m(r/n)^2)$, within time increment $r^2$ the number of infections to agents initially in this ball is $\tilde{\Omega}(1)$ w.h.p.. 2) The 2nd step is to prove that for any ball that has $\tilde{\Omega}(1)$ infected agents at time $t$, its surrounding adjacent balls will also have $\tilde{\Omega}(1)$ infected agents by time $t+r^2$. From these two steps, one can recursively estimate the time to spread infection across the whole space $\mathcal V^2$ to be $n/r\times r^2=nr$ w.h.p.. In other words, at time $nr$ all the balls in $\mathcal V^2$ will have $\tilde{\Omega}(1)$ infected agents. Moreover, every agent in $\mathcal V^2$ is infected in the same order of time units, because $\tilde{\Omega}(1)$ is also the total number of agents in any ball under good density condition. Finally, it is then clear that a good choice of $r$ is then $n/\sqrt{m}$, which would give the optimal upper bound.

The critical difference in the analysis for $d>2$ lies primarily in the magnitude of the meeting probability of random walks. In the case of $d=2$, the meeting probability of two random walks at distance $r$ within time $r^2$ is $\tilde{\Theta}(1)$, whereas for $d>2$ the meeting probability is $\Theta(1/r^{d-2})$. For $d=2$, this means that it is easy, i.e. w.h.p., for infection to transmit from a ball with $\tilde{\Omega}(1)$ infected agents to an adjacent uninfected ball, so that the latter also has $\tilde{\Omega}(1)$ infected agents after a time increment of $r^2$. In the case $d>2$, however, $\Omega(r^{d-2})$ infected agents must be present in a ball to transmit virus effectively to its adjacent uninfected ball within $r^2$ time. Consequently, arguing for transmission across adjacent balls becomes problematic (more details are in Appendix~\ref{subsec:existupper}). In light of this, we take an alternate approach to analyze both the local interactions and the global distribution of infected agents. Instead of focusing on transmission from one infected ball to another, we calculate the spreading rate across the whole space. This turns out to be fruitful in obtaining a tight upper bound.

}

We briefly describe the forthcoming analysis.  As with the lower
bound, we start with local analysis. We partition the space
$\mathcal{V}^3$ into disjoint subcubes each of size $\hat \ell_2\times\hat
\ell_2\times\hat \ell_2$. Here $\hat \ell_2$ is just a
 \yajun{logarithmic}
factor larger than $\ell_2$, the size of subcubes used
for the lower bound, so that with overwhelming probability there are \emph{at least} $\hat \ell_2$ agents in a subcube. We show that, within every subcube, over a time
increment of length $\Theta(\hat \ell^2_2)$ the number of
infections is roughly a $\tilde \Omega(1)$-fraction of the
minimum of the number of infected and uninfected agents.  Hence,
at least locally, we have the desired behavior described above.

We then leverage the local analysis to obtain the global result.
However, this is not straightforward. For example, consider the
beginning when the number of infected agents is small. If infected
agents are distributed uniformly throughout the whole space, it would
be easy to show that new infections would roughly grow in proportion
to the number of infected agents.  However, if infected agents are
concentrated into a small number of subcubes, we have to show that
there are enough neighboring subcubes on the boundary of these infected
subcubes that these subcubes become infected suitably rapidly, so that
after the appropriate time increment the number of infected agents doubles. Similar  \zliu{arguments} arise for the case when infected agents are dominant, with the end result  \zliu{being a} halving  \zliu{of} the uninfected population.

We now make the above discussion rigorous. First, let $b = (2n + 1) / \hat \ell_2$, so there are in total $b^3$
subcubes. As in the previous section, we divide the time into small intervals. We reuse the symbol
$\Delta t$ to represent the length of each interval but here we set $\Delta t = 16\hat \ell^2_2$. Our local bound is built within each subcube (and pair of neighboring subcubes) in the time increment $\Delta t$:

\begin{lemma}\label{lem:dense}Let $\mathcal W \subset \mathcal V^3$ be a region that can be covered
  by a ball of radius $2\hat \ell_2$ under the $L_{\infty}$-norm. Let
  $A^f$ and $A^u$ be subsets of infected and uninfected agents in
  $\mathcal W$ at time $t$ such that $|A^f| = m_1$, $|A^u| = m_2$, and
  $\max\{m_1, m_2\} = \hat \ell_2 / \log^{2}n$.  Given any initial placement of the agents of $A^f$ and
  $A^u$, let $M(t)$ be the
  number of agents in $A^u$ that become infected at time $t + \Delta
  t$. We have
$$\Pr\left[M(t) \geq \frac{\tau_0\min\{m_1, m_2\}}{\log^{4}n}\Big|\mathcal F_t\right]
  \geq \tau_0\log^{-6}n.$$
for some constant $\tau_0$, where $\mathcal F_t$ denotes the information of the whole diffusion process up to time $t$.
\end{lemma}
%

\begin{proof}
The high level idea of our proof is to count the total
  number of times the infected agents meet the uninfected
  ones between time $t$ and $t + \Delta t$.
The probability two agents in $\mathcal W$ can meet each other within
  time $\Delta t$ is approximately $\tilde \Omega(1/\hat \ell_2)$
  (Lemma~\ref{lem:couple}).
The expected number of meetings is thus $\tilde \Omega(1/\hat \ell_2) \times m_1m_2 = \tilde
  \Omega(\min\{m_1, m_2\})$.  The total number of newly infected
  agents is the number of
meetings modulo possible overcounts on
  each originally uninfected agent.
If we can show that the number of meetings
is $\tilde O(1)$ for each uninfected agent, then we can conclude that $\tilde \Omega(\min\{m_1, m_2\})$ more agents become
  infected at time $t + \Delta t$.

Two problems need to be addressed to implement this idea. First, when the agents are close to the boundary, they may behave in a
more complicated way than suggested by Lemma~\ref{lem:couple}.
Second, in the (rare) case that an uninfected agent is surrounded by a
large number of infected agents, it can possibly
meet with $\omega(1)$ infected agents, making it difficult to give an upper bound over the number of overcounts.

To address both problems, we wait $\hat \ell^2_2$ time steps before starting our analysis on infections. This time gap is enough to guarantee that with constant probability, the agents are ``locally shuffled'' so that by time
$t + \hat \ell^2_2$,
\begin{enumerate*}
\item all agents are reasonably far away from the boundaries,
\item the distance between any pair of agents is ``appropriate'' (in our case, the distance is between $\hat \ell_2$ and $9\hat \ell_2$).
\end{enumerate*}

Intuitively, the ``local shuffling" works because central limit theorem implies that the agents' distribution at the end of these steps is approximately multivariate Gaussian.

We now implement this idea.
First, we couple the (sub)process in
$\mathcal W$ with one that has slower diffusion rule.  In the coupled
process, we first wait for $\hat \ell^2_2$ time steps, in which no
agent becomes infected even if it meets an infected agent.
After these $\hat \ell^2_2$ steps, for an arbitrary $\mathrm a_i \in
A^f$ and $\mathrm a_j \in A^u$, let $X_{i, j} = 1$ if both of the
following conditions  {hold},
\begin{itemize*}
\item the $L_1$-distance between $\mathrm a_i$ and $\mathrm a_j$ is between
  $\hat \ell_2$ and ${9}\hat \ell_2$.
\item the $L_1$-distance between $\mathrm a_i$ and any boundary is at least $ \zliu{360}\hat \ell_2$.
\end{itemize*}
By  \zliu{Corollary~\ref{lem:waitboundary}}, $\Pr[X_{i, j} = 1] \geq
\tau$ for some constant $\tau$.
Therefore, $\E[\sum_{\mathrm a_i\in A^f, \mathrm a_j\in A^u}X_{i, j}]
\geq \tau m_1m_2$. On the other hand, $\sum_{i, j} X_{i,j}\leq m_1m_2$.
It follows easily that we have $\Pr[\sum_{i, j}X_{i,j}\geq \frac 1 2\tau m_1m_2] \geq \tau / 2$.

Our slower diffusion rule then allows $\mathrm a_i \in A^f$ to transmit its
virus to $\mathrm a_j \in A^u$ if and only if
\begin{itemize*}
\item $X_{i,j} = 1$,
\item they meet during the time interval $(t+\hell_2^2, t+\Delta t]$,
\item $\mathrm a_i$ and $\mathrm a_j$ have not visited any boundary
  after $t+\hell_2^2$
before they meet. In other words, an agent $\mathrm a_i \in A^f$ ($\mathrm a_j \in A^u$ resp.) loses its ability to transmit (receive resp.) the virus when it hits
a boundary
after the initial waiting stage.
\end{itemize*}

An added rule is that agents in $A^u$ will not have
the
ability to transmit
the
 virus even after they are infected.

 Let $Y_{i, j}$ be
the indicator random variable that is set to $1$ if and only if $\mathrm a_i$ transmits its
virus to $\mathrm a_j$
under the slower diffusion rule.
By Lemma~\ref{lem:twowalkboundary}, we have $\Pr[Y_{i,j} = 1\mid X_{i, j} =
1] = \Omega(1/\hat\ell_2)$. Therefore, we
have
$$\Pr[Y_{i, j} = 1] \geq \Pr[Y_{i, j} = 1\mid X_{i, j} = 1]\Pr[X_{i,j} = 1] = \Omega(1/\hat{\ell}_2).$$
Hence,
$$\E[\sum_{\mathrm a_i \in A^f, \mathrm a_j \in A^u}Y_{i,j}] =
\Omega(m_1m_2/\hat \ell_2) \newxiaorui{\geq} \tau_1m_1m_2/\hat\ell_2$$
for some constant $\tau_1$.
 $\sum_{\mathrm a_i \in A^f, \mathrm a_j \in A^u}Y_{i,j}$ is approximately
the number of newly infected agents except that the same
agent  in $A^u$ may be counted multiple times.
Our next task is thus to give an upper bound on the number of overcounts. Specifically,
  we fix an agent $\mathrm a_j \in A^u$ and argue that
the probability $\sum_{\mathrm a_i \in A^f}Y_{i,j} \geq \log^2n$ is small.

In our slower diffusion model, once an agent in $A^f$ reaches the
boundary, it is not able to transmit the virus further.
We need to bound the probability that there are more than $\log^2n$
agents in ${A^f}$ that transmit the virus to $\mathrm a_j$ {\em before they
hit any boundary}. This probability is at most the probability that
more than $\log^2n$ infected agents performing \emph{unbounded} random walks meet $a_j$,
where each infected agent is at least $\hat \ell_2$ away from $a_j$
initially.

By Lemma~\ref{lem:rwcatchall}, there exists a constant $c_0$
such that for all possible values of  $X_{i,j}$ and sufficiently large $n$:
\begin{eqnarray*}
\Pr[\sum_{\mathrm a_i \in A^f}Y_{i,j} \geq \log^2n \mid X_{1, j}, X_{2, j}, ..., X_{m_1, j}] & \leq & \binom{\sum_{i \leq m_1}X_{i, j}}{\log^2n}\left(\frac{c_0\log^2n}{\hat \ell_2}\right)^{\log^2n} \\
& \leq & \binom{m_1}{\log^2n}\left(\frac{c_0\log^2n}{\hat \ell_2}\right)^{\log^2n} \\
& \leq & \left(\frac{em_1}{\log^2n}\right)^{\log^2n}\left(\frac{c_0\log^2n}{\hat \ell_2}\right)^{\log^2n}  \\
& = & \left(\frac{c_0em_1}{\hat \ell_2}\right)^{\log^2n} \\
& \leq & \exp(-\log^{2}n\log \log n).
\end{eqnarray*}
Therefore,
$$\Pr[\sum_{\mathrm a_i \in A^f}Y_{i,j} \geq \log^2n] =
\E[\Pr[\sum_{\mathrm a_i \in A^f}Y_{i,j} \geq \log^2n \mid  X_{1, j}, ..., X_{m_1, j}] ]
\leq \exp(-\log^{2}n \log \log n).$$
 By a union bound we have
\begin{equation}\label{eqn:sumlb}
\Pr[\exists j: \sum_{\mathrm a_i \in A^f}Y_{i,j} \geq \log^2n  ] \leq m\cdot\exp(-\log^{2}n \log \log n) \leq \exp(-2\log^2n).
\end{equation}
Next, let us fix $\mathrm a_i \in A^f$ and we may argue in a similar way to obtain
$$\Pr[\exists i: \sum_{\mathrm a_j \in A^u}Y_{i,j} \geq \log^2n  ] \leq \exp(-2\log^2n).$$

Define $e_t$ as the event that $\left(\forall i,\, \sum_{a_j\in A^u}Y_{i,j}
\leq \log^2n\right) \land \left(\forall j, \, \sum_{a_i \in A^f} Y_{i,j} \leq
\log^2n\right)$. Therefore, $\pr[e_t] \geq 1-2\exp(-2\log^2n)$. Observe that $e_t$ implies $\sum_{i, j}Y_{i, j} \leq \min\{m_1, m_2\}\log^2n$. We have
\begin{eqnarray*}
\tau_1m_1m_2/\hat \ell_2 & \leq & \E[\sum_{i, j}Y_{i, j}] \\
& = & \E[\sum_{i, j}Y_{i, j}|e_t]\Pr[e_t] + \E[\sum_{i, j}Y_{i,
  j}|\lnot e_t]\Pr[\lnot e_t] \\
& \leq & \E[\sum_{i, j}Y_{i, j}|e_t] + m^2\Pr[\lnot e_t] \\
& \leq & \E[\sum_{i, j}Y_{i, j}|e_t] + 2m^2\exp(-2\log^2n).
\end{eqnarray*}
Therefore,
\begin{eqnarray*}
\E[\sum_{i, j}Y_{i, j}|e_t] &\geq& \tau_1m_1m_2/\hat\ell_2 -2m^2\exp(-2\log^2n) \\
& \geq & \frac{\tau_1 m_1m_2}{2\hat\ell_2}\\
& = &\frac{\tau_1 \max\{m_1, m_2\}\min\{m_1, m_2\}}{2\hat\ell_2} \\
 & \geq & \frac{\tau_1\min\{m_1, m_2\}}{2\log^2n}
\end{eqnarray*}

Next, define indicator variable $I_j=1$  \liusd{if and only if
$\sum_{\mathrm a_i \in A^f}Y_{i,j} > 0$.} The sum $\sum_j I_j$
is the total number of newly infected agents in our weaker process and thus
is a lower bound on $M(t)$. Note that if $e_t$ holds,
$$\sum_{\mathrm a_j \in A^u} I_j \leq \sum_{i, j}Y_{i, j} \leq \min\{m_1, m_2\}\log^2n,$$
and hence
$$E[\sum_{\mathrm a_j \in A^u} I_j|e_t]\leq\min\{m_1, m_2\}\log^2n.$$
On the other hand,
\begin{equation}
\E[\sum_j I_j|e_t] \geq \log^{-2}n\E[\sum_{i, j}Y_{i, j}|e_t] \geq \frac{\tau_1\min\{m_1, m_2\}}{2\log^4n}
\end{equation}
Now define $\tilde{m} = \frac{\tau_1\min\{m_1,m_2\}}{4\log^4n}$. We have
\begin{eqnarray*}
2\tilde{m}
  &\leq & \E[\sum_j I_j |e_t] \\
 & = & \E[\sum_j I_j \Big| e_t, \sum_j I_j \leq \tilde{m} ]\Pr[\sum_j I_j \leq \tilde{m}\Big| e_t] + \E[\sum_j I_j \Big| e_t, \sum_j I_j > \tilde{m} ]\Pr[\sum_j I_j > \tilde{m}\Big| e_t] \\
& \leq & \E[\sum_j I_j \Big| e_t, \sum_j I_j \leq \tilde{m} ] + \min\{m_1, m_2\}\log^2n \Pr[\sum_j I_j > \tilde{m}\Big| e_t] \\
& \leq & \tilde{m} + \min\{m_1, m_2\}\log^2n \Pr[\sum_j I_j > \tilde{m}\Big| e_t].
\end{eqnarray*}

Therefore,
$$\Pr\left[\sum_j I_j > \frac{\tau_1\min\{m_1, m_2\}}{4\log^4n}\Big| e_t\right] \geq  \frac{\tau_1}{4\log^6n}.$$
Finally,
\begin{eqnarray*}
\Pr\left[\sum_j I_j > \frac{\tau_1\min\{m_1, m_2\}}{4\log^4n}\right] &\geq &\Pr\left[\sum_j I_j > \frac{\tau_1\min\{m_1, m_2\}}{4\log^4n}\Big| e_t\right]\Pr[e_t] \\
& \geq & \frac{\tau_1}{4\log^6n}(1-2\exp(-\log^2n)) \geq
\frac{\tau_1}{5\log^6n}
\end{eqnarray*}
By setting $\tau_0 = \tau_1 / 5$, we get our result.

\end{proof}

The next step is to characterize the growth rate at a larger scale. This requires more  \zliu{notation}. We denote the set of $b^3$ subcubes of size $\hat{\ell}_2\times\hat{\ell}_2\times\hat{\ell}_2$ as $\mathfrak C = \{h_{i, j, k}:
i, j, k \in [b]\}$.  For an arbitrary subcube $h_{i, j, k}$, we define
its neighbors as
$N(h_{i, j, k}) = \{h_{i',j',k'}: |i - i'| + |j - j'| + |k-k'| = 1\}$.
In other words, $h_{i', j',k'}$ is a neighbor of $h_{i, j, k}$ if and
only if both subcubes share a facet.  Let $\mathcal H$ be an
arbitrary subset of $\mathfrak C$. We write $N(\mathcal H) =
\bigcup_{h \in \mathcal H}N(h)$.

\begin{definition}[Exterior and interior surface] Let $\mathcal H$ be
  a subset of $\mathfrak C$. The \emph{exterior surface} of $\mathcal
  H$ is $\exsur \mathcal H = N(\mathcal H) - \mathcal H$. Let
  $\overline{\mathcal H}$ be the complement of $\mathcal{H}$. The \emph{interior surface} of $\mathcal H$ is
$\insur \mathcal H = N( \overline{\mathcal H}) -
\overline{\mathcal H}$, i.e., the exterior surface of the
  complement of $\mathcal H$.
\end{definition}


At time step $t= i\Delta t$, let $\mathcal G_t$ be the set of all subcubes that
contain more than $\hat \ell_2 / 2$ infected agents and let $g_t =
|\mathcal G_t|$; let $\mathcal B_t = \overline{\mathcal G_t}$ be the
rest of the subcubes and let $b_t = |\mathcal B_t|$. We say a subcube
in $\mathcal G_t$ an \emph{infected} (\emph{good}) subcube and a subcube in $\mathcal
B_t$ an \emph{uninfected} (\emph{bad}) subcube.

We classify the agents in the process according to the subcubes they
reside in. To facilitate our analysis, we adopt the notational system
$\mfa_t^\cdot$ and $\mfa_t^{\cdot,\cdot}$ to represent the total
number of agents that belong to the type specified in the
 {superscript.}
Specifically, let $\mfa^f_t$ be the set of
infected agents at time $t$; decompose
the set $\mfa^f_t$ as $\mfa^f_t = \mfa^{f, \cG}_t \cup \mfa^{f, \cB}_t$, where $\mfa^{f, \cG}_t$ is the set of infected agents residing in
the subcubes in $\mathcal G_t$ and $\mfa^{f,\cB}_t$ the set of infected
agents in $\mathcal B_t$. Similarly, let
$\mfa^{u}_t$ be the set of all uninfected agents;
decompose the set $\mfa^{u}_t$ as $\mfa^{u}_t = \mfa^{u, \cG}_t \cup
\mfa^{u, \cB}_t$, where $\mfa^{u, \cG}_t$ is the set of uninfected agents
residing in the subcubes in $\mathcal G_t$ and $\mfa^{u, \cB}_t$ the
set of uninfected agents in $\mathcal B_t$. Furthermore, we denote $\Delta \mfa^{\cG}_t$ and $\Delta \mfa^{\cB}_t$ as the set of
agents in $\mathcal G_t$ and $\mathcal B_t$ respectively that are infected
between $t$ and $t + \Delta t$. Hence the total increase in infected agents, or equivalently the total decrease in uninfected agents, between $t$ and $t + \Delta t$ is given by $\Delta \mfa_t = \Delta \mfa^{\cG}_t \cup
\Delta \mfa^{\cB}_t$. Lastly, we let $\widetilde{\Delta\mfa^{\cG}_t}$ be the set of agents in $\mathcal G_t \cup \partial \mathcal G_t$ that are infected
between $t$ and $t + \Delta t$.

Similar to the lower bound analysis, here we also introduce good density conditions that can be easily verified to hold with high probability, and reuse the symbols $D_t$ and $D$ with slightly different
meanings from the last section:

\begin{definition}\label{def:denseup}Let $\{D_t: t \geq 0\}$ be a sequence of binary random
  variables such that $D_t = 1$ if for all
  time steps on or before $t$, the
  number of agents for any subcube in $\mathcal V^3$
  with size $\hat \ell_2 \times \hat \ell_2 \times \hat \ell_2$ is between $\hat \ell_2$ and
  $2\hat \ell_2\log^2n$. Also, let $D = D_{n^{2.5}}$.
\end{definition}

The following lemma shows that $D_t=1$ with high probability, whose proof will be left to Appendix~\ref{sec:missingproofs}:

\begin{lemma}
\label{lem:densitya3d}
For any $t \leq n^{2.5}$, $\Pr[D_t = 0] \leq \exp(-\frac 1{15}\log^2n)$ for sufficiently large $n$.
\end{lemma}

We now state two bounds on the growth rate of the agent types, one relative to the ``boundary subcubes" $\partial \mathcal G_t$ and one relative to the total agents of each type:

\begin{corollary}\label{cor:dense}For some constant $\tau_0$,
{\small
$$\Pr\left[|\widetilde{\Delta \mfa^{\cG}_t} \cap \Delta \mfa^{ \cB}_t| \geq |\partial\mathcal G_t|\cdot \frac{\tau_0\hat \ell_2}{4\log^{13}n}\Bigg|\mathcal F_t, D_t = 1\right]  \geq \tau_0\log^{-6}n.$$
}
%
Consequently,
{\small
$$\Pr\left[|\widetilde{\Delta \mfa^{ \cG}_t}| \geq |\partial\mathcal G_t|\cdot \frac{\tau_0\hat \ell_2}{4\log^{13}n}\Bigg|\mathcal F_t, D_t = 1\right]   \geq \tau_0\log^{-6}n \mbox{ and } \Pr\left[| \Delta \mfa^{\cB}_t| \geq |\partial\mathcal G_t|\cdot \frac{\tau_0\hat \ell_2}{4\log^{13}n}\Bigg|\mathcal F_t, D_t = 1\right]  \geq \tau_0\log^{-6}n.$$}
\end{corollary}

\begin{corollary}\label{lem:sparse}  We have
{\small
$$\Pr\left[|\Delta \mfa^{\cG}_t|\geq \frac{\tau^2_0}{4\log^{38}n}|\mfa^{u, \cG}_t| \Big|
\mathcal F_t, D_t = 1\right] \geq {\tau_0} \log^{-6}n \mbox{ and }\Pr\left[|\Delta \mfa^{\cB}_t|\geq \frac{\tau^2_0}{4\log^{38}n}|\mfa^{f,\cB}_t| \Big|
\mathcal F_t, D_t = 1\right] \geq {\tau_0} \log^{-6}n.$$}
\end{corollary}

The proofs of these two corollaries both rely on using coupled diffusion processes that have slower diffusion rates. These processes only allow infection locally i.e. within each ``pair" of subcubes on the surface of $\mathcal G_t$ in the case of Corollary~\ref{cor:dense} and within each subcube in Corollary~\ref{lem:sparse}, and hence can be tackled by Lemma~\ref{lem:dense}. The surface $\partial \mathcal G_t$ in Corollary~\ref{cor:dense} appears naturally from a matching argument between neighboring infected and uninfected subcubes. Roughly speaking, the bounds in Corollary~\ref{cor:dense} are tighter and hence more useful for the cases where infected/uninfected agents are dense in the infected/uninfected subcubes, while those in Corollary~\ref{lem:sparse} are for cases where the agent types are more uniformly distributed.

\begin{proof}[Proof of Corollary~\ref{cor:dense}] Agents in $\widetilde{\Delta \mfa^{\cG}_t} \cap \Delta \mfa^{\cB}_t$ are those initially in
$\partial \mathcal G_t$ at $t$ and become infected at the time $t + \Delta t$. We focus on how the uninfected agents in
$\partial \mathcal G_t$ become infected.

Let us construct a graph $G = (V, E)$, in which the vertex set $V$ of $G$ consists of subcubes in $\exsur \mathcal G_t\cup \insur \mathcal G_t$ and the edge set is defined as
$$E = \{\{u, v\}: u \in \partial \mathcal G_t, v \in \dot \partial \mathcal G_t, u \in N(v)\}.$$


We may use a greedy algorithm to argue that there  is a matching on $G$ from $\exsur \mathcal
G_t$ to $\insur G_t$ with size at least $|\partial \mathcal G_t| /
11$ (see Lemma~\ref{lem:matchingbnd} for details). Denote the matching as
$$\mathfrak M = \{\{h_1, h'_1\}, \{h_2, h'_2\}, ..., \{h_k, h'_k\}: h_i \in \exsur \mathcal G_t, h'_i = \insur \mathcal G_t\},$$
where $k \geq |\exsur \mathcal G_t| / 11$.
We next define a coupling process with a slower diffusion rule:
an infected agent can transmit virus to an uninfected one if and only if at
time $t$ the infected agent is in $h'_j$ and the uninfected one is in $h_j$ for some $j$.
Let $\rho_j$ be the number of uninfected agents initially in $h_j$
at time $t$ that become
infected by time $t + \Delta t$ under the slower diffusion rule. We have $\sum_{j \leq k}\rho_j$ at most $|\widetilde{\Delta \mfa^{\cG}_t} \cap \Delta \mfa^{\cB}_t|$ in
the original process. We design the coupling in this way because $\rho_j$ {s}
are independent of each other as they are decided by independent walks from disjoint pairs of
subcubes.

Next we apply Lemma~\ref{lem:dense} on each pair of the matching. Fix an arbitrary matched pair $\{h_j, h'_j\}$. Since $h_j \in \exsur \mathcal G_t$, at time $t$ there are at least $\hat \ell_2/2$ uninfected agents in $h_j$; similarly, since $h'_j \in \insur \mathcal G_t$, there are at least $\hat \ell_2 / 2$ infected agents in $h'_j$.
At time $t$ we can find a subset of uninfected agent $A^u$ in $h'_j$
and a subset of infected agents $A^f$ in $h_j$ such that $|A^u| =
|A^f| = \hat \ell_2 / \log^{2}n$. Therefore, by Lemma~\ref{lem:dense}, we have
\begin{equation}\label{eqn:probincrease}
 \Pr[\rho_j \geq \frac{\tau_0\hat
  \ell_2}{\log^{4}n}\mid \mathcal F_t, D_t = 1] \geq \tau_0\log^{-6}n
\end{equation}
for some constant $\tau_0$. From Equation~\ref{eqn:probincrease}, we can see
that
$\E[\rho_j \mid \mathcal F_t, D_t = 1] \geq \tau^2_0\hat\ell_2 \log^{-10}n$. Therefore,
\begin{equation}\label{eqn:explow}
\E[\sum_{j \leq k}\rho_j \mid \mathcal F_t, D_t = 1] \geq
\frac{\tau^2_0}{11}|\partial \mathcal G_t|\hat \ell_2\log^{-10}n.
\end{equation}
Next, we consider two cases.

{\noindent \emph{Case 1.}},  $|\partial \mathcal G_t| \leq \log^{9}n$.
In this case
\begin{eqnarray*}
& & \Pr\left[|\widetilde{\Delta \mfa^{ \cG}_t} \cap \Delta \mfa^{ \cB}_t| \geq
  |\partial\mathcal G_t|\cdot \frac{\tau_0\hat
    \ell_2}{4\log^{13}n}\Big|\mathcal F_t, D_t = 1\right] \\
& \geq & \Pr\left[|\widetilde{\Delta \mfa^{\cG}_t} \cap \Delta \mfa^{\cB}_t| \geq \log^9n\frac{\tau_0\hat \ell_2}{4\log^{13}n}\Big|\mathcal F_t, D_t = 1\right]\\
& \geq & \Pr\left[|\widetilde{\Delta \mfa^{\cG}_t} \cap \Delta \mfa^{\cB}_t| \geq \frac{\tau_0\hat \ell_2}{4\log^{4}n}\Big|\mathcal F_t, D_t = 1\right]\\
& \geq & \Pr[\rho_1 \geq \tau_0 \hat \ell_2\log^{-4}n \mid \mathcal F_t, D_t = 1] \quad \mbox{(only focus on an arbitrary matched pair in the matching)}\\
& \geq & \tau_0\log^{-6}n. \quad \mbox{(Lemma~\ref{lem:dense})}
\end{eqnarray*}
\emph{Case 2.} $|\partial \mathcal G_t| > \log^9n$. Notice that $\rho_1, ..., \rho_k$ are
independent by construction. Also, we have
$$|\partial\mathcal G_t|\cdot \frac{\tau_0\hat \ell_2}{4\log^{13}n}
\leq  \frac{\tau^2_0}{22}|\partial \mathcal G_t|\hat \ell_2\log^{-10}n
\leq \frac 1 2\E\left[|\widetilde{\Delta \mfa^{\cG}_t} \cap \Delta \mfa^{\cB}_t|
  \Big|\mathcal F_t, D_t = 1\right].$$
 By a Chernoff bound (the version we use is Theorem~\ref{thm:chernoff}
 with $\delta=1/2$), we have

\begin{eqnarray*}
& & \Pr\left[|\widetilde{\Delta \mfa^{\cG}_t} \cap \Delta \mfa^{ \cB}_t| \leq |\partial\mathcal G_t|\cdot \frac{\tau_0\hat \ell_2}{4\log^{13}n}\Big|\mathcal F_t, D_t = 1\right] \\
& \leq & 2\exp\left(-\frac{(\frac 1 2)^2\E\left[|\widetilde{\Delta \mfa^{\cG}_t} \cap \Delta \mfa^{\cB}_t| \Big|\mathcal F_t, D_t = 1\right]}3\right) \\
& \leq & 2\exp\left(-\frac{\tau^2_0}{132}|\partial \mathcal
  G_t|\hat \ell_2\log^{-10}n\right) \quad \mbox{ (using Equation~\ref{eqn:explow})}\\
& \leq & 1- \tau_0\log^{-6}n
\end{eqnarray*}
for sufficiently large $n$. Our corollary thus follows.
\end{proof}

%
%
%

\begin{proof}[Proof of Corollary~\ref{lem:sparse}]
  Let us start with proving the first inequality.  We first couple the
  diffusion problem with a slower diffusion process defined as follows. First, all the
  infected agents in $\mathcal B_t$ at time $t$ cannot transmit the
  virus. Second, agents in $\cG_t$ are able to transmit the virus to
  each other if and only if at time $t$ they are in the same $g$ for
  some $g \in \cG_t$.  For an arbitrary $g \in \cG_t$,
  we let $\mfa^{u, \cG}_{t,g}$ be the set of uninfected agents in $g$ at
  time $t$. Accordingly, let $\Delta \mfa^{\cG}_{t,g}$ be the set of
  agents in $\mfa^{u, \cG}_g$ that become infected at $t + \Delta t$
  under the slower coupled process.
By Lemma~\ref{lem:dense}, we have
\begin{equation}\label{eqn:bnds}
\Pr\left[|\Delta \mfa^{\cG}_{t,g}| \geq \frac{\tau_0\min\{\hat\ell_2 / (\log^2n), |\mfa^{u, \cG}_{t,g}|\}}{4\log^4n} \Bigg|\mathcal F_t, D_t = 1\right] \geq \tau_0\log^{-6}n.
\end{equation}
Note that Lemma~\ref{lem:dense} requires that both the number of infected agents
and the number of uninfected agents are at most
$\hat\ell_2/(\log^2n)$. By $\mathcal G_t$'s construction, there are at least $\hat \ell_2/2$ infected agents in each subcube, and we may choose an arbitrary subset of them with size $\hat \ell_2/\log^2n$ to form $A^f$ in Lemma~\ref{lem:dense}. We do not know the exact size of $|\mfa^{u, \cG}_{t, g}|$ but in case $|\mfa^{u, \cG}_{t, g}| > \hat \ell_2/\log^2n$, we let $A^f$ be an arbitrary subset of $\mfa^{u,\cG}_{t,g}$ with size $\hat \ell_2/\log^2n$.

  When $D_t = 1$, we have
$|\mfa^{u, \cG}_{t,g}| \leq 2 \hat \ell_2\log^2n$. Therefore,
$$\min\{\hat \ell_2/\log^2n, |\mfa^{u, \cG}_{t,g}|\} \geq \min\{
\frac{|\mfa^{u, \cG}_{t,g}|}{2\log^4n}, |\mfa^{u, \cG}_{t,g}|\}
\geq |\mfa^{u, \cG}_{t,g}|/(2\log^4n).$$

Equation~\ref{eqn:bnds} can be rewritten as
$$\Pr\left[|\Delta \mfa^{\cG}_{t,g}| \geq \frac{\tau_0|\mfa^{u, \cG}_{t,g}|}{8\log^8n}
  \Bigg|\mathcal F_t, D_t = 1\right] \geq \tau_0 \log^{-6}n.$$
This also provides a lower bound over the expectation
of $|\Delta \mfa^{\cG}_{t,g}|$, i.e.,
$$\E[|\Delta \mfa^{\cG}_{t,g}|\mid \mathcal F_t, D_t = 1] \geq
\frac{\tau^2_0}{8}|\mfa^{u, \cG}_{t,g}|\log^{-14}n.$$
We also have
\begin{equation}\label{eqn:expbnd}
\E[|\Delta \mfa^{\cG}_t| \mid \mathcal F_t, D_t = 1] =
\sum_{g}\E[|\Delta \mfa^{\cG}_{t,g}|\mid \mathcal F_t, D_t = 1] \geq
\sum_g\frac{\tau^2_0}{8}|\mfa^{u, \cG}_{t,g}|\log^{-14}n =
\frac{\tau^2_0}{8}|\mfa^{u, \cG}_{t}|\log^{-14}n.
\end{equation}
Furthermore, by the way we design the coupled process, the random variables $|\Delta \mfa^{\cG}_{t,g}|$ are independent given $\mathcal F_t, D_t = 1$.

We consider two cases.

\emph{Case 1.} There exists $g \in \mathcal G_t$ such that
$|\mfa^{u, \cG}_{t,g}| \geq |\mfa^{u, \cG}_t| / \log^{29}n$. In this case, we have
$$\Pr\left[|\Delta \mfa^{\cG}_t|\geq \frac{\tau_0|\xiaorui{\mfa^{u, \cG}_t}|}{4\log^{38}n}\Bigg|\mathcal F_t, D_t = 1\right] \geq
\Pr\left[|\Delta \mfa^{\cG}_{t,g}| \geq
  \frac{\tau_0|\mfa^{u, \cG}_{t,g}|}{4\log^{9}n}\Bigg| \mathcal F_t, D_t =
  1\right] \geq \tau_0 \log^{-6}n.$$

\emph{Case 2.}
For all $g \in \mathcal G_t$, $|\mfa^{u, \cG}_{t,g}|< |\mfa^{u, \cG}_t|/\log^{29}n$.
Observe, on the other hand, that $\sum_g|
 \mfa^{u, \cG}_{t, g}| = |\mfa^{u, \cG}_t|$.
In this case, we have the summation
$\sum_{g}|\mfa^{u, \cG}_{t,g}|^2 <|\mfa^{u, \cG}_t|^2\log^{-29}n$.
 (and it is maximized
 when every non-zero $| \mfa^{u, \cG}_{t,g}|$ is exactly $|\mfa^{u, \cG}_t|\log^{-29}n$). We therefore have
\begin{equation}\label{eqn:sumbnd}
\sum_{g \in \mathcal G_t}|\mfa^{u, \cG}_{t,g}|^2 \leq |\mfa^{u, \cG}_t|\log^{-29}n
\sum_{g\in \cG_t}|\mfa^{u, \cG}_{t,g}| = |\mfa^{u, \cG}_t|^2\log^{-29}n.
\end{equation}

Next, by Hoeffding's inequality (See, e.g., Theorem~\ref{thm:hoeffding}),
we have
\begin{eqnarray*}
  & &\Pr\left[|\Delta \mfa^{\cG}_t| \leq \frac{\tau_0|\mfa^{u, \cG}_t|}{4\log^{38}n}\Big|\mathcal F_t, D_t = 1\right] \\& \leq&  \Pr\left[|\Delta \mfa^{\cG}_t| \leq \frac{\E[|\Delta \mfa^{ \cG}_t|\mid \mathcal F_t, D_T = 1]}{2}\Big|\mathcal F_t, D_t = 1\right]\\
  & = & \Pr\left[\sum_g|\Delta \mfa^{\cG}_{t,g}| \leq \frac{\E[|\Delta \mfa^{\cG}_t|\mid \mathcal F_t, D_T = 1]}{2}\Big|\mathcal F_t, D_t = 1\right]\\
  & \leq & 2\exp\left(-\frac{2(\frac 1 2\E[|\Delta \mfa^{ \cG}_t|\mid \mathcal
      F_t, D_t = 1])^2}{\sum_{g \in \mathcal G_t}|\mfa^{u, \cG}_{t,g}|^2}\right) \quad
  \mbox{(apply Hoeffding's inequality; we have $|\Delta \mfa^{ \cG}_{t,g}| \leq
    |\mfa^{u, \cG}_{t,g}|$)}\\
  & \leq & 2 \exp\left(-\frac{2\frac{\tau^2_0}{16^2}|\mfa^{u, \cG}_t|^2\log^{-28}n}{|\mfa^{u, \cG}_t|^2\log^{-29}n}\right)\quad \mbox{(by Equation~\ref{eqn:expbnd} and Equation~\ref{eqn:sumbnd})} \\
  & = & \exp(-\Theta(\log n)) \\
  & \leq & 1 - \tau_0\log^{-6}n,
\end{eqnarray*}
for sufficiently large $n$.

Proving the inequality regarding $|\Delta\mfa^{\cB}_t|$ is similar. We  \zliu{provide it} here for completeness, but less patient readers may simply skip this part. We
first couple the diffusion problem with a slower process. First, all
the infected agents in $\cG_t$ at time $t$ cannot transmit
virus. Second, agents in $\mathcal B_t$ are able to transmit virus to
each other if and only if at time $t$ they are in the same $b$ for
some $b \in \mathcal B_t$.  For an arbitrary $b \in \mathcal B_t$, we
let $\mfa^{f,\cB}_{t,b}$ be the set of uninfected agents in $b$ at time
$t$. Accordingly, let $\Delta \mfa^{\cB}_{t,b}$ be the set of agents
in $\mfa^{f,\cB}_{t,b}$ that becomes infected at $t + \Delta t$ under
the coupled process. For technical reasons, we  require the slower
diffusion in the subcube $b$ to halt when $|\Delta \mfa^{\cB}_{t, b}|$ becomes large,
i.e., $|\Delta \mfa^{\cB}_{t,b}| = |\mfa^{f, \cB}_{t,b}|$. This added
 constraint $|\Delta \mfa^{\cB}_{t,b}| = |\mfa^{f, \cB}_{t,b}|$
allows us to apply Hoeffding's inequality in an easier manner.


When $D_t =1$, $\mfa^{f, \cB}_{t,b} \leq 2 \hat\ell_2\log^2n$ and  $\min\{|\mfa^{f, \cB}_{t,b}|,
\hat\ell_2(\log^2n)\} \geq  |\mfa^{f, \cB}_{t,b}|/(2\log^4n)$.
By Lemma~\ref{lem:dense}, we have
$$\Pr\left[|\Delta \mfa^{\cB}_{t,b}| \geq \frac{\tau_0|\mfa^{f, \cB}_{t,b}|}{8\log^8n} \Bigg|\mathcal F_t,
  D_t = 1\right] \geq {\tau_0}\log^{-6}n.$$
Similar to the analysis of $\mfa^{u, \cG}_{t,g}$, this inequality holds because we can always
restrict to a subset of agents if the number of infected/uninfected agents
in the subcube is too large to meet the requirement in Lemma~\ref{lem:dense}.
We also have
\begin{equation}
\label{eqn:EdeltaZ}
\E[|\Delta \mfa^{\cB}_t| \mid \mathcal F_t, D_t = 1] = \sum_{b}\E[|\Delta \mfa^{\cB}_{t,b}|\mid \mathcal F_t, D_t = 1] \geq \frac{\tau^2_0}{8}|\mfa^{f,\cB}_t|\log^{-14}n.
\end{equation}
Furthermore, by the way we design the coupled process, the random
variables $|\Delta \mfa^{\cB}_{t,b}|$ are independent given $\mathcal F_t, D_t = 1$.

We consider two cases.

\emph{Case 1.} There exists an $b \in \cB_t$ such that
$|\mfa^{f,\cB}_{t,b}| \geq |\mfa^{f,\cB}_t| / \log^{29}n$. In this case, we have
$$\Pr\left[|\Delta \mfa^{\cB}_t|\geq \frac{\tau_0|\mfa^{f, \cB}_t|}{4\log^{38}n}\Big|\mathcal F_t, D_t = 1\right] \geq
\Pr\left[|\Delta \mfa^{\cB}_{t,b}| \geq \frac{\tau_0|\mfa^{f, \cB}_{t,b}|}{4\log^{9}n}\Big|\mathcal F_t, D_t = 1\right] \geq
{\tau_0}\log^{-6}n.$$

\emph{Case 2.} For all $b \in \mathcal B_t$, $|\mfa^{f,\cB}_{t,b}|<
|\mfa^{f,\cB}_t|/\log^{29}n$. In this case, we have
\begin{equation}
\label{eqn:squareZ}
\sum_{b \in \mathcal B_t}|\mfa^{f,\cB}_{t,b}|^2 \leq |\mfa^{f,\cB}_t|^2\log^{-29}n.
\end{equation}
Next, by Hoeffding's inequality (again by Theorem~\ref{thm:hoeffding}),


we have
\begin{eqnarray*}
  & &\Pr\left[|\Delta \mfa^{\cB}_t| \leq
    \frac{\tau_0|\mfa^{f, \cB}_t|}{4\log^{38}n}\Big|\mathcal F_t, D_t = 1\right]
  \\& \leq&  \Pr\left[|\Delta \xiaorui{\mfa^{\cB}_t}| \leq \frac{\E[|\Delta
      \mfa^{\cB}_t|\mid \mathcal F_t, D_T = 1]}{2}\Big|\mathcal F_t,
    D_t = 1\right]\\
  & = & \Pr\left[\sum_b|\Delta \mfa^{\cB}_{t,b}| \leq \frac{\E[|\Delta
      \mfa^{\cB}_t|\mid \mathcal F_t, D_T = 1]}{2}\Big|\mathcal F_t, D_t =
    1\right]\\
  & \leq & 2\exp\left(-\frac{2(\frac 1 2\E[|\Delta \mfa^{\cB}_t|\mid \mathcal
      F_t, D_t = 1])^2}{\sum_{b \in \cB_t}|\mfa^{f, \cB}_{t,b}|^2}\right) \quad
  \mbox{(apply Hoeffding's inequality; we have $|\Delta \mfa^{\cB}_{t, g}| \leq
    |\mfa^{f, \cB}_{t, g}|$ by construction. )}\\
  & \leq & 2
  \exp\left(-\frac{2\frac{\tau^4_0}{16^2}|\mfa^{f, \cB}_t|^2\log^{-28}n}{|\mfa^{f, \cB}_t|^2\log^{-29}n}\right)\quad
  \mbox{(by Equation~\ref{eqn:EdeltaZ} and Equation~\ref{eqn:squareZ})}
  \\
  & \leq & 1 - \tau_0\log^{-6}n.
\end{eqnarray*}
\end{proof}

\subsection{Leveraging local analysis}
We now move to the global diffusion upper bound. As discussed in the beginning of this section, the balance between the distributions of each type of subcube and the distributions of actual agents plays a crucial role in our analysis. Fix an arbitrary time $t$, we classify the joint configurations
of the agents into four types:
\begin{itemize}
\item type 1 (namely $\mathcal P_{1, t}$): when $|\mathcal G_t| \leq \frac 1 2 ((2n + 1)/\hat \ell_2)^3$ and $|\mfa^{f, \cG}_t|\geq \frac 1 2 |\mfa^f_t|$.
\item type 2 (namely $\mathcal P_{2, t}$): when $|\mathcal G_t| \leq \frac 1 2 ((2n + 1)/\hat \ell_2)^3$ and $|\mfa^{f, \cG}_t|< \frac 1 2 |\mfa^f_t|$.
\item type 3 (namely $\mathcal P_{3, t}$): when $|\mathcal G_t| > \frac 1 2 ((2n + 1)/\hat \ell_2)^3$ and $|\mfa^{u, \cG}_t|< \frac 1 2 |\mfa^u_t|$.
\item type 4 (namely $\mathcal P_{4, t}$): when $|\mathcal G_t| > \frac 1 2 ((2n + 1)/\hat \ell_2)^3$ and $|\mfa^{u, \cG}_t|\geq \frac 1 2 |\mfa^u_t|$.
\end{itemize}
Recall that $\mathcal F_t$ refers to the information on the global configurations
up to time $t$. We shall abuse notation slightly and say $\mathcal F_t \in \mathcal P_{i, t}$ if the configuration of the agents at time $t$ belongs
to the $i$th type described above.
Notice that $\mathcal F_t$ belongs to exactly one of the sets
$\mathcal P_{1, t}$, $\mathcal P_{2, t}$, $\mathcal P_{3, t}, \mathcal
P_{4, t}$. In brief, scenarios $\mathcal P_{1, t}$ and $\mathcal P_{2, t}$ have a majority of uninfected subcubes, while $\mathcal P_{3, t}$ and $\mathcal P_{4, t}$ have a majority of infected subcubes. From another perspective, $\mathcal P_{1, t}$ and $\mathcal P_{3, t}$ refer to situations when the dominant types (with respect to the status of infection) are dense in their subcube types (infected/uninfected subcubes), while $\mathcal P_{2, t}$ and $\mathcal P_{4, t}$ refer to the more uniform scenarios. The next lemma states that when $\cF_t \in \cP_{1,t} \cup \cP_{2,t}$, the total
number of infected agents $|\mfa^f_t|$ grows in proportion to a monotone function of $|\mfa^f_t|$ within $\Delta t$ steps. On
the other hand, when $\cF_t\in \cP_{3,t}\cup \cP_{4,t}$, the total
number of uninfected agents $|\mfa^u_t|$ is reduced in proportion to a monotone function of
$|\mfa^u_t|$ within $\Delta t$ steps.

\begin{lemma}
\label{lem:4events}
Fix an arbitrary $t$,
define the following events,
{\small
$$
\begin{array}{lclclcl}
e_1(t) &=& \left\{|\Delta \mfa_t| \geq
  0.09\tau_0\left(\frac{|\mfa^f_t|}{4 \hat \ell_2 \log^2
      n}\right)^{2/3}\frac{\hat \ell_2}{\log^{13}n}\right\} & & e_2(t) &=& \left\{|\Delta \mfa_t| \geq
  \frac{\tau^2_0}{8\log^{38}n}|\mfa^f_t|\right\} \\
e_3(t)& = &\left\{|\Delta \mfa_t| \geq 0.015\tau_0\left(\frac{|\mfa^u_t|}{4 \hat \ell_2 \log^2 n}\right)^{2/3}\frac{\hat \ell_2}{\log^{13}n}\right\}
& & e_4(t)& =& \left\{|\Delta \mfa_t| \geq
  \frac{\tau^2_0}{8\log^{38}n}|\mfa^u_t|\right\}.
\end{array}
$$
}
We have
$$\Pr[e_i\mid \mathcal F_t \in \mathcal P_{i, t},  D_t = 1] \geq \tau_0 \log^{-6}n$$
for $i=1,2,3,4$.
\end{lemma}

 Intuitively, $e_1$ and $e_2$ connect the number of newly infected agents to the original
number of infected agents. When $e_1$ or $e_2$ are triggered sufficiently many times,
the number of infected agents doubles. Meanwhile, $e_3$ and $e_4$ connect the number of newly infected agents to
the original number of uninfected agents. When $e_3$ or $e_4$ are triggered sufficiently many times,
the number of uninfected agents halves.

\lam{
The key to proving Lemma~\ref{lem:4events}, which will ultimately lead to a bound on the global growth rate of doubling/halving the total number of infected/uninfected agents as depicted in the next proposition, is a geometric relation between the boundary of $\mathcal{G}_t$, i.e. $\partial\mathcal{G}_t$, and $\mathcal{G}_t$ itself. More specifically, an isoperimetric bound on $\mathcal{G}_t$ guarantees that no matter how packed together these good subcubes are, there are still an order $|\mathcal{G}_t|^{2/3}$ of them exposed to the bad subcubes, hence the global infection rate cannot be too slow.
}
\begin{proof}[Proof of Lemma~\ref{lem:4events}]
\emph{Part 1.}  $\cP_{1,t}$, $|\mathcal G_t| \leq \frac 1 2 ((2n + 1)/\hat \ell_2)^3$ and $|\mfa^{f, \cG}_t|\geq \frac 1 2 |\mfa^f_t|$.
Since $D_t = 1$, the number of agents in each subcube is at most
$2\hell_2 \log^2n$. Therefore, $|\mathcal G_t| \geq |
\xiaorui{\mfa^{f, \cG}_t}|/(2\hell_2\log^2n)$.
To apply Corollary~\ref{cor:dense},
%
we need to derive a relationship between the size of
$\mathcal G_t$ and the size of $\partial \mathcal G_t$.
\liusd{This is an isoperimetric problem studied by \cite{BL91} (see Appendix~\ref{sec:geometry}
for details). By Theorem 8 in \cite{BL91} or
Theorem~\ref{thm:surface} in the appendices,}

$$|\partial \mathcal G_t| \geq 0.36 |\mathcal G_t|^{2/3} \geq
0.36\left(\frac{|\mfa^{f, \cG}_t|}{2\hat\ell_2\log^2n}\right)^{2/3}.$$

We have
\begin{eqnarray*}
& & \Pr\left[|\Delta \mfa_t| \geq 0.09\tau_0\left(\frac{|\mfa^f_t|}{4\hat \ell_2\log^2n}\right)^{2/3}\frac{\hat \ell_2}{\log^{13}n}\Bigg|\mathcal F_t \in \mathcal P_{1, t}, D_t = 1\right] \\ & \geq &
\Pr\left[|\widetilde{\Delta \mfa^{\cG}_t}| \geq 0.09\tau_0\left(\frac{|\mfa^{f, \cG}_t|}{2\hat \ell_2\log^2n}\right)^{2/3}\frac{\hat \ell_2}{\log^{13}n}\Bigg|\mathcal F_t \in \mathcal P_{1, t}, D_t = 1\right] \\
& \geq & \Pr\left[|\widetilde{\Delta \mfa^{\cG}_t}| \geq |\partial \mathcal G_t|\frac{\tau_0\hat \ell_2}{4\log^{13}n}\Bigg| \mathcal F_t \in \mathcal P_{1, t}, D_t = 1\right] \\
& \geq & \tau_0\log^{-6}n \quad \mbox{ (by Corollary~\ref{cor:dense})}
\end{eqnarray*}

\emph{Part 2.} $|\mathcal G_t| \leq \frac 1 2 ((2n + 1)/\hat
\ell_2)^3$ and $|\mfa^{f, \cG}_t|< \frac 1 2 |\mfa^f_t|$. Notice that
$|\mfa^{f, \cB}_t| \geq |\mfa^f_t|
/2
$ and $|\Delta \mfa^{\cB}_t| \leq |\Delta \mfa_t|$.
We have
\begin{eqnarray*}
\Pr\left[|\Delta \mfa_t| \geq \frac{\tau^2_0}{8\log^{38}n}|\mfa^f_t|\Big|\mathcal F_t\in \mathcal P_{2, t}, D_t = 1 \right] & \geq &
\Pr\left[|\Delta \mfa^{\cB}_t| \geq \frac{\tau^2_0}{4\log^{38}n}|\mfa^{f, \cB}_t|\Big|\mathcal F_t \in \mathcal P_{2, t}, D_t = 1 \right] \\
& \geq & {\tau_0}\log^{-6}n \quad \mbox{(by
  Corollary~\ref{lem:sparse})}
\end{eqnarray*}

\emph{Part 3.} $|\mathcal G_t| > \frac 1 2 ((2n + 1)/\hat \ell_2)^3$ and $|\mfa^{u, \cG}_t|< \frac 1 2 |\mfa^u_t|$.
This is similar to part 1. We have
$$|\partial \mathcal G_t| = |\dot \partial \mathcal B_t| \geq |\partial \mathcal B_t| / 6 \geq 0.06|\mathcal B_t|^{2/3} \geq 0.06\left(\frac{|\mfa^{u, \cB}_t|}{2\hat \ell_2\log^2n}\right)^{2/3}.$$
The second inequality holds because the exterior surface $\partial \mathcal B_t$ is in the neighborhood of $\dot \partial \mathcal B_t$ and $|N(\partial \mathcal B_t|)| \leq 6|\partial \mathcal B_t|$.
Notice that $|\mfa^{u, \cB}_t| \geq |\mfa^u_t|/2$.
We have
\begin{eqnarray*}
& & \Pr\left[|\Delta \mfa_t| \geq 0.015\tau_0\left(\frac{|\mfa^u_t|}{4\hat \ell_2\log^2n}\right)^{2/3}\frac{\hat \ell_2}{\log^{13}n}\Bigg|\mathcal F_t \in \mathcal P_{3, t}, D_t = 1\right] \\ & \geq &
\Pr\left[|\Delta \mfa^{\cB}_t| \geq 0.015\tau_0\left(\frac{|\mfa^{u, \cB}_t|}{2\hat \ell_2\log^2n}\right)^{2/3}\frac{\hat \ell_2}{\log^{13}n}\Bigg|\mathcal F_t\in \mathcal P_{3, t}, D_t = 1\right] \\
& \geq & \Pr\left[|\Delta \mfa^{\cB}_t| \geq |\partial \mathcal G_t|\frac{\tau_0\hat \ell_2}{4\log^{13}n}\Bigg| \mathcal F_t\in \mathcal P_{3, t}, D_t = 1\right] \\
& \geq & \tau_0\log^{-6}n \quad \mbox{ (by Corollary~\ref{cor:dense})}
\end{eqnarray*}

\emph{Part 4.} $|\mathcal G_t| > \frac 1 2 ((2n + 1)/\hat \ell_2)^3$
and $|\mfa^{u, \cG}_t|\geq \frac 1 2 |\mfa^u_t|$. This is similar to part 2. Notice that
$|\mfa^u_t| \leq 2|\mfa^{u, \cG}_t|$ and $|\Delta \mfa_t| \geq |\Delta \mfa^{\cG}_t|$.
We have
\begin{eqnarray*}
\Pr\left[|\Delta \mfa_t| \geq \frac{\tau^2_0}{8\log^{38}n}|\mfa^u_t|\Big|\mathcal F_t\in \mathcal P_{4, t}, D_t = 1 \right] & \geq &
\Pr\left[|\Delta \mfa^{\cG}_t| \geq \frac{\tau^2_0}{4\log^{38}n}|\mfa^{u, \cG}_t|\Big|\mathcal F_t\in \mathcal P_{4, t}, D_t = 1 \right] \\
& \geq & {\tau_0}\log^{-6}n. \quad \mbox{(by Corollary~\ref{lem:sparse})}
\end{eqnarray*}
\end{proof}

Our major proposition presented next essentially pins down the number of times these events
need to be triggered to double the number of infected agents or halve the number of uninfected
ones.

\begin{proposition}\label{lem:double}Consider the information diffusion
  problem over $\cV^3$ with $m$ agents. For any fixed $t
\leq n^{2.5}-4\sqrt{\frac m n}
  \log^{45}n\Delta t
$, define the following events
{\small
$$\chi_1(t) \equiv \left(|\mfa^f_{t + 4\sqrt{\frac m n}\log^{45}n\Delta t}| \geq 2|\mfa^f_t|\right) \mbox{ and }\chi_2(t) \equiv \left(|\mfa^u_{t + 4\sqrt{\frac m n}\log^{45}n\Delta t}| \leq \frac 1 2|\mfa^u_t|\right).$$
}
We have
$$\Pr[\chi_1(t) \vee \chi_2(t)] \geq 1-\exp(-\log^2n).$$
\end{proposition}

\lam{Note that this bound suggests that for each time increment $4\sqrt{\frac m
  n}\log^{45}n\Delta t$, either the number of infected agents doubles
or the number of uninfected agents is reduced by half
with high probability.
Therefore, within
time at most
$2\log n\cdot\left(4\sqrt{\frac m n}\log^{45}n\Delta t\right) = \newxiaorui{128n\hat \ell_2\log^{47}n}$
all the agents get infected with probability at least $1 - 2\log n\exp(-\log^2n)$. This proves Theorem~\ref{thm:up3}.

 To summarize our approach, Corollaries~\ref{cor:dense} and \ref{lem:sparse} first translate the local infection rate of Lemma~\ref{lem:dense} into a rate based on the subcube types (i.e. good and bad subcubes). Then Lemma~\ref{lem:4events} further aggregates the growth rate to depend only on the infected and uninfected agents, by looking at the geometrical arrangement of the subcubes. Nevertheless, the bound from Lemma~\ref{lem:4events} is still too crude, but by making a long enough sequence of trials i.e. $4\sqrt{\frac m n}\log^{45}n$ times, at least one of the four scenarios defined in Lemma~\ref{lem:4events} occurs for a significant number of times, despite the $\Omega(\log^{-6}n)$ probability of occurrence for each individual step for any of the four scenarios. This leads to the probabilistic bound for $\chi_1(t) \vee \chi_2(t)$.
}
\begin{proof}[Proof of Proposition~\ref{lem:double}]
First, notice that Lemma~\ref{lem:4events} states that regardless of the diffusion process'
history, one of $e_1$, $e_2$, $e_3$, and $e_4$ is guaranteed to occur with $\tilde \Omega(1)$
probability. On the other hand, it is not difficult to see that
when \emph{any} of the events $e_1$, $e_2$, $e_3$, $e_4$ occurs $\tilde \Omega(n/\hat \ell_2)$ times,
then either $|\mfa^f_t|$ doubles or $|\mfa^u_t|$ reduces by a half.
Although we do not know exactly which event happens at a specific time $t$, we argue
that so long as we wait long enough, the collection of events $\{e_1, ..., e_4\}$ occurs
$4\tilde \Omega(n/\hat \ell_2)$ times and by Pigeonhole principle, at least one of $e_1$, $e_2$, ..., $e_4$
will be triggered  $\tilde \Omega(n/\hat \ell_2)$ times, concluding the proposition. The above argument can be made rigorous via a Chernoff bound on the \emph{total} number of occurrence of all four events.



%
We now implement the idea. Let us define $\varsigma = 4\sqrt{\frac m n} \log^{45}n$. Let $t_i = t +
(i-1)\Delta t$ for $i\in [\varsigma]$. Note that $t_i$ depends on both $i$ and $t$, but we suppress the dependence on $t$ for succinctness.
This also applies to all other defined quantities in this proof. Recall from Lemma \ref{lem:dchild}
that $\mathcal F_{t_i}$ encodes all the available information
up to time $t_i$. For each $i \in [\varsigma]$, define the following pairs of indicator functions
$$
I_{1, 2}(t_i) =
\left\{
\begin{array}{ll}
1 & \mbox{ if $\cF_{t_i} \in \cP_{t_i, 1}\cup \cP_{t_i, 2}$} \\
0 & \mbox{ otherwise}
\end{array}
\right. \quad \mbox{ and } \quad
I_{3, 4}(t_i) =
\left\{
\begin{array}{ll}
1 & \mbox{ if $\cF_{t_i} \in \cP_{t_i, 3}\cup \cP_{t_i, 4}$} \\
0 & \mbox{ otherwise}
\end{array}
\right.
$$
Notice that for arbitrary $t_i$, $I_{1,2}(t_i) + I_{3,4}(t_i) = 1$.
Next define
$$\varphi_i =
I_{1,2}(t_i)\cdot \frac{|\Delta \mfa_{t_i}|}{|\mfa^f_{t_i}|} +
I_{3,4}(t_i)\cdot \frac{|\Delta \mfa_{t_i}|}{|\mfa^u_{t_i}|},$$

We first show a lower bound for $\sum_{i \leq \varsigma} \varphi_i$. Our strategy is to invoke Lemma~\ref{lem:4events} and apply a Chernoff bound. Special care needs to be taken when $D = 0$.

Let

$$r = \min\left\{
0.015\tau_0\left(\frac{\hell_2}{4m\log^2n}\right)^{\frac 1 3}\frac{1}{\log^{16}n},
\frac{\tau_0^2}{8\log^{38}n}\right\}.$$
By Lemma~\ref{lem:4events}, we can see regardless of whether $\mathcal F_{t_i}$ belongs to $\mathcal P_{1, t_i}, ...,$ or $\mathcal P_{4, t_i}$,
\begin{equation}\label{eqn:deltalow}
\pr[ \varphi_i \geq r
\mid \mathcal F_{t_i},
D_{t_i} = 1
]
\geq \tau_0\log^{-6}n.
\end{equation}
where $\tau_0$ is the constant specified in Lemma~\ref{lem:4events}. Here, we verify the case $\mathcal F_{t_i} \in \mathcal P_{1, t_i}$
for Equation~\ref{eqn:deltalow}. The computation for the other
three cases can be carried out similarly.
\begin{eqnarray*}
& & \Pr[\varphi_i \geq r \mid \mathcal F_{t_i} \in \mathcal P_{1, t_i}, \mathcal D_{t_i} = 1] \\
& \geq &  \Pr\left[\frac{|\Delta \mfa_t|}{|\mfa^f_t|} \geq \frac{0.015}{\log^{16}n}\tau_0 \left(\frac{\hat \ell_2}{4m\log^2n}\right)^{\frac 1 3}\Bigg| \mathcal F_{t_i} \in \mathcal P_{1, t_i}, D_{t_i} = 1\right] \\
& \geq & \Pr\left[|\Delta \mfa_t| \geq \frac{0.09}{\log^{16}n}\tau_0 |\mfa^f_t| \left(\frac{\hat \ell_2}{4m\log^2n}\right)^{\frac 1 3}\Bigg| \mathcal F_{t_i} \in \mathcal P_{1, t_i}, D_{t_i} = 1\right] \\
& \geq & \Pr\left[|\Delta \mfa_t| \geq \frac{0.09}{\log^{16}n}\tau_0
\left(\frac{|\mfa^f_t|}{4\hat \ell_2\log^2 n}\right)^{2/3}{\underbrace{|\mfa^f_t|}_{\leq m}}^{1/3}(4\hat \ell_2\log^2 n)^{2/3}\left(\frac{\hat \ell_2}{4m\log^2n}\right)^{\frac 1 3} \Bigg| \mathcal F_{t_i} \in \mathcal P_{1, t_i}, D_{t_i} = 1\right] \\
& \geq & \Pr\left[|\Delta \mfa_t| \geq \frac{0.09}{\log^{3}n}\tau_0
\left(\frac{|\mfa^f_t|}{4\hat \ell_2\log^2 n}\right)^{2/3}\frac{\hat \ell_2}{\log^{13}n} m^{1/3}(4\log^2 n)^{2/3}\left(\frac{1}{4m\log^2n}\right)^{\frac 1 3} \Bigg| \mathcal F_{t_i} \in \mathcal P_{1, t_i}, D_{t_i} = 1\right] \\
& \geq & \Pr\left[|\Delta \mfa_t| \geq 0.09\tau_0
\left(\frac{|\mfa^f_t|}{4\hat \ell_2\log^2 n}\right)^{2/3}\frac{\hat \ell_2}{\log^{13}n} \Bigg| \mathcal F_{t_i} \in \mathcal P_{1, t_i}, D_{t_i} = 1\right] \\
& \geq & \tau_0 \log^{-6}n \mbox{ (Lemma~\ref{lem:4events})}
\end{eqnarray*}

Next, let us define a family of indicator random variables $\{I(i): i \leq \varsigma\}$ so that $I(i)$ is $\mathcal F_{t_i}$-measurable and
$$I(i) =
\left\{
\begin{array}{ll}
1 & \mbox{ if $\varphi_i \geq r$}\\
0 & \mbox{ otherwise}.
\end{array}
\right.$$
Notice that $\sum_{i \leq \varsigma}\varphi_i \geq  r\left(\sum_{i \leq \varsigma}I(i)\right)$.
By Equation~\ref{eqn:deltalow}, we have
$$\Pr[I(i) = 1\mid \mathcal F_{t_i}, D_{t_i} = 1] \geq \tau_0\log^{-6}n.$$

Next, let us
introduce another family of r.v. $\{I'(i): i \leq \varsigma\}$ to incorporate the good density variable as follows:
$$
I'(i) =
\left\{
\begin{array}{ll}
I(i) & \mbox{ if $D_{t_i}=1$}\\
1 & \mbox{ otherwise}
\end{array}
\right.
$$
Since $I'(i)\geq I(i)$, we also have
$$ \zliu{\E}[I'(i) \mid \mathcal F_{t_i}, D_{t_i} = 1] \geq \tau_0\log^{-6}n.$$
On the other hand, by construction
$$\E[I'(i) \mid \mathcal F_{t_i},D_{t_i}=0]=1\geq\log^{-7}n$$
This concludes that
$$\E[I'(i)\mid \mathcal F_{t_i}]\geq\log^{-7}n$$
%
which implies
$$\E[\sum_{i \leq \varsigma}I'(i)]
 \geq
\varsigma \log^{-7}n.
$$


We construct a sequence of random variables $\{\xi_i\}$ such that
$\xi_0 = 0$ and $\xi_{i + 1} = \xi_{i} + (I'(i + 1) - \E[I'(i + 1)
\mid \mathcal F_{t_i}])$. We can verify that $\xi_i$ is a martingale
with respect to $\{\mathcal F_{t_i}\}$ and $|\xi_{i} - \xi_{i -
  1}|\leq 2$ for all $i$. By  {Azuma-Hoeffding}'s inequality
(see Theorem~\ref{thm:azuma}
),
\begin{eqnarray*}
& & \Pr\left[|\xi_{\varsigma}| \geq \frac 1 2 \E[\sum_{i \leq \varsigma}I'(i)]\right] \\
& \leq & 2\exp\left(-\frac{\frac 1 4 \E^2[\sum_{i \leq \varsigma}I'(i)]}{2\sum_{i \leq \varsigma}4}\right) \\
& \leq &
2\exp(-\frac 1{32}\varsigma\log^{-14}n)
\leq \exp(-\log^{30}n).
\end{eqnarray*}
This implies
$$\Pr[\sum_{i \leq \varsigma}I'(i)\leq \frac{1}{
4
}\varsigma\tau_0\log^{-6}n]\leq \exp(-\log^{30}n).$$

Next, notice that
$$\Pr[\sum_{i \leq \varsigma}I'(i)  \neq \sum_{i \leq \varsigma}I(i)]\leq\Pr[D_{t+(\varsigma-1)\Delta t}=0]\leq\Pr[D=0]\leq\exp(-\frac{1}{15}\log^2n)$$
where the last inequality follows from Lemma~\ref{lem:densitya3d}.

We conclude that
$$\Pr[\sum_{i \leq \varsigma}I(i)\leq \frac{1}{
4
}\varsigma\tau_0\log^{-6}n]\leq \exp(-\log^{30}n)+\exp(-\frac 1 {15}\log^2n) \leq 2\exp(-\frac 1 {15}\log^2n).$$

Finally, we show when $\left(\sum_{i \leq \varsigma}I(i)> \frac{1}{
4
}\varsigma\tau_0\log^{-6}n\right)$ occurs, either $\chi_1(t)$ or $\chi_2(t)$ is true. First,
we have a lower bound $\sum_{i \leq \varsigma}\varphi_i \geq r \cdot
\frac{\tau_0\varsigma}{
4
\log^6n} \geq 2$. Next, we show this lower bound results in a minimum guarantee on $|\Delta \mfa_t|$.
Specifically, we have
$$\left(\sum_{i \leq \varsigma}I_{1,2}(t_i)\frac{|\Delta \mfa_{t_i}|}{|\mfa^f_{t_i}|}\right) + \left(\sum_{i \leq \varsigma}I_{3,4}(t_i)\frac{|\Delta \mfa_{t_i}|}{|\mfa^u_{t_i}|}\right) \geq 2,$$
which implies either
$$\sum_{i \leq \varsigma}I_{1,2}(t_i)\frac{|\Delta \mfa_{t_i}|}{|\mfa^f_{t_i}|} \geq 1 \mbox{ (Case 1)}$$
or
$$\sum_{i \leq \varsigma}I_{3,4}(t_i)
\frac{|\Delta \mfa_{t_i}|}{|\mfa^u_{t_i}|}
\geq 1 \mbox{ (Case 2)}$$

{\noindent \emph{Case 1.}} Observe that
$$|\mfa^f_{t + 4\sqrt{\frac m n}\log^{45}\Delta t} - \mfa^f_{t}| = \sum_{i \leq\varsigma}|\Delta \mfa_{t_i}|$$
because $\Delta \mfa^f_{t_i}$ are all disjoint for different $i$.
We have
\begin{eqnarray*}
|\mfa^f_{t + 4\sqrt{\frac m n}\log^{45}\Delta t} - \mfa^f_{t}|
 & \geq & \sum_{i \leq\varsigma}|\Delta \mfa_{t_i}|\\
& \geq & \sum_{i \leq\varsigma}I_{1,2}(t_i)|\Delta \mfa_{t_i}| \\
& = & |\mfa^f_t|\sum_{i \leq\varsigma}I_{1,2}(t_i)\frac{|\Delta \mfa_{t_i}|}{|\mfa^f_t|} \\
& \geq & |\mfa^f_t|\left(\sum_{i \leq\varsigma}I_{1,2}(t_i)\frac{|\Delta \mfa_{t_i}|}{|\mfa^f_{t_i}|}\right)  \quad \mbox{($|\mfa^f_t|$ is non decreasing w.r.t. $t$)} \\
& \geq & |\mfa^f_t|.
\end{eqnarray*}
In this case, the event $\chi_1(t)$ occurs.

{\noindent \emph{Case 2.}} When $\mfa^u_{t_\varsigma} = \emptyset$, nothing needs to be proved. Let us focus on the situation where
$\mfa^u_{t_\varsigma} \neq \emptyset$
\begin{eqnarray*}
|\mfa^u_{t + 4\sqrt{\frac m n}\log^{45}\Delta t} - \mfa^u_{t}| & \geq & \sum_{i \leq\varsigma}|\Delta \mfa_{t_i}|\\
& \geq & \sum_{i \leq\varsigma}I_{3,4}(t_i)|\Delta \mfa_{t_i}| \\
& = & |\mfa^u_{t_{\varsigma}}|\sum_{i \leq\varsigma}I_{3,4}(t_i)\frac{|\Delta \mfa_{t_i}|}{|\mfa^u_{t_{\varsigma}}|} \\
& \geq & |\mfa^u_{t_{\varsigma}}|\left(\sum_{i \leq\varsigma}I_{3,4}(t_i)\frac{|\Delta \mfa_{t_i}|}{|\mfa^u_{t_i}|}\right)  \quad \mbox{($|\mfa^u_t|$ is non increasing w.r.t. $t$)} \\
& \geq & |\mfa^u_{t_{\varsigma}}|.
\end{eqnarray*}
Therefore,  the event $\chi_2(t)$ occurs in this case.

\end{proof}

\section{The case when the number of agents is sparse}
This section focuses on the case where $m = o(n)$:

\begin{proposition}\label{prop:mlessn}Let $\mathrm a_1, \mathrm a_2, ..., \mathrm a_m$ be
  placed uniformly at random
on $\cV^3$, where
 \yajunfinal{$m< n\log^{-2}n$.}
Let $\mathrm a_1$
be the agent that holds a virus at $t = 0$, and $T$ be the diffusion
time. We have for any constant
$c >0$,
$$\Pr[T < \frac{n^3} m \log^{-c}n] \leq \log^{-c}n$$
and
$$\Pr[T > \frac{2n^3} m\log^{15}n] \leq \exp(-(\log^2n)/2). $$
\end{proposition}

Note that our analysis in Section~\ref{sec:lb3} and Section~\ref{sec:ub3} cannot be applied directly to prove Proposition~\ref{prop:mlessn}
because we required the side of each subcube to be of length
$\ell_2 = \sqrt{\frac{n^3}{m}}$, which is larger than
$(2n + 1)$ when $m = o(n)$. The diffusion time for this case turns out to depend on $m$ and $n$
in a way different from the case where $n\log^2n < m < n^3$. Nevertheless, some of the arguments can still be borrowed from Lemma~\ref{lem:dense}, together with the use of mixing time of a random walk in $\mathcal V^3$. Because of the similarity of our analysis with previous sections, we only sketch our proof and highlight the new main technicalities.


We first show the lower bound of the diffusion time:

\begin{lemma}Let $\mathrm a_1, \mathrm a_2, \ldots, \mathrm a_m$ be
  placed uniformly at random
on $\cV^3$, where $m < 2n+1$. Let $\mathrm a_1$
be the agent that holds a virus at $t = 0$. Let $T$ be the diffusion
time. We have, for any constant
$c >0$,
$$\Pr[T < \frac{n^3} m \log^{-c}n] \leq \log^{-c}n$$
\end{lemma}
\begin{proof} Let these $m$ random walks be $S^1$, $S^2$, ...,
  $S^m$. Since each random
walk is already at stationary distribution at $t = 0$, they are all distributed uniformly at
any specific time. Therefore, for any fixed $t$ and fixed $j > 1$, $\Pr[\|S^1_t - S^j_t\|_1\leq1] \leq 7/(2n+1)^3$.
By a union bound,
$$\Pr[\exists t \leq \frac{n^3}{m}\log^{-c}n, i>1: \|S^1_t
- S^i_t\|_1\leq1 ] \leq \frac{n^3}{m\log^cn}\cdot m\cdot 7(2n+1)^{-3} < \log^{-c}n.$$
Therefore, with probability at least $1- \log^{-c}n$, $S^1$ will not meet any
other agent before $t = \frac{n^3}{m}\log^{-c}n$, which also implies
that the diffusion process has not been completed.
\end{proof}

Next we move to the upper bound:

\begin{lemma}\label{thm:mlessnup}Let $\mathrm a_1, \mathrm a_2, ..., \mathrm a_m$ be placed uniformly
at random on $\cV^3$, where $m < \frac{n}{\log^2n}$. Let $\mathrm a_1$
be the agent that holds
a virus at $t = 0$. Let $T$ be the diffusion time. We have
$$\Pr[T > \frac{2n^3} m\log^{15}n] \leq \exp(-(\log^2n)/2). $$
\end{lemma}

The following is a key lemma for the upper bound analysis.
The lemma reuses arguments that appeared in Lemma~\ref{lem:dense}. However, as the agents are
 sparser in this case, new diffusion rules for the coupling process and the corresponding probabilistic bounds are needed.
%
%
%

\begin{lemma}
\label{lem:mlessn}
Consider the diffusion process in which $m <
  \frac{n}{\log^2n}$. Fix a time $t$, and let $A^f$ and $A^u$ be the
    set of infected and uninfected agents at time $t$ with
  $|A^f| = m_1$ and $|A^u| = m_2$. Let $c$ be a sufficiently large constant and $\Delta t = cn^3(\log n)/m$.
 Let $M(t)$ be the number of newly infected agents from time $t$ to
 $t + \Delta t $.
 Assume the agents are arbitrarily (in an adversarial manner)
 distributed at time $t$. We have
 $$\Pr\left[M(t) \geq \frac{\min\{m_1, m_2\}}{\log^5n}\right]
 \geq  \frac 1 2 \log^{-5}n.$$
 \end{lemma}
\begin{proof}
Similar to the proof of Lemma~\ref{lem:dense}, we first count
the number of
times the infected agents meet the uninfected agents. We then show that this number
is close to $M(t)$ by demonstrating that the number of overcounts
is moderate, which yields the desired result.
The
device we use to count the  {number of meetings}, however, is different
 from the one we used for Lemma~\ref{lem:dense}. In
 Lemma~\ref{lem:dense}, we couple each of the walks
 in $\mathcal V^3$ with their unbounded counterparts; since we only
 focus on a short time frame, the bounded walks largely coincide with
 the unbounded ones. Here, the right time frame to analyze is longer and the walks in $\mathcal V^3$ are more likely to hit the boundary. It becomes less helpful to relate these walks with the unbounded ones. Our analysis, instead, utilizes the mixing time property of $\mathcal V^3$.

Specifically, we cut $\Delta t$ into disjoint time intervals, each of
which is of size $c n^2 \log n$ steps for some constant $c$
to be determined later. We refer the $k$-th time interval as the $k$-th
round. The total number of rounds in $\Delta t$ steps is thus $n / m$.

We couple the diffusion process with a slower one. First, only agents
in $A^f$ are allowed to transmit the virus. An agent in $A^u$  will not be able to infect others even if it becomes
infected. This rule holds throughout the $\Delta t$ time increment.

In each round, we also impose more specific constraints on the diffusion rule as follows. At the beginning of the $k$-th round (for any $k$), we first wait for
$c_0n^2\log n$ steps so that the distribution of each agent is
$1/(16n^3)$-close to uniform distribution (see Definition~\ref{def:tv} and Lemma~\ref{lem:mix} for details; $c_0$ is an appropriate constant that exists as a result of Lemma~\ref{lem:mix}).
Within these time steps, no agent becomes infected even if it meets a previously infected agent. After these steps, for an arbitrary $\mathrm a_i \in A^f$ and $\mathrm a_j \in A^u$, let $X^k_{i, j} = 1$ if both of the following conditions hold:
\begin{itemize*}
\item the $L_1$-distance between $\mathrm a_i$ and $\mathrm a_j$ is between $n/450$ and $n/500$.
\item the $L_1$-distance between $\mathrm a_i$ and any boundary is at least $n / 20$.
\end{itemize*}

Since $c_0n^2\log n$ is already the mixing time for random walks on $\mathcal V^d$, it is straightforward to see that with $\Omega(1)$ probability $X^k_{i, j} = 1$, for any $k$.

After $c_0n^2\log n$ steps at $k$th round, our slower diffusion rule
allows $\mathrm a_i \in A^f$ to transmit its virus to $\mathrm a_j \in A^u$ at the $k$th round only if
\begin{itemize*}
\item $X^k_{i, j} = 1$.
\item $\ma_i$ meets $\ma_j$ after the waiting stage and before the
  round ends.
\item $\mathrm a_i$ and $\mathrm a_j$ have not visited any boundary after the waiting stage before they meet. In other words, an agent $\mathrm a_i \in A^f$ ($\mathrm a_j \in A^u$ resp.) loses its ability to transmit (receive resp.) the virus when it hits the boundary.
\end{itemize*}
Let $Y^k_{i,j}$ be an indicator random variable that sets to $1$ if and only if $\mathrm a_i\in A^f$ transmits its virus to $\mathrm a_j\in A^u$ under the slower diffusion rule at the $k$th round, \emph{pretending that $\mathrm{a}_j$ is uninfected at the beginning of the $k$-th round even if it gets infected in the previous rounds}. Hence $Y^k_{i,j}$, for a specific $i$ and $j$, can be 1 for more than one $k$. This apparently unnatural definition is used for the ease of counting in the sequel.

By Lemma~\ref{lem:twowalkboundary},
$$\Pr[Y^k_{i, j} = 1] \geq \Pr[Y^k_{i,j} = 1 \mid X^k_{i, j} = 1]\Pr[X^k_{i,j} = 1] = \Omega(1/n).$$

Therefore, we have
\begin{equation}
\label{eqn:mlessn:expectedY}
\E\left[\sum_{i, j, k}Y^k_{i,j}\right] =
\Omega\left(\frac{m_1m_2}{m}\right) \geq \frac{\tau_1m_1m_2}{m},
\end{equation}
for some constant $\tau_1$.
 \zliu{

We briefly lay out our subsequent analysis. We want to show two properties:
\begin{enumerate}
\item $\Pr[\sum_{i, j, k}Y^k_{i, j} =\Omega(\min\{m_1, m_2\})] = \tilde \Omega(1)$.
\item For all $j$, $\sum_{i, k}Y^k_{i, j} = \tilde O(1)$ with high probability.
\end{enumerate}
We claim that these two properties together concludes our result. Roughly speaking, when $\left(\sum_{i, j, k}Y^k_{i, j} =\Omega(\min\{m_1, m_2\})\right)$ and
$\left(\forall j: \sum_{i, k}Y^k_{i, j} = \tilde O(1)\right)$ occur, each $\mathrm a_j \in A^u$
meets at most $\tilde O(1)$ agents in $A^f$ while the total number of meetings between
infected and uninfected agents is $\min\{m_1, m_2\}$. Consequently,
the total number of uninfected agents that ever meet an infected agent is
$\tilde \Omega(\min\{m_1, m_2\})$, hence our conclusion.

To prove the first property, we need to show with high probability, for any $j$, we have $\sum_{i, k}Y^k_{i, j} = \tilde O(1)$. Similarly, we also need to show with high probability, for any $i$, $\sum_{j, k}Y^k_{i, j} = \tilde O(1)$. Combining both of these we have
$\sum_{i, j, k}Y^k_{i, j} = \tilde O(\min\{m_1, m_2\})$ with high probability. Together with
Equation~\ref{eqn:mlessn:expectedY}, some rearrangement of terms and Chernoff bounds, we can conclude that
$\Pr[\sum_{i, j, k}Y^k_{i, j} =\tilde \Omega(\min\{m_1, m_2\})] = \tilde \Omega(1)$.

We now carry out this scheme. We proceed to show that
\begin{equation}\label{eqn:smallsum}
\Pr[\forall j: \sum_{i, k}Y^k_{i, j} = \tilde O(1)] \geq 1 - \exp(-\Omega(\log^2n)).
\end{equation}
and note that showing $\sum_{j, k}Y^k_{i, j} = \tilde O(1)$ can be done similarly.
We prove Equation~\ref{eqn:smallsum} via the following two steps:

\begin{enumerate*}
\item first, we show that with high probability, $\sum_i Y^k_{i,j} = \tilde O(1)$ for any fixed $k$ and $j$.
\item second, we show that with high probability, the number of $k$'s such that $\sum_{i}Y^k_{i, j} > 0$ is $\tilde O(1)$ for all $j$.
\end{enumerate*}
Intuitively, the first step ensures that there will not be too many
meetings
associated with $\mathrm a_j$ for any single round. The second step
specifies an upper bound on the number of rounds in which $\mathrm
a_j$
meets
 at least one infected agent. When both event occurs, the total
 number of
meetings for $\ma_j$
is $\tilde O(1)$.

Let us start with the first step. Fix a specific $k$ and $\mathrm a_j \in A^u$, by Corollary~\ref{cor:rwcatchall}, we have
\begin{eqnarray*}
\Pr\left[\sum_{\mathrm a_i \in A^f}Y^k_{i,j}\geq\log^2n\Bigg|X^k_{i,j}\right] &
\leq & \binom{\sum_{\mathrm a_i \in
    A^f}X^k_{i,j}}{\log^2n}\left(\frac{c_1\log^2n}{n}\right)^{\log^2n}
\\
& \leq & \binom{m_1}{\log^2n}\left(\frac{c_1\log^2n}{n}\right)^{\log^2n}\\
& \leq & \exp(-\log^2n \log \log n).
\end{eqnarray*}

By a union bound, we can also conclude that
\begin{equation}\label{eqn:sumy}
\Pr\left[\exists k\leq\frac{n}{m}: \sum_{\mathrm a_i\in A^f}Y^k_{i,j} \geq \log^2n\right] \leq \exp(-\frac 1 2 \log^2n \log \log n).
\end{equation}

Next, let us move to the second step. Let us define a family of indicator random variables $I(j, k)$, which sets to $1$ if and only if $\sum_{\mathrm a_i \in A^f}Y^k_{i,j} \geq 1$. When $j$ and $k$ are fixed, we can compute the probability $\Pr[I(j, k) = 1]$:
$$\Pr[I(j, k) = 1] = \E[I(j, k)] \leq \E\left[\sum_{\mathrm
    a_i\in A^f}Y^k_{i,j}\right] \leq \frac{\tau_1m_1}{n} .$$
The probability holds regardless of the history of the process up to
{the time the $k$th round starts} because $c_0n^2\log n$ time steps
are used at $k$th round to
shuffle the agents so that they are
distributed sufficiently uniform
after these steps. We may apply a special case of Chernoff bound
(see, e.g., Theorem~\ref{thm:dependentchernoff}) to show that $\Pr[\sum_{k} I(j,k)
> \log^2 n]  < \exp(-\log^3n)$.

Therefore, we have
\begin{equation}\label{eqn:sumi}
\Pr\left[\exists \mathrm a_j\in A^u: \sum_{k \leq \frac n m}I(j, k) > \log^2n\right] \leq \exp(-\Theta(\log^2n)).
\end{equation}

For a specific $\mathrm a_j \in A^u$, when both $\left( \sum_k I(j, k) < \log^2 n\right)$
and $\left( k: \sum_i Y^k_{i,j} \leq \log^2n\right)$, we know that $\sum_{i \in A^f, k}Y^k_{i, j} \leq\log^4n$.
Hence Equation~\ref{eqn:sumy} and \ref{eqn:sumi} imply $\Pr[\sum_{i \in A^f, k}Y^k_{i, j} > \log^4n] \leq \exp(-\Theta(\log^2n))$ and therefore
\begin{equation}\label{eqn:sumu}
\Pr\left[\exists \mathrm a_j \in A^u: \sum_{\mathrm{a}_i \in A^f, k}Y^k_{i, j} > \log^4n\right] \leq \exp(-\Theta(\log^2n))
\end{equation}
Similarly, we can show
\begin{equation}\label{eqn:sumf}
\Pr\left[\exists \mathrm a_i \in A^f: \sum_{\mathrm{a}_j \in A^u, k}Y^k_{i, j} > \log^4n\right] \leq \exp(-\Theta(\log^2n))
\end{equation}
Equation~\ref{eqn:sumu} and \ref{eqn:sumf} yield
\begin{equation}\label{eqn:ysum}
\Pr\left[\sum_{i,j,k} Y^k_{i,j} < \min\{m_1,m_2\} \log^4n\right] \geq
1-\exp(-\Theta(\log^2n)).
\end{equation}
This gives the first property in the discussion following Equation~\ref{eqn:mlessn:expectedY}. Moreover, Equation~\ref{eqn:sumu} gives the second property.

Now, by using similar argument in the proof of Lemma~\ref{lem:dense}, Equation~\ref{eqn:mlessn:expectedY} and \ref{eqn:ysum} together give
\begin{equation}\label{eqn:ysum estimate}
\Pr\left[\sum_{i, j, k}Y^k_{i,j} \geq \frac{\tau_1\min\{m_1, m_2\}}{4}\right] \geq \Pr\left[\sum_{i, j, k}Y^k_{i,j} \geq \frac{\tau_1 m_1m_2}{2m}\right]\geq \log^{-5}n
\end{equation}
When $\left(\sum_{i, j, k}Y^k_{i,j} \geq \frac{\tau_1\min\{m_1, m_2\}}{4}\right)$ and $\left(\forall\mathrm{a}_j\in A^u:\sum_{\mathrm{a}_i\in A^f,k}Y_{i,j}^k\leq\log^4n\right)$, the total number of
infected agents is at least $\frac{\tau_1\min\{m_1, m_2\}}{4\log^4n}$. Hence, by setting $c = 2c_0$, and using Equation~\ref{eqn:sumu} and \ref{eqn:ysum estimate}, our lemma follows.
}
\end{proof}

From this we can mimic the argument that appeared in Proposition~\ref{lem:double} to reach the conclusion below:
\begin{corollary}
\label{cor:mlessndouble}
Consider the diffusion process in which $m <
  \frac{n}{\log^2n}$. Fix a specific time $t$, and let $A^f$ and $A^u$ be the set of infected and uninfected agents at $t$ such that
  $|A^f| = m_1$ and $|A^u| = m_2$.
 Let $M(t)$ be the number of new infected agents between time $t$ and time $t + \frac{n^3}{m}\log^{14}n$.
 Assume the agents are arbitrarily (in an adversarial manner)
 distributed  {at time $t$}, we have
 $$\Pr\left[M(t) \geq \min\left\{m_1, \frac{m_2} 2\right\}\right]
 \geq 1-\exp(-\log^{2}n).$$
\end{corollary}

Similar to Lemma~\ref{lem:double}, Corollary~\ref{cor:mlessndouble} estimates the growth rate of infection as either doubling the number of infected agents or halving the uninfected ones within a certain time interval. One can then show that this implies Lemma~\ref{thm:mlessnup}. The argument is analogous to Section~\ref{sec:ub3} and hence is skipped here.

\bibliographystyle{plain}  
\bibliography{diffusion}
\appendix
\section{Probability Review}\label{sec:aprob}
This section reviews some probabilistic building blocks that are needed in our analysis.

\subsection{Concentration bounds}
\begin{theorem}[Chernoff bounds]\label{thm:chernoff}Let $X_1, ..., X_n$ be independent  \zliu{Poisson} trials with $\Pr[X_i] = p_i$. Let $X = \sum_{i \leq n}X_i$ and $\mu = \E[X]$. Then the following Chernoff bounds hold:
\begin{itemize}
\item  For $0 < \delta < 1$,
$$\Pr[|X - \mu| \geq \delta  {\mu}
] \leq 2\exp(-\mu\delta^2/3).$$
\item For $R \geq 6 \mu$,
$$\Pr[X \geq R] \leq 2^{-R}.$$
\end{itemize}
\end{theorem}

\begin{theorem}[Chernoff bounds for dependent variables]
\label{thm:dependentchernoff}Let $X_1, ..., X_n$ be possibly dependent Poisson trials with
$\Pr[X_i =1 \mid  \zliu{X_1, ..., X_{i - 1} }] \geq p$. Let $X = \sum_{i
  \leq n}X_i$ and $\mu = np$. Then the following Chernoff bound holds:
\begin{itemize}
\item For $0 < \delta < 1$,
$$\Pr[X \leq (1-\delta)\mu] \leq \exp(-\mu\delta^2/2).$$
\end{itemize}
On the other hand, if $\Pr[X_i =1 \mid   \zliu{X_1, ..., X_{i - 1} }] \leq
p$, the following bound holds:
\begin{itemize}
\item For any $\delta >0$,
$$\Pr[X > (1+\delta)\mu] \leq \exp(-\mu\delta^2/4).$$
\end{itemize}
\end{theorem}

\begin{theorem}[Hoeffding's inequality] \label{thm:hoeffding}Let $X_1, X_2, ..., X_n$ be
independent random variables such that $a_i \leq X_i \leq b_i$.
Let $S = \sum_{i \leq n}X_i$. Then
$$\Pr\left(|S - \E[S]| \geq t\right) \leq 2 \exp\left(-\frac{2t^2}{\sum_{i \leq n}(b_i - a_i)^2}\right).$$
\end{theorem}

\begin{theorem}[Azuma-Hoeffding inequality]\label{thm:azuma} Let $X_1$, ...$X_n$ be a martingale such that
$$|X_{k} - X_{k - 1}| \leq c_k.$$
Then, for all $t \geq 0$ and any $\lambda > 0$,
$$\Pr[|X_t - X_0| \geq \lambda] \leq 2\exp(-\lambda^2/(2\sum_{i = 1}^nc^2_i)).$$
\end{theorem}

\subsection{Random walks in bounded and unbounded spaces}
In this subsection we state some peripheral results for bounded and unbounded random walks that will be useful in handling the meeting time and position of multiple random walks in the next section.

\label{subsec:bounded}
\begin{theorem}\label{thm:discretediffusion} \emph{(First Passage Time,
    Chapter 3 of \cite{Feller50})} Let $\{\mathbf S_t: t \in \mathbf
  N\}$ be a one dimensional random walk from the origin. The probability $\varphi_{r, t}$ that
  the first passage through $r$ occurs at time $t$ is given by
$$\varphi_{r, t} = \frac{r}{t}\binom{t}{\frac{t + r}{2}}2^{-t} \approx
\sqrt{\frac{2}{\pi}}\frac{r}{\sqrt{t^3}}e^{-r^2/(2t)}.$$
Therefore, there exists constant $C$, such that for $r,t\geq
C$, we have $\varphi_{r,t}\in ( {\frac 12}\frac{r}{\sqrt{t^3}}e^{-r^2/(2t)},
\frac{r}{\sqrt{t^3}}e^{-r^2/(2t)})$.
\end{theorem}


\begin{lemma}
\label{lem:basicmove}
  Let $S$ be a bounded random walk in $[-n,n]$ starting from
  position $P$.
   \zliu{For any other position $Q \in [-n, n]$ with $|P - Q| \geq \log^2n$,}
  the
  probability that $S$ visits $Q$ within $|P-Q|^2/\log^{4}n$
 time steps is at most $\exp(-\log^3n)$ when $n$ is sufficiently large.
\end{lemma}
\begin{proof}
 \zliu{
Let us couple $S$ with an unbounded random walk $S'$ that also starts at $P$ in the natural way, i.e. $S$ and $S'$ share the same random tosses to drive their moves.

First, we claim that at the first time $S$ visits $Q$,
the number of distinct lattice points $S'$ visits is at least $|P - Q|$. This claim can be seen through analyzing the following two cases.

\noindent{\emph{Case 1.}} The walk $S$ never visits a boundary before its first visit to $Q$. In this case, $S'$ coincides with $S$, which implies $S'$ also visits all the lattice points between $P$ and $Q$. The claim therefore follows.

\noindent{\emph{Case 2.}} The walk $S$ visits a boundary before it fist visits $Q$. In this case, the boundary that $S$ visits and the point $Q$ lie on different sides of $P$. In other words, the distance between this boundary and $Q$ is at least $|P - Q|$. Now let us only consider the time interval between the last time $S$ visits the boundary (namely, $t_0$) and the first time $S$ visits $Q$. The trajectory of $S'$ within this time interval is identical to
the trajectory of $S$ (up to an offset produced between time $0$ and $t_0$). Therefore, from $t_0$ to the first time $S$ visits $Q$, the coupled walk $S'$ visits at least $|P - Q|$ distinct lattice points.

An immediate consequence of our claim is that a necessary condition for $S$ to visit $Q$ is that $S'$ has to visit either $P - \frac{|P - Q|}2$ or
$P + \frac{|P - Q|}2$. By Theorem~\ref{thm:discretediffusion}, the probability $S'$ ever visits either of these points within time $|P - Q|^2/\log^4n$ is at most $\exp(-\log^3n)$ when $n$ is sufficiently large, which completes our proof.}
\end{proof}
 \zliu{
The next lemma concerns the first passage time for a random walk over bounded space.

\begin{lemma}\label{lem:boundedfirst} Let $S$ be a random walk on $\mathcal V^1 = \{-n, ..., n\}$ that starts at $A$.
Let $B$ be a point on $\mathcal V^1$ such that $|B - A| = r$. Let $T$ be the first time $S$
visits $B$. Fix an arbitrary constant $c$, we have:
$$\Pr[T \leq cr^2] = \Omega(1).$$
\end{lemma}

\begin{proof}Without loss of generality, let us assume $-n \leq A \leq B \leq n$.
We couple $S$ with an unbounded random walk $S'$ that also starts at $A$ in the natural way, i.e.
having $S$ and $S'$ share the same random tosses to drive their moves. Let $T'$
be the first time $S'$ visits $B$. We first show that $T' \geq T$. Note that before $T'$, $S'$ is always to the left of $B$, and hence $n$. It is then easy to see that $S$ is always overlapping or to the right of $S'$ before $T'$. Hence $S'$ hitting $B$ at $T'$ implies that $S$ has already hit it at a time before or at $T'$.

%
%
Finally, by Theorem~\ref{thm:discretediffusion}, we have
$$\Pr[T \leq cr^2] \geq \Pr[T' \leq cr^2] = \Omega(1).$$
\end{proof}

\begin{corollary}\label{cor:boundedregion}Let $S$ be a random walk on $\mathcal V^1$ that starts at $A$, and $B$ be a point on $\mathcal V^1$ such that $|B - A| = r$.
Let $c_1$ and $c_2$ be two arbitrary constants and let $t = c_1r^2$. We have
$$\Pr[|S_t - B| \leq c_2 r] = \Omega(1).$$
\end{corollary}
\begin{proof}Let $T$ be the first time $S$ visits $B$. By Bayes' rule we have
$$\Pr[|S_t - B| \leq c_2 r] \geq \Pr[|S_t - B| \leq c_2 r \mid T < t]\Pr[T < t] = \Pr[|S_t - S_T| \leq c_2 r \mid T < t]\Pr[T < t]$$
By Lemma~\ref{lem:boundedfirst}, $\Pr[T \leq t] = \Omega(1)$. Next we claim $\Pr[|S_t - S_T| \leq c_2 r \mid T < t] = \Omega(1)$. This can be seen by showing $\Pr[|S_t-S_T|\leq c_2 r|T]=\Omega(1)$ uniformly over $T\in[1,t)$. For this, note that $\Pr[|S_t-S_T|\leq c_2 r|T]=\Pr[|\tilde{S}_\tau|\leq c_2r]$ where $\tau=t-T$ and $\tilde{S}$ is a random walk starting at 0. We then write $\Pr[|\tilde{S}_\tau|\leq c_2r]=\Pr[|\tilde{S}_\tau/\sqrt{\tau}|\leq c_2r/\sqrt{\tau}]\geq\Pr[|\tilde{S}_\tau/\sqrt{\tau}|\leq c_2/\sqrt{c_1}]=\Omega(1)$ by Gaussian approximation on $\tilde{S}_\tau/\sqrt{\tau}$. Therefore, $\Pr[|S_t - B| \leq c_2 r] = \Omega(1)$.
\end{proof}
}

For a $d$-dimensional {\em unbounded} random walk starting from the
origin, let \xiaoruiview{$p_d(t,\vec x)$} be the probability that the walk visits
position \xiaoruiview{$\vec x$} at time $t$. Let $q_d(t,x)$ be the probability that the
random walk visits \xiaoruiview{$\vec{x}$} {\em within} time $t$. \zliu{
When $d = 3$, we will silently drop the subscripts and write the functions
as $p(\cdot, \cdot)$ and $q(\cdot, \cdot)$.}

\begin{theorem} \label{thm:standardrw}\cite{Hughes95} The function
  \xiaoruiview{$p_d(t, \vec x)$} has the following analytic form, when $t-\xiaoruiview{\|\vec x\|_1}$ is even:
$$\xiaoruiview{p_d(t, \vec x)} = \frac{2}{t^{d/2}}\left(\frac{d}{2\pi}\right)^{d/2}\exp\left\{\frac{-d \|\xiaoruiview{\vec x}\|^2_2}{2t}\right\} + e_t(\xiaoruiview{\vec x}),$$
where $|e_t(\xiaoruiview{\vec x})| \leq \min\left\{O(t^{-(d + 2)/2}),
  O(\|\xiaoruiview{\vec x}\|_2^{-2}t^{-d/2})\right\}$. $p_d(t,\xiaoruiview{\vec x}) = 0$ when $t-\xiaoruiview{\|\vec x\|_1}$ is odd.
\end{theorem}

%
%

\begin{theorem}\label{thm:lowervisit}\cite{AMPR01} The function
  $q_d(t, \xiaoruiview{\vec x})$ satisfies the following asymptotic relations:
\begin{itemize*}
\item If $d = 2$,  $\xiaoruiview{\vec x  \neq \vec 0}$, and $t \geq \|\xiaoruiview{\vec x\|_2^2}$, then we have
$$q_2(t, \xiaoruiview{\vec x}) = \Omega\left(\frac{1}{\log\|\xiaoruiview{\vec x}\|_2}\right).$$
\item If $d \geq 3$, $\xiaoruiview{\vec x \neq \vec 0}$, and $t \geq \|\xiaoruiview{\vec x\|_2^2}$, then we have
$$q_d(t, \xiaoruiview{\vec x}) = \Omega\left(\frac{1}{\|\xiaoruiview{\vec x}\|_2^{d-2}}\right).$$
\end{itemize*}
\end{theorem}

When $d\geq 3$, it is not difficult to see that the above asymptotic result is tight by using Markov inequality:
\begin{corollary}\label{lem:tightvisit} When $d \geq 3$, the function
  $q_d(t, \xiaoruiview{\vec x})$ satisfies the following asymptotic relation  \zliu{for} $t>\|\xiaoruiview{\vec x}\|_2^2$
$$q_d(t, \xiaoruiview{\vec x}) = \Theta\left(\frac{1}{\|\xiaoruiview{\vec x}\|_2^{d-2}}\right).$$
\end{corollary}

 \zliu{
Next we show for any random walk that could start near the boundary, waiting for a short
period allows the walk to both stay away from the boundary and be sufficiently close to where it starts.}

 \zliu{
\begin{lemma} \label{lem:ball}Consider a random walk $S$ over the $d$-dimensional space $\mathcal V^d$ that starts at $\vec x$, where $\vec x = (x_1, ..., x_d)$ is an arbitrary point in the space. Let $\vec c = (c_1, ..., c_d)$ be a point in $\mathcal V^d$ such that $\|\vec c - \vec x\| = \Theta(r)$. Also let $t = r^2$. We have
\begin{equation}
\Pr[S_{t} \in \mathbb B(\vec c, r)] = \Omega(1).
\end{equation}
\end{lemma}
\begin{proof}Recall that at each step, the random walk $S$ uniformly selects a neighboring point to move to. We may
also interpret a move of $S$ as if it first randomly selects an axis to moves along and next decides which one of the two directions
to take when the axis is fixed. Let $T_i$ be the number of the walk's move that are along the $i$-th axis within $t$ steps. Define the event
$e$ as:
$$e = \left\{\forall i: \frac 1 2 \cdot \frac t d \leq T_i \leq \frac 5 4 \cdot \frac t d \right\}.$$
By Chernoff bounds, we have for any specific $i \in [d]$,
$$\Pr\left[\frac 1 2 \cdot \frac t d \leq T_i \leq \frac 5 4 \cdot \frac t d\right] \geq 1 - \exp(-\Omega(t)) \geq 1 - \frac 1 {4d}$$
for sufficiently large $t$. Therefore,
$$\Pr[e] \geq 1 - d\cdot\frac 1 {4d} \geq \frac 3 4.$$

Let $(S_t)_i$ be the $i$-th coordinate of the point $S_t$. We next
 compute $\Pr[S_{t} \in \mathbb B(\vec c, r)\mid e]$:
\begin{eqnarray*}
& & \Pr[S_{t} \in \mathbb B(\vec c, r)\mid e] \\
& = & \E\left[ \Pr[S_{t} \in \mathbb B(\vec c, r)\mid T_1, ..., T_d, e] \mid e\right] \\
& = & \E\left[\Pr\left[\bigwedge_{i \in [d]}(S_t)_i \in [c_i - r, c_i + r] \Big| T_1, ..., T_d, e
\right]\Big| e\right] \mbox{(By the definition of $\mathbb B(\vec c, r)$)} \\
& = & \E\left[\prod_{i \in [d]}\Pr\left[(S_t)_i \in [c_i - r, c_i + r]\Big| T_i, e\right]\Big| e\right] \\
\end{eqnarray*}
The last equality holds because the moves along the $i$-th axis are independent of the moves along other
axes when $T_i$ is known.
Next, using Corollary~\ref{cor:boundedregion}, we have
$$\Pr\left[(S_t)_i \in [c_i - r, c_i + r]\Big| T_i, e\right]= \Omega(1).$$
Therefore,
\begin{eqnarray*}
\Pr[S_{t} \in \mathbb B(\vec c, r)\mid e] & = & \E\left[\prod_{i \in [d]}\Pr\left[(S_t)_i \in [c_i - r, c_i + r]\Big| T_i, e\right]\Big| e\right] \\
& = &\E\left[\prod_{i \in [d]}\Omega(1)\Big| e\right] = \Omega(1).
\end{eqnarray*}
Finally, we have
$$\Pr[S_{t} \in \mathbb B(\vec c, r)] \geq \Pr[S_{t} \in \mathbb B(\vec c, r)\mid e]\cdot \Pr[e] = \Omega(1).$$
\end{proof}

\begin{corollary}
\label{lem:waitboundary}
Let $r$ be sufficiently large and $r \leq \frac{n}{2(2\beta+ \zliu{6})}$,
where $\beta$ is an arbitrary constant between $1$ and $80d$. Let $A=\vec x$
and $B$ be two points in $\cV^d$ such that $\|A-B\|_1 \leq r$.
Consider two bounded random walks $S^1$ and $S^2$  in $\cV^d$ that start with $A$ and $B$
respectively. Then, with $\Omega(1)$ probability, at time $t = r^2$,
\begin{itemize}
\item $S^1_t$ is at least $\beta\cdot r$ away from any of the
  boundaries,
\item $\|S^1_t-A\|_\infty \leq (\beta+2)r$, and
\item  $\|S_t^1 - S_t^2\|_\infty \in (r,3r)$.
\end{itemize}
\end{corollary}
\begin{proof}Let us first find an arbitrary $\vec c = (c_1, ... c_d)$ such that
\begin{itemize}
\item  \yajun{For all $i \in [d]$: $|c_i - x_i| = (\beta + 1)r$, i.e., $\|\vec c - \vec x\| = O(r)$.}
\item For all $i \in [d]$: $-n + (\beta + 1)r \leq c_i \leq n - (\beta + 1)r$, i.e., $\vec c$ is sufficiently away from the boundary.
\end{itemize}
We set up $\beta$ in a way that such $\vec c$ always exists. By Lemma~\ref{lem:ball}, we have $\Pr[S^1_t \in \mathbb B(\vec c, r)] = \Omega(1)$.
Next, in case $S^1_t \in \mathbb B(\vec c, r)$, let $\vec d(S^1_t)$ be an arbitrary point such that
\begin{itemize}
\item $|d_i(S^1_t) - (S^1_t)_i | = 2r$
\item the distance between $\vec d(S^1_t)$ and any boundary is at least $\beta r$.
\item $ \yajun{(\beta+1)r\leq}\|\vec d(S^1_t) - B\|_{\infty} \leq (\beta + 5)r$.
\end{itemize}
Again by the way we designed $\beta$, such $\vec d(S^1_t)$ always exists so long as $S^1_t \in \mathbb B(\vec c, r)$.
Using Lemma~\ref{lem:ball} again, we have
$$\Pr[S^2_t \in \mathbb B(\vec d(S^1_t), r) \mid S^1_t \in \mathbb B(\vec c, r)] = \Omega(1).$$
Therefore, we have
$$\Pr\left[\left(S^2_t \in \mathbb B(\vec d(S^1_t), r)\right) \wedge \left(S^1_t \in \mathbb B(\vec c, r)\right)\right] = \Omega(1).$$
Finally, observe that when $\left(S^2_t \in \mathbb B(\vec d(S^1_t), r)\right) \wedge \left(S^1_t \in \mathbb B(\vec c, r)\right)$,
the three conditions specified in the Corollary are all met. This completes our proof.
\end{proof}
}
\subsection{Mixing time in graphs}
\begin{definition}[Statistical distance]\label{def:tv}Let $X$ and $Y$ be two
probability distributions over the same support
 {$\mathbf P$.}
The \emph{statistical distance} between $X$ and $Y$ is
$$\Delta(X, Y) = \max_{T\subseteq
 {\mathbf P}
}|\Pr[X \in T]
- \Pr[Y \in T]|.$$
We also say that the distribution $X$ is $\epsilon$-close to $Y$ if $\Delta(X, Y) = \epsilon$.
\end{definition}

\begin{lemma}[Mixing time for $\mathcal V^3$]\label{lem:mix} Consider a random
 walk that  {starts} at point $A$ for an arbitrary $A \in \mathcal V^3$.
Let $\pi_t(A)$ be the distribution of the walk at time $t$, and $\pi$ be the uniform distribution on the nodes in $\mathcal V^3$.
Let $\epsilon > 0$. When $t = \Theta(n^2\log(1/\epsilon))$, we have
$$\Delta(\pi_t(A), \pi) \leq \epsilon.$$
\end{lemma}
Although the mixing time of high dimensional torus were analyzed,
we are not aware of any literature that pins down the exact
  mixing time
for $\mathcal V^3$.
It is, however, straightforward to derive the mixing time
in asymptotic form
via computing the conductance of $\mathcal V^1$ (the one-dimensional
grid) and using
results on mixing times regarding tensoring graphs (e.g.,
Chapter 5 in \cite{Vadhan10} and \cite{Sinclair92}).


\newcommand{\vx}{{\vec x}}
\newcommand{\vy}{{\vec y}}
\section{Multiple random walks in bounded and unbounded spaces}\label{sec:mrw}

\begin{proof}[Proof of Lemma~\ref{lem:couple}]
Let us consider the following two processes
   $\mathbb P_1$ and $\mathbb P_2$ in the same probability space (we slightly abuse the terminology ``process" to mean the expression of the random tosses that drive all random walks of interest).
\begin{enumerate}
\item The process $\mathbb P_1$: consider the random walk $S(A)$. We
  are interested in the event that $S_t(A)$ visits $B$ within time
  $2t$, which occurs with probability $q(2t, \vec x)$. \liusd{Notice that $S(A)$ 
  is unable to visit $B$ at odd steps.}
\item The process $\mathbb P_2$: consider the random walks $S^1(A)$
  and $S^2(B)$.  We are interested in the event that the two walks
   {collide} by the time $t$, which occurs with probability $Q(t, \vec x)$.
\end{enumerate}

We couple the two random processes as follows. We first construct the single random walk in $\mathbb P_1$ from the two walks in $\mathbb P_2$. Note that one time step in $\mathbb P_2$ involves simultaneous moves of the walks $S^1(A)$ and $S^2(B)$. Corresponding to this step, the single walk in $\mathbb P_1$ will be set to move first in the same direction as $S^1(A)$, and then in the reverse direction from
$S^2(B)$. This way the moves at time $t > 0$ in $\mathbb P_2$ are translated into the moves
at time $2t - 1$ and $2t$ in $\mathbb P_1$. The construction can naturally
be reversed to map a walk in $\mathbb P_1$ to two walks in $\mathbb P_2$. This coupling ensures
the  {$L_1$} distance between $S^1$ and $S^2$ at time $t$ in $\mathbb P_2$
is the same as the distance between $S$ and $B$ at time $2t$ in $\mathbb P_1$.
Note that collision in $\mathbb P_2$ can only occur at even steps, and hence the hitting event in $\mathbb P_1$ is well-defined. Therefore $S^1$ and $S^2$ 
 {collide} at or before $t$
if and only if $S$ visits $B$ at or before $2t$.

Using the bound given in Lemma~\ref{lem:tightvisit}, we
have for $t \geq \|\vec x\|_2^2$,
$$Q(t, \vec x) = \Theta\left(\frac 1 {\|\vec x\|_2}\right)$$
\end{proof}

\begin{proof}[Proof of Lemma~\ref{lem:rwcatchall}]
Let  $\Psi^{t_1, ..., t_j}_{C_1, ..., C_j}$ be the
event that
$S^i$ and $S^{j + 1}$ collide at $C_i$ at time step $t_i$ (not necessarily for the first time) for all $i\in [j]$.
Our goal is to bound the following quantity
\begin{eqnarray*}
\Pr\left[\exists t_1, ..., t_j, C_1, ..., C_j: \Psi^{t_1, ..., t_j}_{C_1, ..., C_j} = 1\right] & \leq & j!\Pr\left[\exists t_1 \leq  ...\leq t_j, C_1, ..., C_j: \Psi^{t_1, ..., t_j}_{C_1, ..., C_j} = 1\right].\\
& \leq & j! \sum_{t_1 \leq t_2 \leq ... \leq t_j}\sum_{C_1,..., C_j}\Pr[\Psi^{t_1, ..., t_j}_{C_1, ..., C_j} = 1].
\end{eqnarray*}

We cut the time interval into $j$ frames $[0, t_1]$, $[t_1, t_2]$, ..., $[t_{j - 1} , t_j]$,
so that the random walks with different frames are independent.
Define $D_{i - 1}$ be the position of $S^i$ at time $t_{i - 1}$.
For notational convenience, we let $D_0 = A_1$, $C_0 = B$, and $t_0 = 0$.

The event $\Psi^{t_1, ..., t_j}_{C_1, ..., C_j} = 1$ implies that in the $i$-th time interval $[t_{i - 1}, t_i]$ we have
\begin{enumerate*}
\item $S^i$ moves from $A_i$ at time $0$ to $D_{i - 1}$ at time $t_{i - 1}$.
\item at time $t_{i - 1}$, the walk $S^{j + 1}$ is at $C_{i - 1}$.
\item at time $t_i$, the walk $S^{j + 1}$ and the walk $S^i$ are both at $C_i$.
\end{enumerate*}
 By standard results regarding high dimensional random walks (e.g. see Theorem~\ref{thm:standardrw}),
 the probability that the first event happens is at most
 \begin{equation}\label{eqn:mp2}
 \frac{3}{t^{1.5}_{i - 1}}\left(\frac{d}{2\pi}\right)^{d/2}\exp\left\{\frac{-3\|A_i-D_{i-1}\|^2_2}{ {2}t_{i - 1}}\right\}
\end{equation}
  and the probability that both the second and the third events happen is at most
\begin{equation}\label{eqn:mp1}
\frac{3}{(t_i - t_{i - 1})^3}\left(\frac{d}{2\pi}\right)^d\exp\left\{\frac{-3(\|D_{i - 1}-C_i\|^2_2 + \|C_{i - 1}-C_i\|^2_2)}{2(t_i - t_{i - 1})}\right\}.
\end{equation}
The error term in Theorem~\ref{thm:standardrw} is swallowed by
the larger leading constants 3 in
 {Equations \ref{eqn:mp2} and \ref{eqn:mp1}.}

As $S^i$'s walk before time $t_{i - 1}$ is independent to the walks of $S^i$ and $S^{j + 1}$  between $t_{i - 1}$ and $t_i$,
the probability that the three subevents above happen can be bounded by
taking the product of
 {Equation~\ref{eqn:mp2} and \ref{eqn:mp1} above.}

Let
\begin{equation}\label{eqn:fbnd}
f_i = \frac{3}{(t_i - t_{i - 1})^3}\left(\frac{d}{2\pi}\right)^d
\exp\left\{\frac{-3[\|D_{i - 1}-C_i\|^2_2 + \|C_{i -
      1}-C_i\|^2_2]}{2(t_i - t_{i - 1})}\right\} \mbox{ for } 1 \leq i
\leq j,
\end{equation}
\begin{equation}g_i = \frac 3
  {t^{1.5}_{i - 1}}\left(\frac{d}{2\pi}\right)^{d/2}\exp\left\{\frac{-3\|A_{i
        } -D_{i - 1}\|^2_2}{2 t_{i - 1}}\right\} \mbox{ for } 2 \leq i
\leq j,
\end{equation}
We also let $g_1 = 1$ and $g = \prod_{i \leq j}g_i$.
We have
 $$\Pr[\exists t_1 \leq ... \leq t_j, C_1, ..., C_j: \Psi^{t_1, ..., t_j}_{C_1, ..., C_j} = 1 ]
 \leq  \sum_{t_1, ..., t_j} \sum_{C_1, ..., C_j}\sum_{D_1, ...,
 {D_{j-1}}
}(f_1g_1)(f_2g_2)...(f_j g_j)$$


We now carefully bound this sum.  Observe that in Equation~\ref{eqn:fbnd},
when $t_i - t_{i - 1}$ is fixed and $\|D_{i - 1}-C_i\|^2_2 + \|C_{i -
 1}-C_i\|^2_2$ is sufficiently large, the quantity $f_i$ asymptotically
 becomes $$\exp(-\Theta(\max\{\|D_{i - 1}-C_i\|^2_2, \|C_{i - 1}-C_i\|^2_2\})).$$
This motivates us to group the triples $\{C_{i - 1}, C_i, D_{i - 1}\}$ together,
where the triples are covered by balls with approximately the same size under the $L_{\infty}$ norm . Specifically, we
let $\mathbb D_r$ be the set of triples $(A, B, C)$ where $A, B, C \in \mathbf Z^3$ and $\max\{\|A-B\|_1, \|A-C\|_1, \|B-C\|_1\}\leq r$.
Also, we say $\{A, B, C\} \in \partial \mathbb D_r$ if $\{A, B, C\} \in \mathbb D_r - \mathbb D_{r - 1}$. Notice
by telescoping, we have $\mathbb D_r = \bigcup_{i \leq r}\partial \mathbb D_i$.
We may thus group the variables $C_i$ and $D_i$ by parameterizing the
 {radii}
of the balls,
\begin{eqnarray*}
& & \Pr[\exists t_1 \leq ... \leq t_j, C_1, ..., C_j: \Psi^{t_1, ..., t_j}_{C_1, ..., C_j} = 1] \\
& \leq &  \sum_{t_1, ..., t_j} \sum_{C_1, ..., C_j}\sum_{D_1, ..., D_{j-1}}(f_1g_1)(f_2g_2)...(f_j g_j) \\
& = &  {\sum_{C_1\in \cV^3}}\sum_{r_1 \geq 0} \sum_{r_2 \geq 0}\ldots \sum_{r_{j - 1} \geq 0}
 \sum_{\substack{\{C_1, C_2, D_1\} \\\in \partial  \mathbb
     D_{r_1}}} \sum_{\substack{C_3, D_2: \\ \{C_2, C_3, D_2\} \\\in \partial
     \mathbb D_{r_2}}}\cdots
\sum_{\substack{C_j, D_{j - 1}: \\ \{C_{j - 1}, C_j, D_{j-1}\} \\\in \partial
    \mathbb D_{r_{j - 1}}}} \sum_{t_1 < \ldots < t_j}f_1\cdot  f_2
\cdot\ldots \cdot f_j  \cdot g
\end{eqnarray*}
First observe that by the triangle inequality $\|A-C_1\|_1 + \|B-C_1\|_1 \geq \|A-B\|_1 = x$, and for any vector $\vec v \in \mathbf R^3$,
\begin{equation}\label{eqn:l1l2}
\frac 1{\sqrt 3}\|\vec v\|_1 \leq \|\vec v\|_2 \leq \|\vec v\|_1.
\end{equation}
We have
\begin{equation}\label{eqn:a1bc1}
\|D_0-C_1\|^2_2 + \|C_0-C_1\|^2_2 = \|A-C_1\|^2_2 + \|B-C_1\|^2_2 \geq
\frac 1 3(\|A-C_1\|^2_1 + \|B-C_1\|^2_1) \geq \frac{x^2}6.
\end{equation}
Next, by the triangle inequality again,
 $\|D_i-C_{i + 1}\|_1 + \|C_i-C_{i + 1}\|_1\geq \|D_i-C_i\|_1$. Meanwhile, we have
$$\max\{\|D_i-C_{i + 1}\|_1, \|C_i-C_{i + 1}\|_1, \|D_{i}-C_i\|_1\} = r_{i}.$$
Together with the relationship between the $L_1$ and $L_2$ norms in Equation~\ref{eqn:l1l2},
we obtain
$$
\|D_i-C_{i + 1}\|^2_2 + \|C_i-C_{i + 1}\|^2_2  \geq r^2_i / 6 \quad \mbox{ for } 1 \leq i < j.
$$
Next, for $i\ge 2$ we define
$$\hat f_i = \frac 1 {(t_i - t_{i - 1})^3}\exp\left\{\frac{-r^2_{i - 1}}{4(t_i - t_{i
      -1})}\right\}, $$
and define
$$\hat g = \frac 1 {(t_1\ldots t_{j - 1})^{1.5}}.$$


It is clear that $f_i \leq \hat f_i$ for all $i \geq 2$. For notational convenience, we let $\hat f_1 = f_1$.


Our goal is now to bound the term
\begin{eqnarray*}
\eta
&\equiv &\sum_{r_1, \ldots, r_{j-1}}\sum_{\substack{\mbox{\tiny all
      triples}\\\{C_i, C_{i + 1}, D_i\}}}\sum_{t_1 \leq \ldots \leq
  t_j}\left(\prod_{i \leq j}\hat f_i\right) \cdot \hat g  \\
& =& \underbrace{\left(\frac d{2\pi}\right)^d}_{\mbox{from } \hat g} \cdot \sum_{r_1, ..., r_{j - 1}}
\sum_{\substack{\mbox{all } i < j: \\ \{C_i, C_{i + 1}, D_i\}}}\sum_{t_1}\hat f_1 \underbrace{\frac{1}{t^{1.5}_1}}_{\mbox{from } \hat g}
\sum_{t_2}\hat f_2 \underbrace{\frac{1}{t^{1.5}_2}}_{\mbox{from } \hat g}...
\sum_{t_{j - 1}}\hat f_{j - 1} \underbrace{\frac{1}{t^{1.5}_{j-1}}}_{\mbox{from } \hat g}\sum_{t_j}\hat f_j.
\end{eqnarray*}
Next, let us rearrange the indices and decompose
 the quantity into different parts (in terms of $\Upsilon_i$ defined below) and
 express $\eta$  as
{\small
\begin{equation}
\underbrace{\sum_{t_1 \geq 0} \sum_{C_1 \in  {\cV^3}}
\frac{\hat f_1}{t^{1.5}_1}\underbrace{\left(
\sum_{\substack{r_1 \geq 0 \\ t_2 \geq t_1}}
\sum_{\substack{C_2, D_1 \\ \{C_1, C_2, D_1\}
 \\ \in \partial \mathbb D_{r_1}}}\frac{\hat f_2}
 {t^{1.5}_2}\left(\underbrace{\sum_{\substack{r_2
 \geq 0\\t_3 \geq t_2}}\sum_{\substack{C_3, D_2\\
  \{C_2, C_3, D_2\}\\ \in \partial \mathbb D_{r_2}
   }}\frac{\hat f_3}{t^{1.5}_3}\left(...\underbrace{
\left(\sum_{\substack{r_{j - 2} \geq 0\\ t_{j - 1}
\geq t_{j - 2}}}\sum_{\substack{C_{j - 1}, D_{j -2}:\\\{C_{j-1},
  C_{j-2},\\ D_{j-2}\}  \in \\ \partial \mathbb
 D_{r_{j - 2}}}}\frac{\hat f_{j - 1}}{t^{1.5}_{j - 1}}
 \underbrace{\left(\sum_{\substack{r_{j - 1}\geq 0\\t_j
 \geq t_{j - 1}}}\sum_{\substack{C_{j}, D_{j - 1}:\\ \{C_{j},
 C_{j - 1}, \\D_{j - 1}\} \in \\ \partial \mathbb
 D_{r_{j - 1}}}}\hat f_{j}\right)}_{\Upsilon_{j}}\right)
 }_{\Upsilon_{j - 1}}\right)}_{\Upsilon_3}\right)\right)}_{\Upsilon_2}}_{
 \Upsilon_1}
\end{equation}
}
Let us briefly interpret the meaning of $\Upsilon_i$: this term describes an upper bound for the following two groups of events:
\begin{itemize}
\item the collisions between $S^{j + 1}$ and $S^i$, $S^{i + 1}$, ..., and $S^{j}$ at time $t_i$, $t_{i + 1}$, ..., $t_j$ respectively.
\item the fact that at time $t_{i'}$ the walk $S^{i'+1}$ is at $D_{i'}$ for all $i \leq i' \leq j$ (i.e., $S^{i'+1}_{t_{i'}} = D_{i'}$).
\end{itemize}
which is conditioned on knowing the values for
$$\Im = \{t_1, ..., t_{i - 1}, r_1, ..., r_{i - 1}, C_1, ..., C_{i - 1}, D_1, ..., D_{i - 2}\}.$$
When $C_{i - 1}$ is known, this information imposes a constraint over
the way to enumerate $C_i$ and $D_{i - 1}$ because we require $\{C_{i - 1}, C_i, D_{i - 1}\} \in \partial \mathbb D_{r_{i - 1}}$ for a specific $r_{i - 1}$. Therefore, the computation of $\Upsilon_i$ depends on the value of $C_{i - 1}$. A second constraint imposed from knowing $\Im$ is that we need $t_{i - 1} \leq t_{i} \leq ... \leq t_{j}$. $\Upsilon_i$ does not depend on other values in $\Im$.
In what follows, we write $\Upsilon_i$ as a function of $C_{i - 1}$ and $t_{i - 1}$.



Specifically, let us define the function $\Upsilon_i$ in a forward recursive manner (the summations of $r_i$ and $t_i$ are over integers):
\begin{equation}
\Upsilon_i = \left\{
\begin{array}{ll}
\sum_{r_{j - 1} \geq 0} \sum_{\substack{C_j, D_{j - 1}:\\ \{C_{j - 1}, C_j, D_{j - 1}\} \\ \in \partial \mathbb D_{r_{j - 1}}}}\sum_{ t_{j - 1} \leq t_j \leq x^2}\hat f_j & \mbox{ if $i = j$ (base case)} \\
\sum_{r_{i - 1} \geq 0}\sum_{\substack{C_i, D_{i - 1}:\\\{C_{i - 1}, C_i, D_{i - 1}\} \\ \in \partial \mathbb D_{r_{i - 1}}}}\sum_{t_i \geq t_{i - 1}}\left( \hat f_i \frac 1{t^{1.5}_i}\Upsilon_{i + 1}\right) & \mbox{ if $1 < i < j$}\\
\sum_{t_1, C_1} \frac{\hat f_1}{t^{1.5}_1}\Upsilon_2 & \mbox{if $i = 1$.}
\end{array}
\right.
\end{equation}
%

The variable $\Upsilon_1$
is the quantity we desire to bound.
Let $\Delta t_i = t_i - t_{i - 1}$ for all $i$ (and we shall let $t_0 = 0$).
Let us start with bounding
 $$\Upsilon_j
=\sum_{r_{j -
     1}}\sum_{C_j, D_{j - 1}}\sum_{\Delta t_j > 0}\frac{3}{\Delta
   t^3_j} \exp\left(\frac{-r^2_{j - 1}}{4\Delta t_j}\right)$$

We shall first find the total number of $\{C_j, D_{j - 1}\}$ pairs so that $\{C_j, C_{j - 1}, D_{j - 1}\} \in \partial \mathbb D_{r_{j - 1}}$.
Notice that when $r_{j - 1}$ and $C_{j - 1}$ are fixed, at least one of $\|C_{j - 1}-C_j\|_1$, $\|C_{j - 1}-D_{j - 1}\|_1$, and
$\|C_j- D_{j - 1}\|_1$ is exactly $r_{j - 1}$. When $\|C_{j - 1}-D_{j - 1}\|_1 = r_{j - 1}$, the number of possible $D_{j - 1}$ is $4r_{j - 1}(r_{j - 1} - 1) \leq 4r^2_{j - 1}$. An upper bound on the number of possible $C_{j}$ is $4r^3_{j - 1}$. Therefore, when $\|C_{j - 1}-D_{j - 1}\|_1 = r_{j - 1}$, the number of $\{C_j,D_{j - 1}\}$ pairs is at most $16r^5_{j - 1}$. We may similarly analyze the other two cases to find that
the total number of $\{C_j, D_{j - 1}\}$ pairs is at most $48r^5_{j -
  1}$. Thus, we have

\begin{eqnarray*}
 \Upsilon_j
& = & \sum_{r_{j - 1}}\sum_{C_{j}, D_{j - 1}} \sum_{\Delta t^3_j}\frac{3}{\Delta t^3_j}\exp\left(\frac{-r^2_{j - 1}}{4 \Delta t_j}\right) \\
& = & \sum_{\Delta t_j} \frac{1}{\Delta t^3_j}\left(\sum_{r_{j - 1}}3 \times 48 {r}^5_{j - 1}\exp\left(\frac{-r^2_{j - 1}}{4 \Delta t_j}\right)\right) \\
& \leq & \sum_{\Delta t_j} \frac{1}{\Delta t^3_j}\left(2\cdot
  \int_0^{\infty}144r^5_{j - 1}\exp\left(\frac{-r^2_{j - 1}}{4\Delta
      t_j}\right)\mathrm dr_{j-1}\right)\\
& = & \sum_{\Delta t_j}\frac{1}{\Delta t^3_j}18432\Delta t^3_j \\
& = & 18432x^2 \leq \zeta_0x^2,
\end{eqnarray*}
where $\zeta_0 =  \zliu{18432}$. The last equality holds because we are
considering a time frame of length $x^2$ and therefore $\Delta t_j
\leq x^2$.
%
Let us explain the derivation in greater detail because similar techniques will be used again in the rest of the analysis.
Define $h(x) = x^5 \exp\left(-\frac{x^2}{4\Delta t_j}\right)$. The function $h(x)$ is a unimodal function with a unique global maximal value.
Let $x_0 = \mathrm{arg} \inf_{x \geq 0}h(x)$. Then we have
\begin{eqnarray*}
\sum_{x \in \mathbf N}h(x)
& \leq & \sum_{x = 1}^{\lfloor x_0\rfloor}h(x) + \sum_{x = \lceil x_0\rceil }^{+\infty} h(x)  \\
& \leq & \int_{1}^{x_0 + 1}h(x) \mathrm dx + \int_{\lfloor
  x_0\rfloor}^{\infty}h(x)\mathrm dx \\
& \leq & 2\int_0^{\infty}h(x)\mathrm dx.
\end{eqnarray*}

While this bound is quite rough,
%
it suffices
 for our purpose; the same approach is used to bound the summation of unimodal functions
 elsewhere. The third equality holds because
  of the following fact,


\begin{equation}\label{eqn:integrate}
\int_{0}^{\infty}x^5 \exp\left(-\frac{x^2}{4 \ell }\right)\mathrm dx = 64\ell ^3
\end{equation}
for any $\ell$. (This can be verified through standard software packages such as Mathematica).
%
%

We can prove the following hypothesis for $\Upsilon_i$:
$$\mbox{for all }  {1 \leq }\ell \leq j - 2: \Upsilon_{j - \ell}(C_{j - \ell -
  1}, t_{j-\ell-1}) \leq x^2\zeta_0^{\ell + 1}
\cdot 4^{\ell} \cdot \frac{1}{\ell!}\cdot t^{-\ell / 2}_{j - \ell - 1}.$$
We shall show this by induction (with the base case, in which $\ell = 0$, being proven above).

\begin{eqnarray*}
& & \Upsilon_{j - \ell - 1}(C_{j - \ell - 2}, t_{j-\ell-2}) \\
& \leq & x^2\zeta_0^{\ell + 1} \cdot 4^{\ell - 1}\cdot (\ell!)^{-1}
 {\sum_{r_{j-\ell-2}}
} \sum_{C_{j - \ell - 1}, D_{j - \ell - 1}}\sum_{\Delta t_{j - \ell - 1}}\frac{3}{\Delta t^3_{j}}\exp\left(\frac{-r^2_{j - \ell - 2}}{4\Delta t_{j - \ell - 1}}\right)\frac{1}{t^{1.5 + \ell / 2}_{j - 1}}
\\
& \leq & x^2\zeta_0^{\ell + 1} \cdot 4^{\ell -
  1}(\ell!)^{-1}\sum_{\Delta t_{j - \ell - 1}}\frac{1}{\Delta t^3_{j -
    \ell - 1}\cdot t^{1.5 + \ell / 2}_{j -
 {\ell-}
1}} \left(2 \int_0^{\infty}144r^5_{j - \ell - 2}
\exp\left(\frac{-r^2_{j - \ell - 2}}{4 \Delta t_{j - \ell -
      1}}\right)\mathrm dr_{j - \ell - 2}\right) \\
& = & x^2\zeta_0^{\ell + 2}4^{\ell}(\ell!)^{-1} \sum_{\Delta
  t_{j - \ell - 1}}\frac{1}{t^{1.5 + \ell / 2}_{j - \ell - 1}} \\
& \leq & x^2\zeta_0^{\ell + 2}4^{\ell + 1} ((\ell + 1)!)^{-1}\frac{1}{t^{(\ell + 1)/2}_{j - \ell -   {2}}}.
\end{eqnarray*}
The last inequality holds because


$$\sum_{\Delta t_{j - \ell - 1}}\frac{1}{t^{1.5 + \ell / 2}_{j -\ell -
    1}}  \leq 2 \int_{t_{j - \ell -  {2}}}^{\infty}\frac{1}{t^{1.5 +
    \ell / 2}}\mathrm dt
\leq  4(\ell + 1)^{-1}\frac{1}{t^{(\ell + 1)/2}_{j -\ell -  {2}}}.$$
This completes the induction. Finally, we have
\begin{eqnarray*}
 \Upsilon_1  & = & \sum_{C_1}\sum_{t_1}f_1 \frac{1}{t^{1.5}_1}\Upsilon_2(C_1,t_1) \\
& \leq & x^2 \zeta_0^{j {-1}}4^{j -  {2}}((j -  {2})!)^{-1}\sum_{C_1, t_1}\frac 1{t^{j/2+3.5}_1}\cdot \exp\left(\frac{-3\|B-C_1\|^2_{ {2}} + \|A_1-C_1\|^2_{ {2}}}{2t_1}\right).
\end{eqnarray*}



Next, let $\|B-C_1\|_1 = r_0$.  {By Equation~\ref{eqn:l1l2},
$\|B-C_1\|_2^2 \geq r_0^2/3$. By Equation~\ref{eqn:a1bc1}, we
have $\|B-C_1\|^2_2 + \|A_1-C_1\|^2_2 \geq x^2/6$.} Therefore, we have
$\|B-C_1\|^2_2 + \|A_1-C_1\|^2_2  \geq
\frac 1 6 (r^2_0 + x^2/2)$. We have
\begin{eqnarray*}
& & \sum_{C_1, t_1}\frac{1}{t^{j/2 + 3.5}_1}\exp\left(\frac{-3(\|B-C_1\|^2 + \|A_1-C_1\|^2)}{2t_i}\right) \\
& \leq & \sum_{t_1}\frac{1}{t^{j/2 + 3.5}_1}\sum_{\|B-C_1\|_1 = r_0}\sum_{C_1}\exp\left(\frac{-3r^2_0}{12t_1}\right)\cdot \exp\left(\frac{-3x^2}{24t_1}\right) \\
& \leq & \sum_{t_1}\frac{1}{t^{j / 2 +
    3.5}_1}\left(2\int_{0}^{\infty}4r^2_0\exp\left(\frac{-r^2_0}{4t_1}
    \right) \mathrm dr_0\right)\cdot \exp\left(\frac{-3x^2}{24t_1}\right)\\
& \leq & \sum_{t_1}\frac{1}{t_1^{j / 2 + 3.5}}16\sqrt{\pi}t_1^{1.5} \exp\left(\frac{-x^2}{8t_1}\right) \\
& \leq & 30 \sum_{t_1}\frac{1}{t_1^{j / 2 + 2}} \exp\left(\frac{-x^2}{8t_1}\right) \\
& \leq & 60\frac{\Gamma(j/2 + 1)}{(x^2/8)^{j/2 + 1}}.
\end{eqnarray*}


The third inequality holds because
$$\int_{0}^{\infty}r^2_0\exp\left(-\frac{r^2_0}{4t_1}\right)\mathrm dr_0 = 2\sqrt{\pi}t_1^{1.5}$$


The last inequality holds because
$$\int_{0}^{\infty}y^{-c}\exp(\frac{-x^2}{8y}) {\mathrm dy} = \frac{8^{c-1}}{x^{2(c-1)}}\Gamma(c - 1)$$
for any constant $c$ and real number $x$.

We thus conclude that
$$\Upsilon_1  \leq  \frac{x^2\zeta^{j {-1}}_04^{j- {2}}}
{(j- {2})!}\cdot  {60}\cdot\frac{\Gamma(\frac j 2 + 1)}{x^{2(j/2+1)}}8^{j/2+1}
\leq  \frac{ {30}(8\sqrt 2)^j\zeta_0^{j- {1}}\Gamma(\frac j 2 + 1)}{(j- {2})!x^j}.
$$
When the permutation is considered, we have
\begin{eqnarray*}
& & \Pr[\exists t_1, ..., t_j, C_1, ..., C_j \Psi^{t_1, ...,
  t_j}_{C_1, ..., C_j} = 1] \\
& \leq & j!\frac{ {30}(8\sqrt
  2)^j\zeta_0^{j- {1}}\Gamma(\frac j 2 + 1)}{(j- {2})!x^j}\\
& \leq & \frac{ {30}j(j-1)(8\sqrt
  2)^j\zeta_0^{j- {1}}\Gamma(\frac j 2 + 1)}{x^j}\\
& \leq & \frac{(8\sqrt 2\zeta_0)^jj^j}{x^j}\\
& \leq & \left(\frac{8\sqrt 2\zeta_0j}{x}\right)^j
\end{eqnarray*}
By setting $\zeta = 8 \sqrt 2 \zeta_0 < 210000$, our lemma follows.
\end{proof}

\begin{proof}[Proof of Corollary~\ref{cor:rwcatchall}]
Notice first that if $S^{j + 1}(B)$ and $S^i(A_i)$ meet at a time step $t_0$, then there exists a point $A'_i$ with $\|A'_i-A_i\|_1 = 1$ such that the walk $S^{i'}(A'_i)$ that mimics the moves of $S^i$ at each step collides with $S^{j + 1}$ at time $t_0$.

Therefore, a necessary condition for $S^{j + 1}$ to meet the rest of
agents is  {that there exist} $S^{1'}(A'_1)$, $S^{2'}(A'_2)$, ..., $S^{j'}_{A'_j}$ such that
\begin{itemize*}
\item $\|A'_i-A_i\|_1 = 1$ for all $i \leq j$.
\item $S^{i'}$ mimics the moves of $S^i$ at all steps for all $i \leq j$.
\item $S^{j + 1}$ \emph{collides} with all of $S^{1'}$, ..., $S^{j'}$ before time $t$.
\end{itemize*}
 {For any} $A'_1$, ..., $A'_j$, the collision probability is at most $\left(\frac{\zeta j}{x}\right)^j$ by Lemma~\ref{lem:rwcatchall}.
The total number of possible $j$-tuples $A'_1$, ..., $A'_j$ is $7^j$. By using a union bound, the probability there exists a $j$-tuple such that all $j$ walks collide with $S^{j + 1}$ is at most $7^j \left(\frac{\zeta j}{x}\right)^j$. The corollary follows.
\end{proof}

 \zliu{We next move to prove Lemma~\ref{lem:twowalkboundary}}. Since we need
to frequently compare bounded random walks with their unbounded
counterparts, we use $\mathbf S$ to represent unbounded walks and $S$ to represent
bounded walks in the rest of this section.

 \zliu{Our analysis consists of two steps.
We first tackle a simpler problem, in which we need to understand the probability for
a random walk starting from a point near the boundary to visit another point in $\mathcal V^3$
within a short time frame. We then utilize results from this scenario to
 prove Lemma~\ref{lem:twowalkboundary}.}

\begin{lemma}\label{lem:bndonewalk}Let $\mathcal V^3 = \{-n, ..., n\}^3$. Let $A$ and $B$
be two points in $\mathcal V^3$ such that $A-B = \vec x$
and the distance under $L_{\infty}$ norm between
$A$ and any boundary is at least $20\|\vec x\|_1$. Consider a random walk $S(A)$ that
starts at $A$. Let $e^1_t$ be the event that $S(A)$ is at
 $B$ at time $t$. Let $e^2_t$ be the event that $S(A)$ hits a boundary
at or before $t$. When $t =\Theta( \|\vec x\|_1^2)$, we have
$\Pr[e^1_t \wedge \neg e^2_t] \geq c_0 p(t, \vec x)$ for some constant $c_0$.
\end{lemma}

\begin{proof} First, let us couple the random walk $S(A)$ with
a standard unbounded random
walk $\mathbf S(A)$ in the natural way. Let $\hat e^1_t$ be the event
that $\mathbf S(A)$ is at $B$ at time $t$ and let $\hat e^2_t$ be the
event that $\mathbf S(A)$ ever visits
a boundary at or before time $t$.
When $\neg e^2_t$ occurs, $\mathbf S(A)$
and $\mathrm S(A)$ coincide and $\Pr[e^1_t \wedge \neg e^2_t] = \Pr[\hat e^1_t \wedge \neg \hat e^2_t]$.

On the other hand, we have
$$p(t, \vec x) = \Pr[\hat e^1_t \wedge \hat e^2_t] + \Pr[\hat e^1_t \wedge \neg \hat e^2_t].$$
Notice that in the event $\hat e^1_t \wedge \hat e^2_t $,
$\mathbf S(A)$ has to travel from the boundary to $B$ within a time interval shorter than $t$.
The distance between the boundary and $B$ is at least $19\|\vec x\|_1$.
 Together with the analytic form of $p(\cdot, \cdot)$ in Lemma~\ref{thm:standardrw}, we have $\Pr[\hat e^1_t \wedge \hat e^2_t = 1] \leq \max_{\|\vec y\|_1 \geq 19\|\vec x\|_1}p(t, \vec{y})$. Therefore,
 $\Pr[\hat e^1_t \wedge \neg \hat e^2_t] \geq p(t, \vec x) - \max_{\|\vec y\|_1 \geq 19\|\vec x\|_1} p(t, \vec{y})$.
 Finally, we have
 $$\Pr[e^1_t \wedge \neg e^2_t] = \Pr[\hat e^1_t \wedge \neg \hat e^2_t] \geq p(t, \vec x) - \max_{\|\vec y\|_1 \geq 19\|\vec x\|_1} p(t, \vec{y}) \geq \frac 1 2 p(t, \vec x).$$
We may use the analytic form of the function $p(\cdot, \cdot)$ (Lemma~\ref{thm:standardrw})
for $t =\Theta( \|\vec x\|_1^2)$ to verify the last inequality.
\end{proof}

\begin{proof}[Proof of Lemma~\ref{lem:twowalkboundary}]
Let $X$ be the number of collisions between
$S^1$ and $S^2$ that are before time  
 {$t$} and before either of them visits a boundary.
Also let $\hat e(S_t)$ be the event that the random walk
 $S$ ever visits a boundary at or before time $t$.
We have
\begin{eqnarray*}
\E[X] & = & \sum_{t\leq \|\vec x\|_1^2}\Pr[(S^1_t = S^2_t) \wedge (\neg \hat e(S^1_t) \wedge \neg \hat e(S^2_t)] \\
& = & \sum_{t \leq \|\vec x\|_1^2}\sum_{C \in \mathcal V}\Pr[(S^1_t = C) \wedge \neg \hat e(S^1_t)]
\Pr[(S^2_t = C) \wedge \neg \hat e(S^2_t)] \quad \mbox{(two walks are independent)} \\
& \geq & \sum_{t \leq \|\vec x\|_1^2}\sum_{C: \|C - A\|_1 \leq \|\vec x\|_1}\Pr[(S^1_t = C) \wedge
\neg \hat e(S^1_t)]\Pr[(S^2_t = C) \wedge \neg \hat e(S^2_t)] \quad \mbox{(only focus on a subset of $\mathcal V^3$)} \\
\end{eqnarray*}
Since $\|C-A\|_1 \leq \|\vec x\|_1$ and $\|A-B\|_1 \leq \|\vec x\|_1$, we have $\|C-B\|_1 \leq 2\|\vec x\|_1$.

By Lemma~\ref{lem:bndonewalk},
$$\Pr[(S^1_t = C) \wedge
\neg \hat e(S^1_t)] \geq \frac 1 2 p(t, A-C)  \quad \mbox{ and } \Pr[(S^2_t = C) \wedge \neg \hat e(S^2_t)]\geq \frac 1 2 p(t, B-C) $$
We now have
\begin{eqnarray*}
& & \sum_{t \leq \|\vec x\|_1^2}\sum_{C: \|C - A\|_1 \leq \|\vec x\|_1}\Pr[(S^1_t = C) \wedge
\neg \hat e(S^1_t)]\Pr[(S^2_t = C) \wedge \neg \hat e(S^2_t)] \\
& \geq & \sum_{1 \leq t \leq \|\vec x\|_1^2}\sum_{C: \|C - A\|_1 \leq \|\vec x\|_1}\frac 1 2p(t, A-C)\frac 1 2p(t, B-C) \quad \mbox{(by Lemma~\ref{lem:bndonewalk})} \\
& = & \Omega\left(\sum_{1 \leq t \leq \|\vec x\|_1^2}\|\vec x\|_1^3\min_{C:  \|C - A\|_1 \leq \|\vec x\|_1}\left\{p(t, C - A)p(t, B - C)\right\}\right) \\
& = & \Omega(1/\|\vec x\|_1).
\end{eqnarray*}
The last equality can be shown by using the analytic form of $p(\cdot, \cdot)$ again (Lemma~\ref{thm:standardrw})
and the fact that $\|A - C\|_2$ and $\|B - C\|_2$ are in $O(\|\vec x\|_2)$.

Next, let us compute $\E[X|X \geq 1]$, i.e., the expected number of collisions when they collide at least once.
Upon the first time $S^1$ and $S^2$ collide (before either of them visit the boundary), we couple $S^1$ and $S^2$
with two unbounded random walks $\mathbf S^1$ and $\mathbf S^2$ in the natural way respectively.
The expected number of collisions between $\mathbf S^1$ and $\mathbf
S^2$ for
 {$t$} steps
(when they start at the same point)
is an upper bound on $\E[X|X \geq 1]$. On the other hand, we may couple $\mathbf S^1$ and $\mathbf S^2$ with a single random walk $\mathbf S$
in the way described in Lemma~\ref{lem:bndonewalk} so that the expected number of collisions between $\mathbf S^1$ and $\mathbf S^2$
is the expected number of times $\mathbf S$ returns to the point where it starts at.

Finally, the expected number of return for an unbounded random walk is a constant (which can be derived from
$\sum_{t \geq 0}p(t, \vec{0})$, where $p(\cdot, \cdot)$'s analytic
form is in Theorem~\ref{thm:standardrw}
).
 Therefore,
$\E[X\mid X \geq 1] = O(1)$. Now since $\E[X] = \E[X \mid X \geq 1]\Pr[X \geq 1]$. Therefore, $\Pr[\tilde e_{\|\vec x\|_1^2}] =
\Pr[X \geq 1] = \Omega(1/\|\vec x\|_1)$.

\end{proof}

\section{Missing proofs for upper and lower bound analysis}\label{sec:missingproofs}
\begin{proof}[Proof of Lemma~\ref{lem:3dgoodevent}]
We first show the good density property holds with high probability.
For any specific time $t\leq n^{2.5}$, all the agents are uniformly
distributed due to stationarity. For an arbitrary $P \in \mathcal{V}^3$, and $i \leq (\ell_2 / \ell_1)\log^{-3}n$, define
$Y(t, P, i)$ as the number of agents that are in $\partial
\mathcal B_i(P)$ at time $t$.
Notice that $\E[Y(t, P, i)] = |\partial \mathcal B_i(P
)|m/(2n+1)^3$ and $m_i(P) \geq 6\E[Y(t, P, i)]$. By Chernoff bounds (e.g., the second part of Theorem~\ref{thm:chernoff}),
$$\Pr[Y(t, P, i) \geq m_i]\leq 2^{-m_i}\leq \exp(-0.65m_i) \leq
\exp\left(0.65 \frac{\ell^3_1\log^{-3}n}{n^3}m\log^5n\right)\leq
\exp(-0.65\log^2n)$$
for sufficiently large $n$.
Next, by a union bound,
$$\Pr[D_t = 0] \leq \sum_{t,P,i}\Pr[Y(t,P,i) \geq m_i] \leq \left(n^{2.5}(2n+1)^3(\log^{-3}n)\ell_2/\ell_1\right)\exp(-0.65\log^2n) \leq \exp(-\frac 1 2\log^2n).$$
Therefore, we have $\pr[D_t =1] \geq 1-\exp(-\frac 1 2\log^2n)$.

To show the diffusion process has the small islands property with high probability,
we mimic the proof of Lemma 6 in \cite{PPPU11}. Let $B_k$
  be the event that there exists an island with parameter $\gamma=\ell_1\log^{-1}n$ that has at least $k$ agents. The
  quantity $\Pr[B_k]$ is upper bounded by the probability that
  $G_t(\gamma)$ contains a tree of $k$ vertices of $A$ as a
  subgraph. Since $k^{k - 2}$ is the number of unrooted labeled trees on
  $k$ nodes, and $\gamma^3 / n^3$ is an upper bound to the probability
  that a given agent lie within distance $\gamma$ from another given
  agent, we have that
$$\Pr[B_k] \leq \binom{m}{k}k^{k -
  2}\left(\frac{\gamma^3}{n^3}\right)^{k - 1} \leq
\left(\frac{em}{k}\right)^k
\cdot k^{k - 2}\left(\frac{1}{m \log^3n}\right)^{k - 1} = \frac{e^km}{
  k^2} \cdot (\log n)^{-3(k -1)}.
$$
By setting $k = 3\log n + 1$, we have $\Pr[B_k] \leq \exp\{-7 \log n
\cdot \log \log n\}$. Finally, we apply a union bound across all
agents and all time steps. Hence $\Pr[E_t=1 ] > 1-n^{2.5}
m
\exp(-7 \log n
\log \log n)$.

Finally, consider the short travel distance property.  For any fixed
$i \in [m]$ and $t_1 \leq t_2 \leq n^{2.5}$ such that $t_2 - t_1 \leq
\ell^2_2\log^{-12}n$, we have $\Pr[\|S^i_{t_1} - S^i_{t_2}\|_{\infty}
\geq \ell_2\log^{-4}n] \leq \exp(-\log^2n)$ by
Lemma~\ref{lem:basicmove}. There is a factor of 3 lost when we
translate the metric from $L_{\infty}$-norm to $L_1$-norm.  The total
number of possible $i$, $t_1$, and $t_2$ are $mn^5$.  Next we may
apply a union bound across all these possible $i$, $t_1$, and $t_2$
triples. We have $\pr[L_t=0]\leq
mn^5\exp(-\log^2n)$.  The lemma follows by combing the three results
together with one more union bound.
\end{proof}

\begin{proof}[Proof of Lemma~\ref{lem:densitya3d}]
Fix a time $t$ and let $\tilde{m}$ be the number of agents in an arbitrary subcube of size $\hat \ell_2 \times \hat \ell_2 \times \hat \ell_2$. We have
$\E[\tilde{m}] \geq
\hat \ell_2\log^2n/27 \geq \log^2n$. Therefore, by Chernoff bounds (Theorem~\ref{thm:chernoff}),
$\Pr[\tilde{m} \in [\frac 1 2 \E[\tilde{m}], \frac 3 2 \E[\tilde{m}]] \geq 1 - 2\exp(-\log^2n / 12)$.
Now the total number of
possible
subcubes is at most $(2n+1)^3$ and the total number of time
steps is $n^{2.5}$. By a union bound, we have
$$\Pr[D = 0] \leq (2n+1)^{3}\cdot n^{2.5}\cdot 2\exp(-\frac 1{12}\log^2n) \leq \exp(-\frac 1 {15}\log^2n)$$
for sufficiently large $n$.
\end{proof} 
\newcommand{\nnewxiaorui}[1]{{{#1}}}

\section{Near optimal bounds for the isoperimetric problem for closed
  hypercubes}
\label{sec:geometry}
This section studies an isoperimetric problem we need for our upper bound analysis. In what follows, we let
$\bC^d = \{1, 2, ..., b\}^d$, where $b$ is an arbitrary integer.

Our first lemma builds up a matching between the subcubes in the interior surface and
those in the exterior surface. This result allows us to focus on one type of surface for the purpose of understanding the completion time for the diffusion process.

\begin{lemma}\label{lem:matchingbnd}
Let $\cG$ be an arbitrary subset of $\bC^d$. Define $\dot \partial \cG$ and $\partial \cG$ as the interior and exterior surfaces of $\cG$ (i.e. the set of points in $\cG$ that neighbor with $\cG^c$ and the set of points in $\cG^c$ that neighbor with $\cG$ resp.; $\vec u$ and $\vec v$ are neighbors if $\|\vec u- \vec v\|_1=1$). Define a bipartite graph with nodes denoting $\dot \partial \cG$ and $\partial \cG$, in which an edge $(\vec u, \vec v), \vec u\in\dot \partial \cG,\vec v\in\partial \cG$ exists whenever $\vec u$ and $\vec v$ are neighbors. Then there exists a matching $M$ in this graph with $|M| \geq |\partial \cG|/(4\nnewxiaorui{d}-1)$.
\end{lemma}
\begin{proof}
We prove this statement by explicitly constructing the matching
$M$. First notice that the degree of each node is in the range
$[1,2\nnewxiaorui{d}]$. We build $M$ iteratively. Each time, we pick an edge $(\vec u, \vec v)\in E$
and place the edge into $M$. We then remove nodes $\vec u, \vec v$ from $L$ and
$R$ respectively as well as all edges incident to them. Since the
degrees of $\vec u, \vec v$ are bounded by $2\nnewxiaorui{d}$, we will remove at most $4\nnewxiaorui{d}-1$
edges from $E$. We continue this process until no edge is
left. Clearly, the edges we place into $M$ form a matching. Because
there are at least $|\partial \cG|$ number of edges by the lower bound
of degrees, we conclude that $|M| \geq \frac{|\partial \cG|}{4\nnewxiaorui{d}-1}$.
\end{proof}

\begin{theorem} \label{thm:surface}
Let $\cG$ be an arbitrary subset of $\nnewxiaorui{\bC^d}$. There exists a pair of
constants $\nnewxiaorui{\alpha(d)}>1/2$ and $\nnewxiaorui{\beta(d)}>0$, such that:
\begin{eqnarray*}
\mbox{if } |\cG| \leq \nnewxiaorui{\alpha(d)}\cdot |\nnewxiaorui{\bZ^d}| =
\nnewxiaorui{\alpha(d)}\cdot \nnewxiaorui{b^d}, & & \mbox{then }| \partial \cG| \geq \nnewxiaorui{\beta(d)} {|\cG|}^{\nnewxiaorui{(d-1)/d}}
\end{eqnarray*}
Specifically, $\beta(3) \geq 0.36$.
\end{theorem}

\liusd{The isoperimetric problem over $\bC^d$ was studied in \cite{BL91}, in which the optimal structure of $\cG$ that minimizes $|\partial \cG|$ is presented. Here, we provide another asymptotically optimal proof based on a recursive argument. This proof could be of independent interest.}

To begin, let us prove the special case $d = 2$. The analysis
for this case demonstrates important ideas that are needed for showing the case for general $d$.

\begin{lemma}\label{lem:2dgeom}
Let $\cG$ be an arbitrary subset of $\bC^2$. If $|\cG| \leq \frac{2}{3} b^2$, we have
$$ |\partial \cG| \geq \frac{2}{5} |\cG|^{1/2}.$$
\end{lemma}
\begin{proof}
Let $V=|\cG|$ and $X(i)$ be the collection of lattice points in $\bC^2$ whose $x$ coordinates are $i$. Also we refer $V(i):=X(i) \cap \cG$ as the $i$th \emph{stripe} of $\cG$. Define
$$i^* = \arg
\max_i |V(i)| \quad  \mbox{ and } \quad i_* = \arg\min_i |V(i)|.$$

We next analyze two possible cases regarding the size of $V(i^*)$.

{\noindent \emph{Case 1.}} $0 <|V(i^*)| < \sqrt{\frac{3V}{2}}$.
 Since $\sqrt{\frac{3V}{2}} \leq b$, for each $i$ such that $V(i) \neq \emptyset$, we have
$$0 < |V(i)| \leq |V(i^*)|< b.$$
 \zliu{On the other hand, when $0 < |V(i)| < b$, there is at least one element of $X(i)$ that is also in $\partial \cG$.
Since the
 cardinality  of $\cG$ is $V$,  the number of non-empty
stripes in $\cG$ is at least $\frac{V}{|V(i^*)|}$. Hence
 we have}
$$|\partial \cG| \geq \frac{V}{|V(i^*)|} \geq
\sqrt{\frac{2V}{3}} > \frac{2}{5}\sqrt V$$

{\noindent \emph{Case 2.}} $|V(i^*)| \geq \sqrt{\frac{3V}{2}}$.
 \zliu{By an averaging argument, $|V(i_*)| \leq V/b$. Using the fact that $V \leq \frac 2 3b^2$, we have $|V(i_*)| \leq \sqrt{\frac 2 3 V}$.

Next we show that $\partial \cG \geq |V(i^*)|-|V(i_*)|$. Consider an arbitrary $j$ such that $(i^*, j) \in V(i^*)$ and
$(i_*, j) \notin V(i_*)$. Since $(i^*, j) \in \cG$ and $(i_*,j)\notin \cG$, there exists
 a lattice point on the ``line segment''
 $\{(i, j): i \in \{i^*, ...,  i_*\}\}$ that is in $\partial \cG$.

Finally, we have
$$\partial \cG \geq |V(i^*)|-|V(i_*)| \geq \left(\sqrt{\frac 3 2} - \sqrt{\frac 2 3 }\right) \sqrt V \geq \frac 2 5 \sqrt V.$$
}
\end{proof}
 \zliu{
We use induction to prove Theorem~\ref{thm:surface}.
Our idea of proving general $d$ is similar to the case $d = 2$.
First, we let $X(i)$  be the collection of lattice points in $\bC^d$ whose first coordinates are $i$ and $V(i) = X(i)\cap
\cG$. Next, we also define $i^* = \arg
\max_i |V(i)|$ and $i_* = \arg\min_i |V(i)|$. Then, we mimic the analysis for the case $d = 2$ and discuss two
possible cases: when $|V(i^*)|$ is small and when $|V(i^*)|$ is large. When $|V(i^*)|$, we need to invoke results
on lower dimension cases. When $|V(i^*)|$ is large, we shall show that $|V(i^*)| - |V(i_*)|$ is a
lower bound on the size of $\partial G$, which is sufficient for proving the theorem. }

Let us proceed with the following lemma, which is the main vehicle for analyzing the case $|V(i^*)|$
is large.

\begin{lemma}
\label{lem:boundary-lowerbound}
 \zliu{
Let $\cG$ be an arbitrary subset of $\bC^d$. We have
$$|\partial \cG| \geq |V(i^*)| - |V(i_*)|.$$
}
\end{lemma}
 \zliu{
\begin{proof}First, define the set $\Delta$ as
$$\Delta = \left\{(i_2, i_3,..., i_d) \in \mathfrak C^{d - 1}\Big| \left((i^*, i_2, i_3,..., i_d) \in V(i^*)\right) \wedge\left((i_*, i_2, i_3,..., i_d) \notin V(i_*)\right)\right\}$$
Notice that by the definitions of
$V(i^*)$ and $V(i_*)$, we have $|\Delta|\geq |V(i^*)| - |V(i_*)|$.
Next, we show that for any $(i_2, ..., i_d) \in \Delta$, there exists an $i_1$ such that $(i_1, ..., i_d) \in \partial \mathcal G$, which immediately implies the lemma.

Fix a $(d - 1)$-tuple $(i_2, ..., i_d) \in \Delta$. Observe that
$(i^*,i_2, ..., i_d) \in V(i^*) \subseteq \mathcal G$ and
$(i_*, i_2, ...., i_d) \notin V(i_*)$ and thus
$(i_*, i_2, ...., i_d) \notin \mathcal G$. Let us walk from the point $(i^*,i_2, ..., i_d)$ to the point $(i_*,i_2, ..., i_d)$. Because we start with an interior point of $\mathcal G$ and end at a point outside $\mathcal G$, we leave the polytope $\mathcal G$ at least once. Hence, there exists an $i_1$ such that $(i_1, ..., i_d) \in \partial G$.
\end{proof}
}
%

Now we are ready to prove the main theorem.
\begin{proof}[Proof of Theorem~\ref{thm:surface}] We prove by induction on $d$. Specifically, we show that for any $d$ and any $\cG(d) \subseteq \bC^d$, there exists a pair of constants (that depends only on $d$) $\alpha(d) \geq 1/2$ and $\beta(d) > 0$ such that
\begin{eqnarray*}
\mbox{if } |\cG| \leq \alpha(d) |\bC^d|, & & \mbox{then } |\partial \cG(d)| \geq \beta(d)|\cG|^{d - 1}.
\end{eqnarray*}
The base case was considered in Lemma~\ref{lem:2dgeom}. Now let us assume the theorem holds
up to the $d$-dimensional space. We now prove the $d + 1$ dimensional case.

Our $\alpha(d + 1)$ and $\beta(d + 1)$ are set up in the following way:
\begin{equation}
\begin{array}{lll}
\alpha(d + 1) & = & \alpha(d)/2 + 1/4\\
\beta(d + 1) & = &  \min \left\{
\frac{\alpha(d)}{(\alpha(d+1))^{\frac{d}{d+1}}}-
(\alpha(d+1))^{\frac{1}{d+1}}, \frac{\beta(d)
  (\alpha(d+1))^{\frac{1}{d+1}}}{
  (\alpha(d))^{\frac{1}{d}}}\right\}
\end{array}
\end{equation}
 \zliu{Let $T = \left(\frac{\alpha(d)}{\left(\alpha(d + 1)\right)^{\frac d{d + 1}}}\right)V^{\frac d{d + 1}}$ and consider the following two cases.}

{\noindent \emph{Case 1.}}   \zliu{$|V(i^*)| < T$. Our $\alpha(d + 1)$ is set up in a way that  \zliu{when $V \leq \alpha(d + 1)b^{d + 1}$, $T < \alpha(d)b^d$.} Next, we invoke the result for $d$ dimensional case on each $V(i)$, $i \in [b]$.
Notice that a lattice on the exterior surface of $V(i)$ in the space $\mathfrak C^{d}$ is also on the exterior surface of $\cG$. Let us call the size of the exterior surface of $V(i)$ as $|\partial V(i)|$. By induction hypothesis, we have $|\partial V(i)|\geq \beta(d)|V(i)|^{\frac{d - 1}{d}}$. Note also $\sum_{i \leq b}|V(i)| = V$.

Next, define $f(x) = x^{\frac{d - 1} d}$, which is a concave function. We have
\begin{eqnarray*}
|\partial \mathcal G|& \geq & \sum_{i \leq b}|\partial V(i)| \\
& \geq & \sum_{i \leq b}\beta(d)f(|V(i)|) \quad \mbox{(induction hypothesis)} \\
& \geq & \sum_{i \leq b}\frac{\beta(d)|V(i)|}{T}f(T) \quad \mbox{($|V(i^*)| < T$ and using the concave properties of $f(\cdot)$)} \\
& = & \frac{\beta(d)V}{T}f(T) \\
& = & \frac{\beta(d)(\alpha(d + 1))^{\frac 1 {d + 1}}}{(\alpha(d))^{\frac 1 d}}V^{\frac{d}{d + 1}} \quad \mbox{(using the definition of $T$)} \\
& \geq & \beta(d + 1)V^{\frac d{d + 1}} \quad\mbox{(by the construction of $\beta(d)$)}
\end{eqnarray*}}

{\noindent \emph{Case 2.}} When $|V(i^*)| \geq T$. By Lemma~\ref{lem:boundary-lowerbound}, $|\partial \cG| \geq |V(i^*)| - |V(i_*)|$. Also by an averaging argument we have $|V(i_*)| \leq V/b$. The theorem then follows.
\end{proof}

\section{Existing techniques}\label{sec:existing}
This section briefly reviews existing lower bound and upper bound analysis techniques and explains the difficulties in  generalizing them to the three dimensional case.

\subsection{Lower bound}\label{subsec:existlower}
Two existing approaches that can potentially be adopted to our lower bound analysis are:
\begin{enumerate}
\item Geometrically understand the growth rate of the smallest ball that covers all the infected agents (hereafter, \emph{the smallest covering ball}). An \emph{upper bound} on the ball's growth rate translates into a lower bound on the completion time for diffusion. Examples of this approach include \cite{AMP02, KS05}.
\item Analyze the interaction of the agents locally to conclude that the influence of infection is constrained to a small region around the initially infected agent, over a small time increment. A union bound or recursive argument is then applied to give a global result. This approach is exemplified by \cite{PPPU11}.
\end{enumerate}

Let us start with the first approach. Alves et~al. and Kesten et~al. \cite{AMP02,KS05} assume the density of the agents is a constant; recall that the density of the agents is the ratio between the total number of agents and the volume of the space.  Their model has infinite space, and hence there is no size parameter $n$. With this assumption, they obtain that the radius of the smallest covering ball grows linearly in time almost surely. Translating to our setting, an $o(1)$ density of agents would lead to a growth rate that is also linear in time $t$ but scales in some way with the density. Directly applying a linear growth rate would still give a valid lower bound of order $\Omega(n)$ on the diffusion time, but this is substantially worse than the bound we need. One potential way to improve their argument is to analyze the scaling of the growth rate with respect to the density. While this approach may well be feasible, it is by no means immediate. For example, the analysis of \cite{AMP02,KS05} appears to depend on the fact that two nearby agents have constant probability to meet within a small number of steps, which leads to the conclusion that uninfected agents near the smallest covering ball are quickly infected. This requires crucially that the density of agents is constant, and relaxing this assumption to $o(1)$ density appears non-trivial.

We have chosen instead to follow the technique developed by Pettarin et~al.\cite{PPPU11}, extending it via our diffusion tree argument.  We now argue that this extension appears necessary. Recall the island graph at time $t$ defined in Definition~\ref{def:island}. Pettarin et~al.'s approach can be summarized by the following three steps:

\begin{enumerate}
\item At any time step, the island graph $G_t(\gamma)$
 is constructed, where $\gamma$ is an appropriately selected parameter.
\item Specify $\delta t$ such that within
$\delta t$ time increment, w.h.p. a piece of virus is unable to travel from one island
to another.
\item Argue that the information has to travel across
$n/\gamma$ islands sequentially to complete the diffusion so that a lower bound
$\frac{n}{\gamma}\cdot \delta t$  is established. The parameter $n/\gamma$ is asymptotically optimal because the space $\mathcal V^3$ cannot pack more than $n/\gamma$ islands along any directions (including those that are not parallel to the axes).
\end{enumerate}

Now let us discuss the internal constraints over the parameters under this framework that prevents us from optimizing
the lower bound for the 3-dimensional case.

At step 1, we need to decide $\gamma$. When $\gamma$ is set to be larger than $n\cdot m^{-1/3}$ i.e. the critical percolation point \cite{PPPU11}, $G_t(\gamma)$ becomes connected w.h.p. and the subsequent arguments break down. Therefore, $\gamma \leq n\cdot m^{-1/3}$.

At step 2, for illustration let us only focus on two islands $\isd_1$ and $\isd_2$, and let $\mathrm a_1 \in \isd_1$ and $\mathrm a_2 \in \isd_2$ be two arbitrary agents each from the two islands.
We now need to decide on the value of $\delta t$. We are facing two options:
\begin{enumerate}
\item If $\delta t$ is set to be smaller than $\gamma^2$, then w.h.p. $\mathrm a_1$ and $\mathrm a_2$ do not meet in time $\delta t$ \cite{PPPU11}.
\item If $\delta t$ is larger than $\gamma^2$, then with probability $\Theta(1/\gamma)$, $\mathrm a_1$ and $\mathrm a_2$ will meet in time $\delta t$ (Lemma~\ref{lem:couple}).
\end{enumerate}

We consider both options to examine the quality of lower bounds we can get, using step 3 above. For the first option, the lower bound we get is $n \gamma \leq n^2\cdot m^{-1/3}$, which is suboptimal. For instance when $m = n^{1.5}$, the lower bound is $n^{1.5}$ as opposed to $\tilde \Omega(n^{1.75})$. For the second option, regardless of the choice of $\delta t$, the lower bound always \emph{fails} to hold with probability $\Omega(1/\gamma) = \Omega(m^{1/3}/n)$ and so step 2 cannot be satisfied with high probability.

Our analysis corresponds to setting $\delta t$ large, but doing a more careful analysis on the local infected region by considering a branching process that represents a historical trace of the infection. Our island diffusion rule is correspondingly modified from the rule of \cite{PPPU11} to control the growth rate of this branching process.

\subsection{Upper bound}\label{subsec:existupper}
We also explain why existing upper bound techniques such as those from \cite{CPS09,PPPU11} do not appear to
generalize immediately to the three dimensional case.  The analyses in \cite{CPS09, PPPU11}, which are based on percolation, follow a proof strategy that contains two steps:
\begin{enumerate}
\item Let $\mathrm a_1$ be the initially infected agent. Identify a ball $\mathcal B$ (under $L_{\infty}$ norm) of radius $r$ that covers $\mathrm a_1$'s initial position so that after $t_1$ time steps, where $t_1$ is a parameter to be decided, a constant portion of the agents in $\mathcal B$
become infected (i.e. fraction of infected agents to total number of agents in $\mathcal B$ is $\tilde{\Theta}(1)$). Moreover, these infected agents are well clustered i.e. at distance $\tilde{O}(r)$ from the ball $\mathcal B$.
\item Show that if a ball $\mathcal B'$ has a constant portion of infected agents at time $t$,
then at $t + t_2$, all adjacent balls with the same
radius will also have a constant portion of infected agents. Here, $t_2$ is a parameter to be decided. Moreover, these newly infected agents are well clustered i.e. at distance $\tilde{O}(r)$ from the balls.
\end{enumerate}

One usually also needs a good density condition i.e. agent density in any $r$-ball is $\Theta(m(r/n)^d)$. By repeatedly applying the second step, one can establish an upper bound on the time that all balls in $\mathcal V^3$ have constant portion of infected agents. Once this happens, usually it becomes straightforward to find the diffusion time. The asymptotic
upper bound will be $\frac{n}{r}\cdot t_2 + t_1$.

Let us explain this in more detail for the case $d=2$. Assume good density condition. First, we need to set $t_2 =\tilde{\Theta}(r^2)$ so that the newly infected agents at step 2
are well clustered. This ensures that the infected agents do not scatter uncontrollably outside a distance from the ball and jeopardize our next recursion. We now sketch a bound on $r$. Consider step 2. Suppose the number of infected agents in $\mathcal B'$ at $t$ is $m(r/n)^2\times\tilde{\Theta}(1)$.
By our choice $t_2 =\tilde{\Theta}(r^2)$, each infected agent in $\mathcal B'$ has probability $\tilde{\Theta}(1)$ to
meet each agent in the adjacent ball (by using Lemma 1 in \cite{PPPU11}). Therefore, the expected number of infections in the adjacent ball is given by
$$\underbrace{m (r/n)^2\times\tilde{\Theta}(1)}_{\mbox{\# of infected agents in $\mathcal B'$}}\times \underbrace{m (r/n)^2}_{\substack{\mbox{\# of uninfected agents} \\\mbox{in an adjacent ball}}}\times\underbrace{\tilde{\Theta}(1)}_{\mbox{infection prob.}}$$ which, by the requirement of step 2, should be equal to $m (r/n)^2\times\tilde{\Theta}(1)$. This gives $r = \Theta(\sqrt{n^2/m})$. Note that this also leads to the condition that the number of infected agents in $\mathcal B'$ at $t$ and the adjacent ball at $t+t_2$ are both $\tilde{\Theta}(1)$.

Now set $t_1=\tilde{\Theta}(r^2)$ and so the number of infected agents in $\mathcal B$ at time $t_1$ is $m(r/n)^2\times\tilde{\Theta}(1)=\tilde{\Theta}(1)$. Note that both steps 1 and 2 are now satisfied. By recursively applying the second step, we can see that by time $\frac{n}{r}\cdot t_2 + t_1=\tilde{\Theta}(n^2/\sqrt{m})$ all the balls in $\mathcal V^2$ will have $m(r/n)^2\times\tilde{\Theta}(1)$ infected agents. Hence in the same order of time period $\tilde{\Theta}(n^2/\sqrt{m})$, all the agents in $\mathcal V^2$ will be infected. This time period gives the optimal upper bound of the diffusion time for $d=2$.

We now argue that this strategy does not work for $d=3$. Let us attempt to mimic the above argument step by step. Again set $t_2=\tilde{\Theta}(r^2)$ so that the infected agents are well clustered. Next, note that in contrast to the two-dimensional case, Lemma~\ref{lem:couple} states that the meeting probability of two random walks in $\mathcal V^3$ with initial distance $r$ apart within time $\Theta(r^2)$ is $\Theta(1/r)$. Hence, in light of step 2, we require
$$\underbrace{m(r/n)^3\times\tilde{\Theta}(1)}_{\mbox{\# of infected agents in $\mathcal B'$ }}\times \underbrace{m(r/n)^3}_{\substack{\mbox{\# of uninfected agents }\\ \mbox{in an adjacent ball}}}\times\underbrace{\Theta(1/r)}_{\mbox{infection prob.}}=\underbrace{m(r/n)^3\times\tilde{\Theta}(1)}_{\substack{\mbox{desired \# of}\\ \mbox{infections}}}$$ which gives $r=\tilde{\Theta}(\sqrt{n^3/m})$. Note that the number of infected agents in $\mathcal B'$ at $t$ and that of the adjacent balls at $t+t_2$ in step 2 are now both $m(r/n)^3\times\tilde{\Theta}(1)=\tilde{\Theta}(\sqrt{n^3/m})=\tilde{\Theta}(r)$.


We now try to set an appropriate value for $t_1$. First, note that step 1 requires the number of infected agents in $\mathcal B$ at time $t_1$ being $\tilde{\Theta}(r)$. Then the question is to find the approximate time for one initially infected agent to infect $\tilde{\Theta}(r)$ agents that are from $\mathcal B$. Moreover, we need that these infected agents do not travel at distance outside $\tilde{\Omega}(r)$ in the same time period.

To give a bound for this $t_1$, let us look into the method of \cite{PPPU11}. Note that in the case of $d=2$, the number of agents in $\mathcal B$ at any time is $\tilde{\Theta}(1)$. In this case, \cite{PPPU11} suggests chopping the time $t_1$ into intervals each of length $\tilde{\Theta}(r^2)$. During each of these intervals, one only focuses on a pair of agents from $\mathcal B$ and see if they meet each other; this method aims to reduce the analysis of correlation among multiple agents' meetings, a complicated quantity, to a sequence of independent problems that involve only the meeting of two random walks. Since there are only $\tilde{\Theta}(1)$ such pair combinations, and that each such meeting probability is $\tilde{\Theta}(1)$, a $t_1=\tilde{\Theta}(r^2)$ is enough to guarantee that the number of infected agents is $\tilde{\Theta}(1)$. Also these infected agents are well clustered at $\mathcal B$. Thus the argument works well for $d=2$.

However, such an argument breaks down for $d=3$ because now we are required to have $\tilde{\Theta}(r)$ infected agents at $t_1$, and the meeting probability between any two agents is $\tilde{\Theta}(1/r)$. As a result the following tradeoffs cannot be balanced: 1) $t_1$ is set to be $\tilde{\Theta}(r^2)$ so that the infected agents are well clustered, but the number of infected agents at $t_1$ will only be $\tilde{\Theta}(1)$; 2) $t_1$ is set to be $\tilde{\omega}(r^2)$, but then the infected agents are not well clustered and may not constitute $\tilde{\Theta}(r)$ of infected agents within $\mathcal B$ at $t_1$. The first tradeoff appears if one uses the chopping argument of \cite{PPPU11}: divide $t_1$ into intervals of length $\tilde{\Theta}(r^2)$. For each interval, observe the number of meetings between any infected and uninfected agents. This gives an expected total number of infections at $t_1$ as $r\cdot \tilde \Theta(1/r) = \tilde \Theta(1)$, which is less than the required number of $\tilde \Theta(r)$. Secondly, setting $t_1=\tilde \omega(r^2)$ boosts up the number of infected agents, but also increases the chance that an infected agent escapes from the vicinity of $\mathcal B$. An accurate analysis of these two effects seems highly non-trivial and does not follow from the existing results of \cite{PPPU11}.

Finally, we mention the work of Clementi et~al. \cite{CPS09} to deal with issues similar to above. At each step, conditioned on the positions of the infected agents, the infection event of each uninfected agent becomes independent of each other. The change in the infected population over time can then be analyzed. However, such analysis is possible in \cite{CPS09} because the agents in their model can jump at a distance $\Theta({\sqrt{n}})$ at each step. This leads to much less serial dependence for each agent and consequently requires less effort in keeping track of each agent's position. These phenomena, unfortunately, do not apply to our settings.

\section{An example of the diffusion tree}
\begin{figure}[htb!]
\centering
\includegraphics[scale=0.5, angle=270]{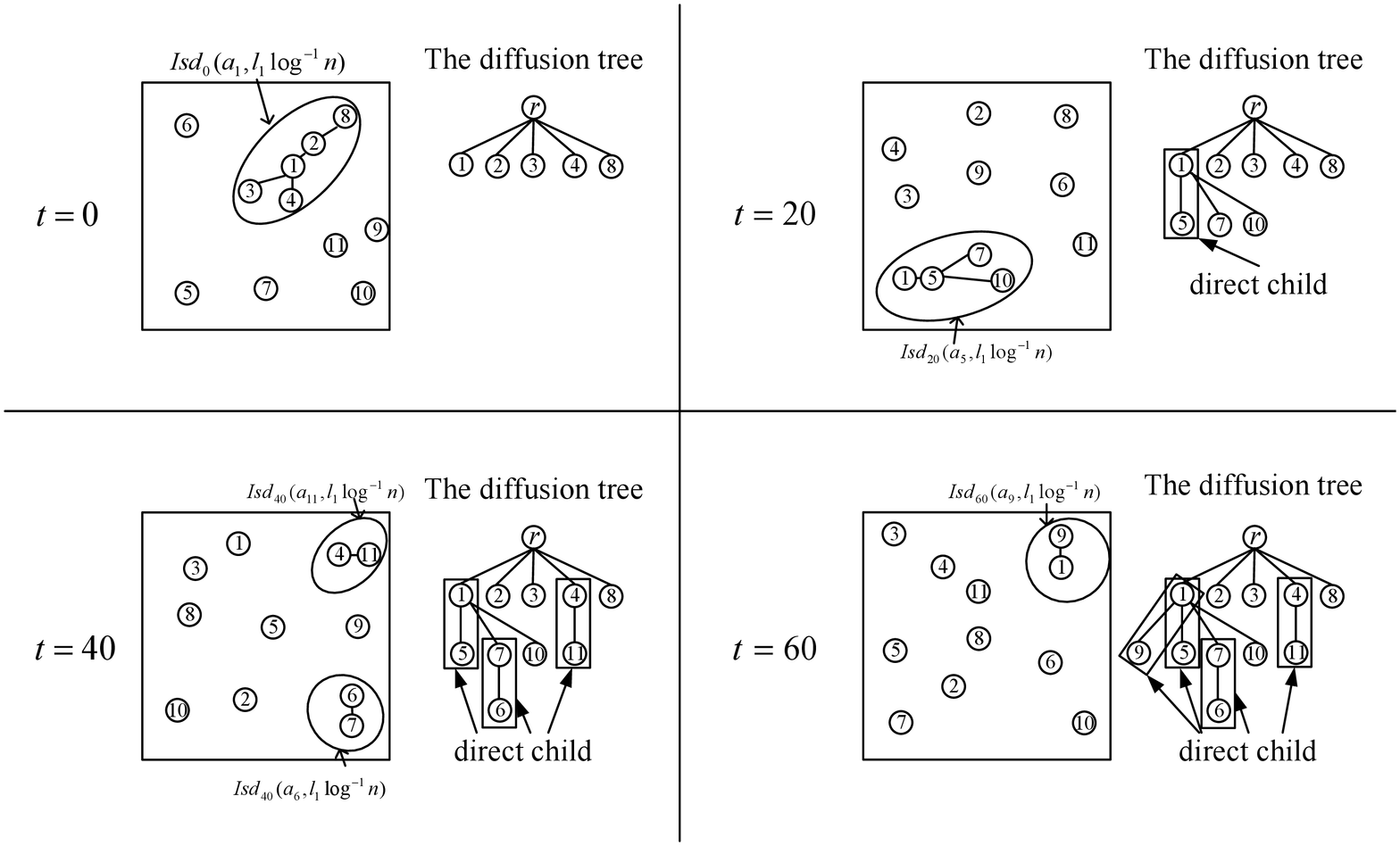}
\caption{
\label{fig:branching}
An example of the diffusion process and its corresponding diffusion tree at $t = 0, 20, 40, 60$. Assume no collisions happen
beyond these 4 time steps.
 \\ \\
$\quad$ \underline{At $t = 0$}, the agent $\mathrm a_1$ is initially infected. Since $\isd_0(\mathrm a_1,\ell_1\log^{-1}n) = \{\mathrm a_1, \mathrm a_2, \mathrm a_3, \mathrm a_4, \mathrm a_8\}$, the agents $\mathrm a_2, \mathrm a_3$, $\mathrm a_4$, and $\mathrm a_8$ are all considered infected at $t = 0$. Also, $\rt$ does not have a direct child.\\ \\
$\quad$ \underline{At $t = 20$}, the agent $\mathrm a_1$ meets $\mathrm a_5$. Since $\isd_{20}(\mathrm a_1, \ell_1\log^{-1}n) = \{\mathrm a_1, \mathrm a_5, \mathrm a_7, \mathrm a_{10}\}$, $\mathrm a_{7}$ and $\mathrm a_{10}$ are also infected. At this time step, $\dchild(\mathrm a_1) = \{\mathrm a_5\}$ and $\child(\mathrm a_1) = \{\mathrm a_5, \mathrm a_{7}, \mathrm a_{10}\}$.\\ \\
$\quad$ \underline{At $t = 40$}, $\mathrm a_{4}$ meets $\mathrm a_{11}$ and $\mathrm a_6$ meets $\mathrm a_7$; $\isd_{40}(\mathrm a_4, \ell_1\log^{-1}n) = \{\mathrm a_4, \mathrm a_{11}\}$ and
 $\isd_{40}(\mathrm a_7, \ell_1\log^{-1}n) = \{\mathrm a_6, \mathrm a_7\}$.
  At this time step, $\child_4(\mathrm a_4) = \dchild(\mathrm a_4) = \{\mathrm a_{11}\}$ and $\child(\mathrm a_7) = \dchild(\mathrm a_7) = \{\mathrm a_6\}$. Notice that $\mathrm a_{11} \in \mathbb F_{1}$ and $\mathrm a_{6} \in \mathbb F_3$. These two generations grow simultaneously at $t = 40$.\\ \\
$\quad$ \underline{At $t = 60$}, $\mathrm a_1$ meets $\mathrm a_9$. $\isd_{60}(\mathrm a_9) = \{\mathrm a_1, \mathrm a_9\}$. We have $\mathrm a_9 \in \child(\mathrm a_1)$ and $\mathrm a_9 \in \dchild(\mathrm a_1)$. Also, notice that $\mathrm a_1$ now contains two \zliu{direct} children.
 }
\end{figure}

\end{document}